\documentclass[1 [leqno,11pt]{amsart}
\usepackage{amssymb, amsmath}
\allowdisplaybreaks[4]
\usepackage{color}

\def\ccite#1{~\cite{#1}}

\def\inte#1{
\displaystyle\mathop{#1\kern0pt}^\circ }

\let\pa=\partial
\let\al=\alpha

\let\d=\delta
\let\e=\varepsilon

\let\r=\rho

\let\f=\frac
\let\vf=\varphi
\let\p=\psi

\let\D=\Delta

\def\cD{{\mathcal D}}

\def\cF{{\mathcal F}}

\def\pa{\partial}

\def\dB{\dot{B}}

\def\dD{\dot{\Delta}}
\def\virgp{\raise 2pt\hbox{,}}
\def\cdotpv{\raise 2pt\hbox{;}}

\def\eqdefa{\buildrel\hbox{\footnotesize def}\over =}

\def\Re{\mathop{\rm Re}\nolimits}

\def\im {\mathop{\rm Im}\nolimits}

\def\C{\mathop{\mathbb C\kern 0pt}\nolimits}
\def\DD{\mathop{\mathbb D\kern 0pt}\nolimits}
\def\EE{\mathop{{\mathbb E \kern 0pt}}\nolimits}
\def\K{\mathop{\mathbb K\kern 0pt}\nolimits}
\def\N{\mathop{\mathbb N\kern 0pt}\nolimits}
\def\Q{\mathop{\mathbb Q\kern 0pt}\nolimits}
\def\R{\mathop{\mathbb R\kern 0pt}\nolimits}
\def\SS{\mathop{\mathbb S\kern 0pt}\nolimits}
\def\ZZ{\mathop{\mathbb Z\kern 0pt}\nolimits}
\def\TT{\mathop{\mathbb T\kern 0pt}\nolimits}
\def\P{\mathop{\mathbb P\kern 0pt}\nolimits}

\newcommand{\Z}{{\ZZ}}

\def\curl{\mathop{\rm curl}\nolimits}

\def\no{\noindent}
\def\na{\nabla}
\def\p{\partial}
\def\th{\theta}

\newcommand{\w}[1]{\langle {#1} \rangle}
\newcommand{\beq}{\begin{equation}}
\newcommand{\eeq}{\end{equation}}
\newcommand{\ben}{\begin{eqnarray}}
\newcommand{\een}{\end{eqnarray}}
\newcommand{\beno}{\begin{eqnarray*}}
\newcommand{\eeno}{\end{eqnarray*}}
\newcommand{\andf}{\quad\hbox{and}\quad}
\newcommand{\with}{\quad\hbox{with}\quad}
\newtheorem{defi}{Definition}[section]

\newtheorem{lem}{Lemma}[section]
\newtheorem{rmk}{Remark}[section]

\newtheorem{prop}{Proposition}[section]

\newcommand{\vv}[1]{\boldsymbol{#1}}

\def\div{\text{div}\,}
\def\curl{\text{curl}\,}

\def\ga{\gamma}

\newtheorem*{Main Theorem}{Main Theorem}
\newtheorem{theorem}{Theorem}[section]

\numberwithin{equation}{section}

\begin{document}
\title[Strong solutions of CNS with short pulse type data]{Global strong solutions of 3D Compressible Navier-Stokes equations with short pulse type initial data}

\author[L.B. HE]{Ling-Bing He}
\address[L.B. He]{Department of Mathematical Sciences, Tsinghua University\\ Beijing, China}
\email{hlb@tsinghua.edu.cn}

\author[L. XU]{Li Xu}
\address[L. Xu]{School of Mathematical Sciences, Beihang University\\  100191 Beijing, China}
\email{xuliice@buaa.edu.cn}

\author[P. Zhang]{Ping Zhang}
\address[P. Zhang]{Academy of Mathematics $\&$ Systems Science
and  Hua Loo-Keng Key Laboratory of Mathematics, Chinese Academy of
Sciences, Beijing 100190, China, and School of Mathematical Sciences,
University of Chinese Academy of Sciences, Beijing 100049, China.}
\email{zp@amss.ac.cn}

\date{}

\begin{abstract} Short pulse initial datum is referred to the one   supported in the ball of radius $\delta$ and with amplitude $\delta^{\f12}$ which looks like a pulse. It was first introduced by Christodoulou to prove the formation of black holes for Einstein equations and also to  catch the shock formation for compressible Euler equations. The aim of this article is to consider the same type initial data, which   allow the   density of the fluid to have  large amplitude $\delta^{-\f{\alpha}{\gamma}}$ with $\delta\in(0,1],$  for the compressible Navier-Stokes equations. We prove the global well-posedness  and  show that the initial  bump region of the density with large amplitude will disappear within a very short time. As a consequence, we obtain the global dynamic behavior of the solutions and   the boundedness of $\|\na\vv u\|_{L^1([0,\infty);L^\infty)}$. The key ingredients of the proof lie in the new observations for the effective viscous flux and new decay estimates for the density via the Lagrangian coordinate.
\end{abstract}

\maketitle

\setcounter{tocdepth}{1}

\setcounter{equation}{0}
\section{Introduction}
The main purpose of this paper is to investigate the global existence and dynamic behavior of strong solutions to
  the following 3D barotropic compressible Navier-Stokes equations with short pulse type initial data:
\beq\label{CNS}\left\{\begin{aligned}
&\p_t\r+\div(\r\vv u)=0,\quad\text{in}\ (t,x)\in\R^+\times\R^3,\\
&\p_t(\r\vv u)+\div(\r\vv u\otimes\vv u)-\D\vv u+\na\r^\gamma=\vv 0,\\
&\r|_{t=0}=\r_0,\quad\vv u|_{t=0}=\vv u_0,\, \lim_{|x|\rightarrow\infty}\r(t,x)=1,
\end{aligned}\right.\eeq
where $\gamma\geq 1$. The unknown functions $\r=\r(t,x)\in\R^+$ and $ \vv u=\vv u(t,x)\in\R^3$ denote the density and velocity of the fluid respectively.
For simplicity, here  we only consider the case with only shear viscosity to be $1$. We remark that all the results in the present paper  hold also for the general 3D barotropic compressible Navier-Stokes equations with both shear and bulk viscosities.

\smallskip

\subsection{Short review of the previous work} Before introducing the short pulse type initial data, we review some classical results on the compressible Navier-Stokes equations.

$\bullet$  When the initial density is away from vacuum, the local well-posedness of the system \eqref{CNS} was proved in \cite{Nash, Solo76,Tani}.
 Matsumura and  Nishida \cite{MN1,MN2} proved the global existence and uniqueness of smooth solutions to the full compressible Navier-Stokes
 equations with initial data near constant state in Sobolev spaces. One may check \cite{Charve, CMZ, Dan-Inve} for the  results of this type concerning the well-posedness
 of \eqref{CNS} with initial data in the so-called critical  Besov spaces. With initial data of finite energy, Lions \cite{Lions98} and Feireisl
\cite{Feisl} proved the global existence of weak solutions to \eqref{CNS} for $\gamma>\f{d}2$ where $d$ designates the space dimension.

 $\bullet$ In general, the global well-posedness of the multi-dimensional
 compressible Navier-Stokes system with large data is widely open. Except in two space dimension, under the assumption that the bulk
 viscosity coefficient $\lambda=\r^\beta$ for $\beta>3,$ Va$\breve{i}$gant and   Kazhikhov \cite{VK95} proved the global existence and uniqueness of solutions to the isentropic compressible Navier-Stokes system. Huang and Li \cite{HL16} improved the previous result for $\beta>\f43$ even allowing the
 appearance of vacuum. Lately, Lu and the third author \cite{LYZ1} proved the global existence of smooth solutions to 2D barotropic compressible Navier-Stokes
 system with slow variable initial data.   Huang, Li and  Xin \cite{Huang} proved the global existence of classical solutions with small initial energy and allowing vacuum to the three dimensional isentropic compressible Navier-Stokes equations, and see \cite{Huang1} for the corresponding result of full compressible Navier-Stokes system.

\subsection{Short pulse type initial data} Short pulse initial datum is the one chosen to be supported in the ball of radius
 $\delta$ and with amplitude $\delta^{\f12}$ which looks like a pulse. They are often used to treat the focusing of incoming waves for quasilinear wave equations. It was first introduced by Christodoulou in the breakthrough work \cite{Christ} to show  dynamically the formation of
   black holes for Einstein equations.  When  the initial data  are additionally assumed to be irrotational and isentropic, in \cite{ChristMiao}, Christodoulou and Miao proved the   formation of shock for the classical, non-relativistic, compressible Euler equations. We also refer to the work by Miao and Yu  \cite{MiaoYu} on the geometric perspectives of shock formations for the 3-dimensional quasilinear wave equation with short pulse initial data. We finally mention the short pulse initial  data can also be used to construct the global and large smooth solutions for wave equations  with null conditions(see \cite{DXY1,MLP19}) and for 2D irrotational and isentropic Chaplypin gases system(see Ding-Xin-Yin's work \cite{DXY2}).

\smallskip

In the present work,  we shall impose the short pulse type data on the compressible Navier-Stokes equations. Comparing to the standard short pulse data, we allow the density to have the large amplitude. More precisely, let $\delta\in(0,1]$, then the typical data can be taken as follows:
\beq\label{initial data}\begin{aligned}
&\r_0(x)=\bigl(1+\delta^{-\al}\varphi({x}/{\delta})\bigr)^{\f{1}{\gamma}},\\
\vv u_0(x)=&\delta^{1-\al}
(\na\D^{-1}\varphi)({x}/{\delta})+\delta^{1-\f{\al}{2}}\vv v({x}/{\delta}),
\end{aligned}\eeq
where $\alpha$ is a positive constant, $\varphi\in\mathcal{S}(\R^3)$ is a non-negative function and $\vv v=(v^1,v^2,v^3)$ with  each component of which belongs to $\mathcal{S}(\R^3)$. In this situation, we are interested in the following questions:
 \smallskip

 {\bf (Q1):} In very recent work \cite{MRRS}, the authors construct a set of finite energy smooth initial data for which the corresponding solutions to the equations implode (with infinite density) at a later time at a point. But they consider the equations in the case that $\lim\limits_{|x|\rightarrow \infty}\r(t,x)=0$.  In our case, we require that $\lim\limits_{|x|\rightarrow \infty}\r(t,x)=1$. And the initial bump region of the density will become more and more singular in the process of $\delta$ approaching to zero. 
 It is natural to ask what is the exact behavior of evolution of the density. Will it induce the singularity? If not, can we describe the precise behavior of the density  and what is the impact on the velocity $\vv u$?

  {\bf (Q2):}  As we mentioned  before, the short  pulse data will induce the blow-up of $\na_x \vv u$ for the compressible Euler equations (see \cite{ChristMiao}). It is interesting to investigate the role of the viscosity in preventing the formation of the singularity.
\smallskip

Before going further, let us give   some comments on short pulse type  data:

\begin{enumerate}
\item[(i).]    When $\gamma=1$ and $\al=\f23-$,  \eqref{initial data} implies that $\|\na \rho_0\|_{L^2}\sim \delta^{-\f16+}$ and $\|\rho_0\|_{L^\infty}\sim \delta^{-\f23+}$. This shows that initially $\rho_0-1$ is sufficiently large even in the critical space.  Obviously it will bring troubles for the proof of global  existence.
\item[(ii).] The main feature of the short pulse type initial data lies in the fact that one more derivative on the  function will produce additional factor $\delta^{-1}$. It means that the Sobolev regularity and  $L^p$ energy of the data will be $\delta$-dependent and behave quite different from each other. This property indicates that we have to be cautious when we use interpolation method  to the estimates of the equations.
\end{enumerate}

 \subsection{Main results and  ideas of the proof}
We first introduce some notations as follows:

 $\bullet$ Let
\beq\label{def of a, q and dot u}
a\eqdefa\r^\gamma-1,\quad F\eqdefa\div\vv u-a,
\andf \dot{\vv u}\eqdefa\p_t\vv u+\vv u\cdot\na\vv u,
\eeq
where $F$ is the so-called ``effective viscous flux" which was introduced by Hoff in \cite{Hoff95}. In particular,
in view of the momentum equation of \eqref{CNS}, we denote
\beq\label{intital a, q dot u}\begin{aligned}
&\sqrt{\r_0}\dot{\vv u}_0\eqdefa\f{-\curl\curl\vv u_0+\na F_0}{\sqrt{\r_0}},\,\,
\text{with}\ F_0\eqdefa\div\vv u_0-a_0,\ a_0\eqdefa \r_0^\gamma-1.
\end{aligned}\eeq

\smallskip

 $\bullet$ We  recall the following potential energy functional
\beq\label{def of H rho}
H(\r)\eqdefa\int_{\R^3}h(\r(t,x))dx\quad\text{with}
\quad h(\r)\eqdefa \left\{\begin{aligned}
&\f{1}{\gamma-1}[(\r^\gamma-1)-\gamma(\r-1)],\,\,\text{if}\,\, \gamma>1,\\
&\r\ln\r-(\r-1),\,\,\text{if}\,\, \gamma=1.
\end{aligned}\right.
\eeq

To answer the questions {\bf (Q1,Q2)}, we divide our main results into three parts: global well-posedness,   global dynamic behavior of the lower order energy and the  uniform-in-time propagation of the regularity.

 \subsubsection{Main result (I): Global well-posedness}
 Our first main result is concerned with the global existence of strong solutions to \eqref{CNS} with data
 which in particular include "short pulse type initial data" introduced in \eqref{initial data}. To state it, we introduce the following  Sobolev spaces: for $q>2$, $T>0$,
 $$
W^{1,q}(\R^3)\eqdefa \left\{\ f\in L^q(\R^3)\,|\, \, \na f\in L^q(\R^3)\ \right\},
 $$
$$
W^{2,1}_q([0,T]\times\R^3)\eqdefa \left\{\ f\in L^q([0,T]\times\R^3)\,|\,\, \p_tf, \na^2 f\in L^q([0,T]\times\R^3)\ \right\}.
 $$
The homogeneous spaces $\dot{W}^{1,q}(\R^3)$ and $\dot{W}^{2,1}_q([0,T]\times\R^3)$ can be defined in a similar way.

\begin{theorem}\label{global existence thm}
{\sl Let $\delta\in(0,1]$,  $\gamma\geq 1$ and  $\al\in \bigl(0,\f{2\gamma}{1+{2\gamma}}\bigr]$. We assume that $\na\r_0\in L^q(\R^3)$ for some
$q\in (3,10/3]$  and $\vv u_0\in H^2(\R^3)$ which satisfy
\beq\label{initial ansatz for density}
\inf_{x\in\R^3}\r_0(x)\geq 1, \quad\|\r_0^\gamma-1\|_{L^\infty}=\|a_0\|_{L^\infty}= \delta^{-\al},
\eeq
and
\beq\label{initial ansatz}\begin{aligned}
\mathcal{E}_0\eqdefa&\delta^{\al-3}\bigl(\|(\sqrt{\r_0}\vv u_0, \curl\vv u_0, F_0)\|_{L^2}^2+H(\r_0)\bigr)\\
&+\delta^{2\al-1}\bigl(
\|(\sqrt{\r_0}\dot{\vv u}_0,\div\vv u_0)\|_{L^2}^2+\|a_0\|_{L^6}^2\bigr)= \e^2.
\end{aligned}\eeq
Then there exists a constant $\e_0>0$ which is independent of $\ \delta$ such that for any $\e\in(0,\e_0)$,
 \eqref{CNS} admits a unique global solution $(\r,\vv u)$ so that $\na\r\in L^\infty([0,T];L^q(\R^3)), \vv u\in W^{2,1}_q([0,T]\times\R^3)$ for any $T>0$ and
\beq\label{estimate for a}
\inf_{t>0,x\in\R^3}\r(t,x)\ge 2^{-\f1\gamma},\quad\|a\|_{L^\infty(\R^+;L^\infty)}\leq C\delta^{-\al},
\eeq
\beq\label{total enery estimate}\begin{aligned}
&\delta^{\al-3}\Bigl\{\sup_{t>0}\bigl(\|(\sqrt{\r}\vv u, \curl\vv u, F)(t)\|_{L^2}^2+ H(\r)(t)\bigr)
+\|(\na\vv u,\sqrt\r\dot{\vv u})\|_{L^2(\R_+;L^2)}^2\Bigr\}\\
&+\delta^{2\al-1}\Bigl\{\sup_{t>0}
\bigl(\|(\sqrt\r\dot{\vv u},\div \vv u)(t)\|_{L^2}^2
+\|a(t)\|_{L^6}^2\bigr)+\|\na\dot{\vv u}\|_{L^2(\R_+;L^2)}^2+\|a\|_{L^2(\R_+;L^6)}^2\Bigr\}\leq C\mathcal{E}_0.
\end{aligned}\eeq
Here and in what follows, $C>0$ is an universal constant which is independent of $\delta$. Furthermore, there holds
\beq\label{energy in H 2}\begin{aligned}
&\delta^{2\al-1+\f{\al}{\gamma}}\sup_{t>0}\|(\na\curl\vv u,\na F)(t)\|_{L^2}^2+\delta^{\al-3+\f{\al}{\gamma}}\|(\na\curl\vv u,\na F)\|_{L^2(\R_+;L^2)}^2
\leq C\mathcal{E}_0.
\end{aligned}\eeq
}
\end{theorem}

 \begin{rmk}
 \begin{enumerate}
\item[(1).]  We normalize the constants in \eqref{initial ansatz for density} for simplicity. Indeed, if  \eqref{initial ansatz for density} is  replaced by
\[\inf_{x\in\R^3}\r_0(x)\geq \underline{c} , \quad\|\r_0^\gamma-1\|_{L^\infty}=\|a_0\|_{L^\infty}= \bar{c}\delta^{-\al},\]  where $\underline{c}$ and $\bar{c}$ are two universal constants, our theorems here and hereafter  still hold  with the related constants depending on $\underline{c}$ and $\bar{c}$.

\item[(2).] The ansatzes \eqref{initial ansatz for density} and \eqref{initial ansatz} on the parameter $\delta$  stem from the typical example \eqref{initial data}. For $(\r_0,\vv u_0)$ defined by \eqref{initial data}, we have
\beno\begin{aligned}
&a_0=\delta^{-\al}\varphi({x}/{\delta}),\quad\div\vv u_0=\delta^{-\al}\varphi({x}/{\delta})+\delta^{-\f{\al}{2}}\div\vv v({x}/{\delta}),\\
&\curl\vv u_0=\delta^{-\f\al2}\curl\vv v({x}/{\delta}),\quad F_0=\delta^{-\f{\al}{2}}\div\vv v({x}/{\delta})\\
 \andf &\sqrt{\r_0}\dot{\vv u}_0=\delta^{-1-\f\al2}\r_0^{-\f12}(x)\Delta\vv v({x}/{\delta}).
\end{aligned}\eeno
which together with \eqref{def of H rho} implies that for $\gamma>1$,
\beno\begin{aligned}
&H(\r_0)\leq(\gamma-1)^{-1}\|a_0\|_{L^1}\sim O( \delta^{3-\al}),\quad\|\curl\vv u_0\|_{L^2}^2\sim
\|F_0\|_{L^2}^2\sim O( \delta^{3-\al}),\\
&\|a_0\|_{L^2}^2\sim O( \delta^{3-2\al}),\quad \|a_0\|_{L^6}^2\sim O( \delta^{1-2\al}).
\end{aligned}\eeno

\item[(3).] The ansatzs \eqref{initial ansatz for density} and \eqref{initial ansatz} show that  the initial density may be sufficiently large in   $L^\infty$ norm while the initial velocity may be small in the critical space. In general, with initial
density being a small perturbation of some positive constant in  $L^\infty$ norm and with initial velocity in the slightly subcritical
spaces, Danchin, Fanelli and Paicu \cite{DFM20} proved local existence of weak solutions of \eqref{CNS} but without uniqueness.

\item[(4).] If in addition $\vv u_0,\rho_0-1\in H^3(\R^3)$ and $\delta=1$, by interpolation theory,  \eqref{initial ansatz} can be reduced to the smallness assumption only on  the initial  energy, $\|\sqrt{\r_0}\vv u_0\|^2_{L^2}+H(\r_0)$. Then Theorem 1.1 of \cite{Huang} ensures the
    global well-posedness part of Theorem \ref{global existence thm} here. Yet we remark that whether or not the density function will be away from vacuum when the initial density has a positive lower bound is one of the most interesting open questions for multi-dimensional compressible Navier-Stokes system.
    Here we present a uniform lower bound for the density function in \eqref{estimate for a}.
    \end{enumerate}
\end{rmk}

\begin{rmk} We first remark that (\ref{estimate for a}) and (\ref{total enery estimate}) show that the profile of the initial data w.r.t. $\delta$ parameter can be propagated for all the time. This in particular shows that the formation of the singularity for the density does not occur. However we have no idea that the high order  norms enjoy the same property. What we can show is that they  are finite for any positive time $t$ (for instance, $\|\na\rho(t)\|_{L^q(\R^3)}$ and $\|\na^2\vv u(t)\|_{L^2(\R^3)}$). However this is crucial  to get the uniqueness part of Theorem \ref{global existence thm}.
\end{rmk}

\begin{rmk}
Similar result as in Theorem \ref{global existence thm} holds for the following general 3D barotropic compressible Navier-Stokes equations{}
\beq\label{CNS mu}\left\{\begin{aligned}
&\p_t\r+\div(\r\vv u)=0,\\{}
&\p_t(\r\vv u)+\div(\r\vv u\otimes\vv u)-\mu\D\vv u-\lambda\na\div\vv u+\na\r^\gamma=\vv 0,\\
&\r|_{t=0}=\r_0,\quad\vv u|_{t=0}=\vv u_0,
\end{aligned}\right.\eeq
where $\mu>0$ and $\lambda\geq0$ are shear and bulk viscosities respectively. As a matter of fact,
let $\nu\eqdefa \mu+\lambda$ and $F_0\eqdefa\div\vv u_0-\f{a_0}{\nu}$.
If we replace the ansatz \eqref{initial ansatz} by
\beno\begin{aligned}
\mathcal{E}_0^{\mu,\nu}\eqdefa&\delta^{\al-3}\mu^{-\f72}\nu^{\f32}\bigl(\gamma\nu^{-1}\bigl(\|\sqrt{\r_0}\vv u_0\|_{L^2}^2+H(\r_0)\bigr)+\mu\|\curl\vv u_0\|_{L^2}^2+\nu\|F_0\|_{L^2}^2\bigr)\\
 &+\delta^{2\al-1}\mu^{-\f32}\nu^{\f12}\bigl(
\|\sqrt{\r_0}\dot{\vv u}_0\|_{L^2}^2+\gamma\|\div\vv u_0\|_{L^2}^2+\mu^{-2}\nu^2\|a_0\|_{L^6}^2\bigr)=\e^2,
\end{aligned}\eeno
then under the assumptions of  Theorem \ref{global existence thm}, there exists sufficiently small $\e_0>0$ which  is independent of  $\mu,\nu$ and $\delta$ such that, for any $\e\in(0,\e_0)$, \eqref{CNS mu} has a unique global solution $(\r,\vv u)$. And similar energy estimate as \eqref{total enery estimate} still holds.
\end{rmk}

\subsubsection{Main result(II): Global dynamic behavior of the lower order  energy}

The second result is about the global dynamic behavior of the solutions. In particular, we give the precise estimates on the behavior of the density in $L^\infty$ space.
\begin{theorem}\label{global dynamics thm}
{\sl Let $\delta\in\bigl(0, 2^{-\f{15}{\al}}\bigr)$ and $N_0\in\N_{\geq 15}$ so that $\delta^{-\al}\in\bigl[2^{N_0},2^{N_0+1}\bigr).$
In addition to the assumptions of Theorem \ref{global existence thm}, we assume moreover that
 $\r_0-1,\,\r_0\vv u_0\in L^1(\R^3)$ which  satisfies
\beq\label{initial ansatz in L 1}
\|\r_0-1\|_{L^1}+\|\r_0\vv u_0\|_{L^1}\leq C\delta^{3-\al}.
\eeq
Let  $T_0\eqdefa \f{7}{8\gamma}\bigl(2^{-14}-2^{-N_0}\bigr)\in\bigl( 7\gamma^{-1}\cdot2^{-18},\ 7\gamma^{-1}\cdot2^{-17}\bigr)$.

\begin{enumerate}

\item[(1)]{\bf(Behavior of $\|a(t)\|_{L^\infty}$)}  There exist  $t_j\eqdefa{7}{\gamma}^{-1}\sum_{j'=1}^j2^{-N_0-4+j'}$ with $j=1,2,\cdots, N_0-14$ and
 $t_0=0$ so that
\beq\label{dynamic for a in L infty}
\|a(t)\|_{L^\infty}\leq
\left\{\begin{aligned}
&\Bigl(\f14(1+2^{N_0+1-j})^{-1}+\gamma(t-t_{j-1})\Bigr)^{-1}\leq C\delta^{-\al},\\
&\qquad\qquad\text{for }\ t\in[t_{j-1},t_j], \quad j=1,2,\cdots, N_0-14, \\
&C(e^{-\f{\gamma}{2}t}+\e\delta^{1-\f34\al}),\quad\text{for }\ t\geq T_0= t_{N_0-14}=\f{7}{8\gamma}\bigl(2^{-14}-2^{-N_0}\bigr).
\end{aligned}\right.
\eeq
If  $T_2\eqdefa(\e\delta^{1-\f34\al})^{-2}$, there holds moreover that
\beq\label{dynamic for a in L 1 L infty}
\int_0^{T_0}\|a(t)\|_{L^\infty}\,dt\leq 3\gamma^{-1}\al\ln\delta^{-1}\quad\text{and}\quad
\int_{T_0}^{4T_2}\|a(t)\|_{L^\infty}\,dt\leq C.
\eeq

\item[(2)]{\bf(Behavior of $\|a(t)\|_{L^6}$)} It holds that
\beq\label{dynamic for a in L 6}
\|a(t)\|_{L^6}^2\leq
\left\{\begin{aligned}
&C\e^2\delta^{1-2\al}, \quad\text{if }\ t\leq T_0,\\
&C\e^2\bigl(\delta^{1-2\al}e^{-\f{t}{6}}+\delta^{3-\al}\bigr), \quad\text{if}\ t\geq T_0.
\end{aligned}\right.
\eeq

\item[(3)]{\bf(Decay of the lower order energy)}
There exists a constant $c_0\in(0,\f12)$ and  $T_*\eqdefa\f{1}{c_0}>T_0$ so that
\beq\label{decay of energy 1}\begin{aligned}
\|(\vv u, \na\vv u, F,\dot{\vv u})(t)\|_{L^2}^2+\|a(t)\|_{L^6}^2\leq C\delta^{1-2\al}\w{t}^{-2}+C\delta^{\f72-\f{13}{4}\al}\w{t}^{-\f32},\quad\forall\ t\geq T_*,
\end{aligned}\eeq
and
\beq\label{decay of energy 2}\begin{aligned}
\sup_{t\geq T}\bigl(\|(\vv u, \na\vv u, F, \dot{\vv u})(t)\|_{L^2}^2&+\|a(t)\|_{L^6}^2\bigr)
+\int_{T}^\infty\bigl(\|(\na\vv u, \dot{\vv u}, \na\dot{\vv u})(t)\|_{L^2}^2+\|a(t)\|_{L^6}^2\bigr)\,dt\\
&\qquad\leq C\delta^{1-2\al}\w{T}^{-2}+C\delta^{\f72-\f{13}{4}\al}\w{T}^{-\f32},
\quad\forall\ T\geq T_*.
\end{aligned}\eeq
Here and in all that follows, we always denote $\w{t}\eqdefa (1+t^2)^{\f12}.$
\end{enumerate}
}
\end{theorem}

\begin{rmk}
\begin{enumerate}
\item[(1)] Our theorem holds also for $\delta\in[2^{-\f{15}{\alpha}},1]$ with $T_0=0$ and $T_2=\e^{-2}$. In this case, the parameter $\delta$ can be regarded as an universal constant and thus our proof used here is still valid.

\item[(2)] We remark that the condition \eqref{initial ansatz in L 1} will only be used in the proof of
the decay estimate \eqref{decay of energy 1}. In fact, we only need the boundedness of the quantity
$\|\r_0-1\|_{L^1}+\|\r_0\vv u_0\|_{L^1}.$ The reason why we assume \eqref{initial ansatz in L 1} is to make it
to be consistent with the assumption \eqref{initial ansatz}.
\end{enumerate}
\end{rmk}

\begin{rmk}
\begin{enumerate}
  \item[(1)] Roughly speaking, \eqref{dynamic for a in L infty}   indicates that the initial bump region of the density with large amplitude will disappear within a very short time which can be regarded as the characteristic phenomenon induced by the short pulse type initial data.

    \item[(2)] To make \eqref{dynamic for a in L infty} clear, we first remark that   at the very beginning,  the density function decays as fast as the solution to the toy model:   $\f{d}{dt}f+\gamma f^2=0$ with $f|_{t=0}=\delta^{-\alpha}$. Indeed, the toy model has the explicit solution, i.e., $f(t)=(\delta^\alpha+\gamma t)^{-1}$. Due to the monotonic property of $f(t)$, one may introduce the dyadic decomposition to the solution itself and calculate that in which interval, $[t_{j-1},t_j]$, the solution  $f(t)$ will behave like $2^{N_0+1-j}$ with $1\le j\le N_0-14$. For our case, this leads to the time sequence $\{t_j\}_{1\le j\le N_0-14}$ and the corresponding estimates in \eqref{dynamic for a in L infty}. We emphasize that it is crucial to get  \eqref{dynamic for a in L 1 L infty}.

     \item[(3)] The factor $\ln \delta^{-1}$ in \eqref{dynamic for a in L 1 L infty} is very important to  the propagation of $L^p$ norm for the derivative of the density. As a direct consequence, we show that it has single exponential growth(i.e.,\, $\exp(\delta^{-\th})$  with some $\th>0$). We refer to \eqref{high order energy estimate} for more details.  This improvement gives us  the chance to prove the third result.

\end{enumerate}
\end{rmk}

\begin{rmk} \begin{enumerate}
\item[(1)]  \eqref{decay of energy 1} and \eqref{decay of energy 2}, the decay estimates of the lower order energy, show that eventually we can get the sharp decay rate, $\w{t}^{-\f34},$ of the convergence to the equilibrium for $\|\vv u\|_{L^2}$ and $\|\rho-1\|_{L^2}$ by the assumption \eqref{initial ansatz in L 1}. We emphasize that the sharpness means that it is consistent with the linearized theory in \cite{MN1}.
  \item[(2)]  To get the decay estimate of the lower order energy, we do not follow the standard perturbation framework (see \cite{MN1}) since initially the density is far away from the constant state.  Instead, we use the time-frequency splitting method to calculate   $\int_{|\xi|\lesssim\langle t\rangle^{-1/2}}\bigl(\gamma|\widehat{{(\rho-1)}}(t,\xi)|^2+|\widehat{\r\vv u}(t,\xi)|^2\bigr)\,d\xi$.
This strategy is motivated by the one introduced  by Wiegner (see \ccite {Wiegner}) and by Schonbeck (see \ccite{Schonbek}, \ccite{Schonbek2}) to derive the decay estimate for the solutions of classical Navier-Stokes equation (see also \cite{CZ6} for the
application of this method to 2D inhomogeneous Navier-Stokes
system).
\end{enumerate}
\end{rmk}

\subsubsection{Main result(III): uniform-in-time propagation of high order regularity} Based on the decay estimates of the lower order energy, we further investigate the uniform-in-time propagation of the higher order regularity. Before that, we recall
the following definition of the Besov norms from \cite{bcd} for
instance.

\begin{defi}
\label {S0def1} {\sl  Let us consider a smooth function~$\vf $
on~$\R,$ the support of which is included in~$[3/4,8/3]$ such that
$$
\forall
 \tau>0\,,\ \sum_{j\in\Z}\varphi(2^{-j}\tau)=1\andf \chi(\tau)\eqdefa 1 - \sum_{j\geq
0}\varphi(2^{-j}\tau) \in \cD([0,4/3]).
$$
Let us define
\beq \label{S1eq18}
\dot\Delta_ja\eqdefa\cF^{-1}(\varphi(2^{-j}|\xi|)\widehat{a}),
 \andf \Delta_{-1}a\eqdefa\cF^{-1}(\chi(|\xi|)\widehat{a}).
\eeq
Let $(p,r)$ be in~$[1,+\infty]^2$ and~$s$ in~$\R$. We define the Besov norm by
$$
\|a\|_{\dB^s_{p,r}}\eqdefa\big\|\big(2^{js}\|\dot\Delta_j
a\|_{L^{p}}\big)_j\bigr\|_{\ell ^{r}(\ZZ)}.
$$
We  denote the constant $ \bar{c}_*\eqdefa \|\mathcal{F}^{-1}\bigl(\f{\xi\otimes\xi}{|\xi|^2}\varphi(|\xi|)\bigr)\|_{L^1}$ in the rest of this paper.
}
\end{defi}

\begin{theorem}\label{propagation regularity thm}
{\sl Let
$\al\leq\min\Bigl\{\f{1+\f{1}{3\gamma}}{2+\f{1}{4\gamma}},\ \f{1}{24\bar{c}_*+\f{7}{4}}\Bigr\}$ and $\delta\in\bigl(0,2^{-\f{15}{\al}}\bigr).$  We assume in addition that
\beq\label{initial ansatz for na a in L 6}
\delta^{\al-3}\bigl(\|\r_0-1\|_{L^1}+\|\r_0\vv u_0\|_{L^1}\bigr)+\delta^{2\al-1}\|\na a_0\|_{L^2}^2+\delta^{2\al+1}\|\na a_0\|_{L^6}^2\leq C.
\eeq
Then under the assumptions of Theorem \ref{global existence thm},
\eqref{CNS} has a unique global solution $(\r,\vv u)$ such that $a\in C(\R_+;H^1\cap W^{1,6}(\R^3))$, $\vv u\in C(\R_+;H^2(\R^3))$ and
\beq\label{high order energy estimate}
\|(\na a,\na^2\vv u)(t)\|_{L^2}^2\leq
\left\{\begin{aligned}
&C\exp\Bigl(C\delta^{-(24\bar{c}_*+1)\al}\Bigr),\quad\text{for }\ t\in[0,2T_2]\\
&C\e^2\delta^{2-\f32\al},\quad\text{for }\ t\geq 2T_2.
\end{aligned}\right.
\eeq
where $T_2\eqdefa(\e\delta^{1-\f34\al})^{-2}$.
Moreover, there hold
\beq\label{dynamics of na a in L 6}
\|\na a(t)\|_{L^6}^2\leq
\left\{\begin{aligned}
&C\exp\Bigl(C\delta^{-(24\bar{c}_*+1)\al}\Bigr),\quad\text{for}\  t\in[0,2T_2],\\
&C\e^3\delta^{3-\f72\al},\quad\text{for }\ t\in[2T_2,4T_2],\\
&C\e^3\delta^{5(1-\al)},\quad\text{for }\ t\in[4T_2,\infty),
\end{aligned}\right.
\eeq
and
\beq\label{energy in critical space}
\sup_{t\geq 2T_2}\bigl(\|(\vv u,\r-1)(t)\|_{\dB^{\f12}_{2,1}}+\|(\r-1)(t)\|_{\dB^{\f34}_{4,1}}\bigr)+\int_{2T_2}^\infty\|\na\vv u(t)\|_{L^\infty}dt\leq C\e^{\f54}\delta^{\f{21}{32}}.
\eeq}
\end{theorem}

\begin{rmk} Again our theorem holds also for $\delta\in[2^{-\f{15}{\alpha}},1]$.
The additional restriction that $\al\leq\min\Bigl\{\f{1+\f{1}{3\gamma}}{2+\f{1}{4\gamma}},\ \f{1}{24\bar{c}_*+\f{7}{4}}\Bigr\}$   enables us to use  \eqref{dynamic for a in L infty} and \eqref{decay of energy 1}   to get the uniform-in-time propagation of the high order regularity. If it is removed,  (\ref{high order energy estimate}-\ref{dynamics of na a in L 6}) still hold before the moment $T_2$ but after that the
high order energy will grow exponentially in time.
\end{rmk}

\begin{rmk}
  Estimates  (\ref{high order energy estimate}) and \eqref{dynamics of na a in L 6}  show that  $T_2$ is a crucial moment such that the behavior of the high order energy will behave quite differently before and after.  Roughly speaking, before this moment, the high order   energy will grow exponentially compared to the initial time. This is mainly   induced by the fast collapse of the initial bump region of the density(see \eqref{dynamic for a in L 1 L infty}).  After this moment, \eqref{decay of energy 1} plays the essential role in showing that   the high order energy decays and eventually becomes sufficiently small in critical Besov space(see \eqref{energy in critical space}).
\end{rmk}

\begin{rmk} As a direct consequence of the estimate $\int_{2T_2}^\infty\|\na\vv u(t)\|_{L^\infty}\,dt\leq C\e^{\f54}\delta^{\f{21}{32}}$ in  \eqref{energy in critical space}, any higher order regularities of the solutions to  \eqref{CNS} can be propagated uniformly in time.
\end{rmk}

\subsubsection{Idea of the proof} In this subsection, we shall outline the main idea and the strategy to the proof of our main results.
\smallskip

\underline{\it (i).Proof of Theorem \ref{global existence thm}.}  The main idea of the proof stems from the new observations for effective viscous flux and the application of the interpolation method.
\smallskip

$\bullet$ {\it Effective viscous flux.} Our first observation is to let effective viscous flux $F$ involved directly in the lower order energy by multiplying $\vv u$-equation with $\dot{\vv u}=\p_t\vv u+\vv u\cdot\na\vv u$ (see Proposition \ref{prop for curl u and q}). It has two  advantages. On
   one hand, it helps a lot to control the density and velocity in the nonlinear terms. On the other hand, it enables to have the full dissipation or damping structure for the lower order energy when the density is bounded from below and above. This is crucial to get the decay estimate in Theorem \ref{global dynamics thm}. Our second observation lies in Lemma \ref{lemma for rho} and Lemma \ref{lower bound lemma}. The key point is that  $\|F\|_{L^2_t(L^\infty)}$ can be applied to derive the pointwise control on the density from below and above via Lagrangian coordinate.
\smallskip

$\bullet$ {\it Interpolation method.} As we mentioned before,  the short pulse type data are very sensitive to the Sobolev norms and $L^p$ norms but they are consistent with the interpolation theory. To catch the precise order w.r.t. $\delta$ of the lower order energy, in particular, for $\|\na \vv u\|_{L^3}$ and $\|a\|_{L^3}$, we make full use of the interpolation, for instance, see Lemma \ref{prop for a in L 3}, where we recognize that $H(\rho)$ plays the role as $\|a\|_{L^1}$ from the point view of the order w.r.t. $\delta$. In general we always think that $H(\rho)$ behaves like $\|a\|_{L^2}^2$.
\smallskip

\underline{\it (ii). Proof of Theorem \ref{global dynamics thm}.} The key points rely on the new decay estimates on the density and the time-frequency splitting method to calculate   $\int_{|\xi|\lesssim\langle t\rangle^{-1/2}}\bigl(\gamma|\widehat{{(\rho-1)}}(t,\xi)|^2+|\widehat{\r\vv u}(t,\xi)|^2\bigr)\,d\xi$.

$\bullet$ {\it Estimates on the density.} By using Lagrangian coordinate, one may find that the density enjoys the similar structure as the toy model $\pa_tf+f^2=0$ with $f|_{t=0}=\delta^{-\alpha}$. This shows that at the beginning $f(t)$ decreases very fast. To quantize it for the true model, we introduce the dyadic decomposition on the initial data and show that within a very small time, the true model behaves exactly the same as the toy model. We refer to Proposition \ref{dynamics prop for a over [0,1]} for details. These estimates have two direct applications. The first one  enables us to find that the lower order energy has full dissipation or damping structure after the moment $T_0$ since from this moment the density is bounded from below and above by two universal constants. The second application is to show that the high order energy will have single exponential growth(i.e.,\, $\exp(\delta^{-\th})$  with some $\th>0$, see \eqref{high order energy estimate}) while it is expected to have double exponential growth(i.e.,\, $\exp(\exp(\delta^{-\th}))$  with some $\th>0$, see \eqref{S2eq8}). This will be crucial for us to investigate the uniform-in-time control of the high order energy.

$\bullet$ {\it Calculation of  $\int_{|\xi|\lesssim\langle t\rangle^{-1/2}}\bigl(\gamma|\widehat{{(\rho-1)}}(t,\xi)|^2+|\widehat{\r\vv u}(t,\xi)|^2\bigr)\,d\xi$.} As we mentioned before,  after the moment $T_0$, the lower order energy becomes fully dissipated or damped. But we are still blocked to get the decay estimates  because we have no idea to propagate the condition \eqref{initial ansatz in L 1} till the moment $T_0$. Another difficulty comes from the fact that initially the density is far away from the constant state. To overcome this difficulty, we resort to calculate  $\int_{|\xi|\lesssim\langle t\rangle^{-1/2}}\bigl(\gamma|\widehat{{(\rho-1)}}(t,\xi)|^2+|\widehat{\r\vv u}(t,\xi)|^2\bigr)\,d\xi$ instead of $\int_{|\xi|\lesssim\langle t\rangle^{-1/2}}\bigl(\gamma|\widehat{{(\rho-1)}}(t,\xi)|^2+|\widehat{\vv u}(t,\xi)|^2\bigr)\,d\xi$. The advantage lies in the fact it enables us to avoid using the smallness condition on $\rho-1$.
\smallskip

\underline{\it (iii). Proof of Theorem \ref{propagation regularity thm}.} The main strategy is to put together all the previous results in a proper way. The key point is to obtain the dynamic behavior of $\|\na a\|_{L^6}$.

$\bullet$ We first use \eqref{dynamic for a in L 1 L infty} to prove that $\|\na a(t)\|_{L^6}$  has single exponential growth(i.e.,\, $\exp(\delta^{-\th})$  with some $\th>0$, see \eqref{high order energy estimate}) within $[0,4T_2]$. It implies that  $\ln(1+\|\na a(t)\|_{L^6})$ grows at most linearly depending on $\delta^{-\th}$.

$\bullet$ In the time interval $[2T_2,4T_2]$, due to  \eqref{dynamic for a in L infty},  we conclude that the density is now sufficiently close to the constant state. Under the restriction on $\alpha$, it is possible to get that  $\|a(t)\|_{L^\infty}\ln(1+\|\na a(t)\|_{L^6})$ is bounded from above by some universal constant. From this point together with the logarithmic interpolation inequality for $\|\na\vv u\|_{L^\infty}$(see \eqref{E 74}-\eqref{E 77}), we find that the norm $\|\na a\|_{L^6}$ decays exponentially in $[T_2,4T_2]$ and becomes small at the moment $2T_2$. By the decay estimates of the lower order energy and  standard interpolation inequality, we show that at the moment $2T_2$, $\r-1$ and $\vv u$ are both small in some critical Besov spaces which immediately imply that $\|\na\vv u\|_{L^1([0,\infty);L^\infty)}$ is finite. As a consequence, high order regularities of thus constructed solutions can be propagated uniformly in time.

\medskip

We end this section with some notations that will be used throughout the paper.
\smallskip

\noindent{\bf Notations:}  The notation $C$ denotes a universal positive constant that is independent of $\delta$ but may change from line to line.
 The notation $a\lesssim b$,
 means that there is a universal constant $C>0$, such that $a\leq Cb$. While notaion $a\sim b$ means both $a\lesssim b$ and $b\lesssim a$. We shall denote by~$(a\,|\,b)$
the $L^2(\R^3)$ inner product of $a$ and $b.$  And for any $s\in\R$, $H^s(\R^3)$ denotes the classical  $L^2$ based Sobolev spaces with the norm $\|\cdot\|_{H^s}$.
The notation $\|\cdot\|_{L^p}$ stands for  the $L^p(\R^3)$ norm for $1\leq p \leq \infty$.
 Finally, we
denote $L^p_T(L^q)$ to be the space $L^p([0,T];
L^q(\R^3))$ for $1\leq p,q\leq\infty$.

\setcounter{equation}{0}
\section{Global well-posedness of the system (\ref{CNS})}

In this section, we shall present the {\it a priori} estimates for smooth enough solutions of the system (\ref{CNS}), which ensure
  the global existence of strong solutions to \eqref{CNS}, under the assumptions of Theorem \ref{global existence thm}. As a matter of fact, the existence and uniqueness of strong solutions to the system \eqref{CNS} on
  sufficiently small time interval were proved in \cite{Solo76}.
 In order to close the energy estimates below, we split the proof into the following two parts:
\begin{itemize}
\item The first part focuses on the {\it a priori} estimates for smooth enough solutions of the system (\ref{CNS}). In particular, we investigate the size of the energy norm  of  $F=\div \vv u-a$ (effective viscous   flux) and   $\sqrt{\r}\dot{\vv u}$ (material derivative of the velocity);

\item The second part concentrates on a continuous argument, that is, under the ansatz on the solution itself, we prove that the lifespan of the solutions can actually be extended to be infinity, and then we shall prove the propagation of $\dot{W}^{1,q}(\R^3)$ regularity of the density function and also
    the $\dot{W}^{2,1}_q([0,T]\times\R^3)$ regularity of the velocity field (without the uniform-in-time bounds). So that the local unique strong solution is the global one.
\end{itemize}


\subsection{The {\it a priori} estimates for the system \eqref{CNS}} We begin with the standard energy equality for  (\ref{CNS}).

\begin{prop}\label{prop for basic energy}
{\sl Let $(\r,\vv u)$ be a smooth enough solution of \eqref{CNS} on $[0,T]$. Then for $t\leq T$, one has
\beq\label{basic energy equality}
\f{d}{dt}\Bigl(\f12\|\sqrt\r\vv u(t)\|_{L^2}^2+H(\r(t))\Bigr)+\|\na\vv u\|_{L^2}^2=0.
\eeq}
\end{prop}
\begin{proof}
In view of the $\r$ equation of \eqref{CNS}, we write the $\vv u$ equation of \eqref{CNS} as
\beq\label{eq for u}
\r(\p_t\vv u+\vv u\cdot\na\vv u)-\D\vv u+\na\r^\gamma=\vv 0.
\eeq
By taking $L^2$-inner product of \eqref{eq for u} with $\vv u$ leads to
\beq\label{S2eq1}
\f12\int_{\R^3}\r[\p_t(|\vv u|^2)+\vv u\cdot\na(|\vv u|^2)]dx+\|\na\vv u\|_{L^2}^2=-\int_{\R^3}\vv u\cdot\na\r^\gamma\,dx.
\eeq
Yet by using integration by parts and the transport equation of \eqref{CNS}, one has
\beno
\int_{\R^3}\r[\p_t(|\vv u|^2)+\vv u\cdot\na(|\vv u|^2)]\,dx=\f{d}{dt}\|\sqrt\r\vv u\|_{L^2}^2.
\eeno

On the other hand, recalling the definition of $h(\r)$ and $H(\r)$ in \eqref{def of H rho}, for $\ga>1,$ we get, by multiplying $\f\ga{\ga-1}\bigl(\r^{\ga-1}-1\bigr)$ to the transport equation of \eqref{CNS},
\beno\begin{aligned}
\p_t h(\r)+\vv u\cdot\na h(\r)=&-\f\ga{\ga-1}\bigl(\r^\ga-\r\bigr)\div\vv u\\
=&-h(\r)\div\vv u-(\r^\gamma-1)\div \vv u,
\end{aligned}\eeno
which implies
\beno
\p_t h(\r)+\div \left(h(\r)\vv u\right)+(\r^\ga-1)\div \vv u=0.
\eeno
By integrating the above equation over $\R^3,$ we find
\beq \label{S2eq2}
\f{d}{dt}H(\r)+\int_{\R^3}(\r^\ga-1)\div \vv u\,dx=0\quad
\text{i.e.,}\quad\int_{\R^3}\vv u\cdot\na\r^\gamma\,dx=\f{d}{dt}H(\r).
\eeq

When $\ga=1,$ by multiplying $(1+\ln\r)$ to the transport equation of \eqref{CNS}, we write
\beno
\p_t(\r\ln\r)+\vv u\cdot\na (\r\ln\r)+\left(\r\ln\r+\r\right)\div \vv u=0.
\eeno
By subtracting the transport equation of \eqref{CNS} from the above equation, we obtain
\beno
\p_t\left(\r\ln\r-(\r-1)\right)+\vv u\cdot\na \left(\r\ln\r-(\r-1)\right)+(\r\ln\r-(\r-1)) \div \vv u+(\r-1)\div\vv u=0,
\eeno
which ensures \eqref{S2eq2} for $\gamma=1.$

By inserting the above equalities into \eqref{S2eq1}, we obtain \eqref{basic energy equality}. This finishes the proof of  Proposition \ref{prop for basic energy}.
\end{proof}

\subsubsection{Estimates of $\curl\vv u$ and $F$} In view of  \eqref{def of a, q and dot u} and the fact that
\beq \label{S2eq3} -\D\vv u=\curl\curl\vv u-\na\div\vv u, \eeq
 we write \eqref{eq for u} as
\beq\label{E 7}
-\curl\curl\vv u+\na F=\r\dot{\vv u}.
\eeq
Before deriving the energy estimates for $\curl\vv u$ and $F$, we first consider the elliptic estimate for \eqref{E 7}.

\begin{lem}\label{lemma for elliptic system}
{\sl Let $(\r,\vv u)$ be a smooth enough solution of \eqref{CNS} on $[0,T]$. Then for $t\leq T$, one has
\beq\label{elliptic estimates}
\|\na\curl\vv u\|_{L^p}+\|\na F\|_{L^p}\lesssim\|\r\dot{\vv u}\|_{L^p},\quad\text{for any}\,\, p\in(1,\infty).
\eeq}
\end{lem}
\begin{proof}
By applying the operators $\curl$ and  $\div$ to  \eqref{E 7} respectively, one has
\beno
\D\curl\vv u=\curl(\r\dot{\vv u})\quad\text{and}\quad \D F=\div(\r\dot{\vv u}),
\eeno
so that
\beno
\na\curl\vv u=-\na(-\D)^{-1}\curl(\r\dot{\vv u}),   \quad\na F=-\na(-\D)^{-1}\div(\r\dot{\vv u}).
\eeno
which implies \eqref{elliptic estimates}.
\end{proof}

With Lemma \ref{lemma for elliptic system}, we derive  the energy estimates of $\curl\vv u$ and $F$.

\begin{prop}\label{prop for curl u and q}
{\sl Let $(\r,\vv u)$ be a smooth enough solution of \eqref{CNS} on $[0,T]$. Then for $t\leq T$, one has
\beq\label{estimates for curl u and q}\begin{aligned}
&\f{d}{dt}\Bigl(\f12\|(\curl\vv u, F)(t)\|_{L^2}^2-\gamma H(\r)\Bigr)+\|\sqrt\r\dot{\vv u}\|_{L^2}^2-\gamma\|\div\vv u\|_{L^2}^2\\
&\lesssim\|\r\|_{L^\infty}^{\f12}\Bigl(\bigl(\|\curl\vv u\|_{L^2}+\|F\|_{L^2}\bigr)^{\f12}\|\sqrt\r\|_{L^\infty}^{\f12}\|\sqrt\r\dot{\vv u}\|_{L^2}^{\f12}+\|a\|_{L^3}\Bigr)\|\sqrt\r\dot{\vv u}\|_{L^2}\|\na\vv u\|_{L^2}.
\end{aligned}\eeq}
\end{prop}
\begin{proof}
Noticing that $\dot{\vv u}=\p_t\vv u+\vv u\cdot\na\vv u$ and
\beno
\curl\curl\vv u
-\na F=\curl\curl\vv u
-\na\div\vv u+\na a,
\eeno
we get, by taking $L^2$-inner product of \eqref{E 7} with $\dot{\vv u},$ that
\beq\label{E 9a}\begin{aligned}
&\f12\f{d}{dt}\Bigl(\|\curl\vv u\|_{L^2}^2+\|\div\vv u\|_{L^2}^2\Bigr)+\|\sqrt\r\dot{\vv u}\|_{L^2}^2\\
&\qquad=-\bigl(\curl\curl\vv u\,|\,\vv u\cdot\na\vv u\bigr)+\bigl(\na F\,|\,\vv u\cdot\na\vv u\bigr)-\bigl(\na a\,|\,\p_t\vv u\bigr).
\end{aligned}\eeq

\no{\Large $\bullet$} {\bf The estimate of $\bigl(\na a\,|\,\p_t\vv u\bigr)$.} We first observe that
\beq\label{E 10}
\bigl(\na a\,|\,\p_t\vv u\bigr)=-\bigl(a\,|\,\p_t\div\vv u\bigr)
=-\f{d}{dt}\bigl(a\,|\,\div\vv u\bigr)+\bigl(\p_ta\,|\,\div\vv u\bigr).
\eeq
Note that $a=\r^\gamma-1$, we deduce from the transport equation of \eqref{CNS} that
\beq\label{E 11}
\p_ta+\vv u\cdot\na a+\gamma(a+1)\div\vv u=0.
\eeq
By taking $L^2$-inner product of \eqref{E 11} with $\div\vv u,$ we write
\beno\begin{aligned}
\bigl(\p_ta\,|\,\div\vv u\bigr)&=-\gamma\|\div\vv u\|_{L^2}^2
-\gamma\bigl(a\,\div\vv u\,|\,\div\vv u\bigr)-\bigl(\vv u\cdot\na a\,|\,\div\vv u\bigr)\\
&=-\gamma\|\div\vv u\|_{L^2}^2
-\gamma\bigl(a\,\div\vv u\,|\,F\bigr)-\bigl(\vv u\cdot\na a\,|\,F\bigr)
-\bigl[\gamma\bigl(\div\vv u\,|\,a^2\bigr)+\bigl(\vv u\cdot\na a\,|\,a\bigr)\bigr],
\end{aligned}\eeno
which implies
\beq\label{E 12}
\bigl(\p_ta\,|\,\div\vv u\bigr)=-\gamma\|\div\vv u\|_{L^2}^2
-(\gamma-1)\bigl(a\,\div\vv u\,|\,F\bigr)+\bigl(\vv u\cdot\na F\,|\,a\bigr)
+(2\gamma-1)\bigl(\vv u\cdot\na a\,|\,a\bigr).
\eeq
Yet it follows from  \eqref{E 11} that
\beno\begin{aligned}
\bigl(\vv u\cdot\na a\,|\,a\bigr)&=-\bigl(\p_t a\,|\,a\bigr)-\gamma\bigl((a+1)\div\vv u\,|\,a\bigr)\\
&=-\f12\f{d}{dt}\|a\|_{L^2}^2+2\gamma\bigl(\vv u\cdot\na a\,|\,a\bigr)+\gamma\int_{\R^3}\vv u\cdot\na a\, dx,
\end{aligned}\eeno
from which together with  \eqref{S2eq2},
we deduce that
\beq\label{E 13}
(2\gamma-1)\bigl(\vv u\cdot\na a\,|\,a\bigr)=\f{d}{dt}\Bigl(\f12\|a\|_{L^2}^2-\gamma H(\r)\Bigr).
\eeq

By inserting the estimates \eqref{E 12} and \eqref{E 13} into  \eqref{E 10}, we obtain
\beq\label{E 14}\begin{aligned}
\bigl(\na a\,|\,\p_t\vv u\bigr)=&\f{d}{dt}\Bigl(\f{1}{2}\|a\|_{L^2}^2
-\bigl(a\,|\,\div\vv u\bigr)-\gamma H(\r)\Bigr)\\
&\quad
-\gamma\|\div\vv u\|_{L^2}^2
-(\gamma-1)\bigl(a\,\div\vv u\,|\,F\bigr)+\bigl(\vv u\cdot\na F\,|\,a\bigr).
\end{aligned}\eeq

Notice that
\begin{align*}
\f12\|\div\vv u\|_{L^2}^2+\f{1}{2}\|a\|_{L^2}^2
-\bigl(a\,|\,\div\vv u\bigr)&=\f12\|\div\vv u-a\|_{L^2}^2=\f12\|F\|_{L^2}^2.
\end{align*}
By substituting \eqref{E 14} into \eqref{E 9a},
 we get
\beq\label{E 16}\begin{aligned}
&\f{d}{dt}\Bigl(\f12\|\curl\vv u\|_{L^2}^2+\f12\|F\|_{L^2}^2
-\gamma H(\r)\Bigr)+\|\sqrt\r\dot{\vv u}\|_{L^2}^2
-\gamma\|\div\vv u\|_{L^2}^2\\
&\quad=-\bigl(\curl\curl\vv u\,|\,\vv u\cdot\na\vv u\bigr)+\bigl(\na F\,|\,\vv u\cdot\na\vv u\bigr)
+(\gamma-1)\bigl(a\,\div\vv u\,|\,F\bigr)-\bigl(\vv u\cdot\na F\,|\,a\bigr).
\end{aligned}\eeq

\no{\Large $\bullet$}  {\bf The estimate of the r.h.s of \eqref{E 16}.}
We first observe that
\begin{align*}
&\bigl|\bigl(\curl\curl\vv u\,|\,\vv u\cdot\na\vv u\bigr)\bigr|+\bigl|\bigl(\na F\,|\,\vv u\cdot\na\vv u\bigr)\bigr|\\
&\leq \bigl(\|\curl\curl\vv u\|_{L^2}
+\|\na F\|_{L^2}\bigr)\|\vv u\|_{L^6}\|\na\vv u\|_{L^3}\\
&
\lesssim \bigl(\|\na\curl\vv u\|_{L^2}+\|\na F\|_{L^2}\bigr)\|\na\vv u\|_{L^2}
\|\na\vv u\|_{L^3}.
\end{align*}
Due to facts that
$
-\D\vv u=\curl\curl\vv u-\na\div\vv u
$
and $F=\div\vv u-a$,
one has
\beno
\vv u=(-\D)^{-1}\curl\curl\vv u-\na(-\D)^{-1} F-\na(-\D)^{-1} a,
\eeno
which implies
\beq\label{E 18}
\|\na\vv u\|_{L^p}
\lesssim\|\curl\vv u\|_{L^p}+\|F\|_{L^p}+\|a\|_{L^p},\quad\text{for any}\
p\in(1,\infty).
\eeq
Therefore, we obtain
\beq\label{E 19}\begin{aligned}
&\bigl|\bigl(\curl\curl\vv u\,|\,\vv u\cdot\na\vv u\bigr)\bigr|
+\bigl|\bigl(\na F\,|\,\vv u\cdot\na\vv u\bigr)\bigr|\\
&
\leq\bigl(\|\na\curl\vv u\|_{L^2}+\|\na F\|_{L^2}\bigr)\|\na\vv u\|_{L^2}\bigl(\|\curl\vv u\|_{L^3}+\|F\|_{L^3}+\|a\|_{L^3}\bigr).
\end{aligned}\eeq

For the last two terms in the r.h.s. of \eqref{E 16}, we have
\beno\begin{aligned}
&(\gamma-1)\bigl|\bigl(a\,\div\vv u\,|\,F\bigr)\bigr|+\bigl|\bigl(\vv u\cdot\na F\,|\,a\bigr)\bigr|\lesssim\gamma\|a\|_{L^3}\|\na\vv u\|_{L^2}\|\na F\|_{L^2},
\end{aligned}\eeno
which along with \eqref{E 19} implies the r.h.s. of \eqref{E 16} is bounded by
\beq\label{E 20}\begin{aligned}
\lesssim&\bigl(\|\curl\vv u\|_{L^3}+\|F\|_{L^3}+\|a\|_{L^3}\bigr)\Bigl(\|\na\curl\vv u\|_{L^2}+\|\na F\|_{L^2}\Bigr)\|\na\vv u\|_{L^2}\\
\lesssim&\Bigl(\|\curl\vv u\|_{L^2}^{\f12}\|\na\curl\vv u\|_{L^2}^{\f12}+\|F\|_{L^2}^{\f12}\|\na F\|_{L^2}^{\f12}\Bigr)\Bigl(\|\na\curl\vv u\|_{L^2}+\|\na F\|_{L^2}\Bigr)\|\na\vv u\|_{L^2}\\
&\quad
+\|a\|_{L^3}\Bigl(\|\na\curl\vv u\|_{L^2}+\|\na F\|_{L^2}\Bigr)\|\na\vv u\|_{L^2},
\end{aligned}\eeq
where we used the interpolation inequality, $\|f\|_{L^3}\lesssim\|f\|_{L^2}^{\f12}\|f\|_{\dot{H}^1}^{\f12},$ in the last inequality.

Thanks to \eqref{elliptic estimates}, we get, by inserting  \eqref{E 20} into  \eqref{E 16}, that
\beno\begin{aligned}
&\f{d}{dt}\Bigl(\f12\|(\curl\vv u, F)(t)\|_{L^2}^2
-\gamma H(\r)\Bigr)+\|\sqrt\r\dot{\vv u}\|_{L^2}^2
-\gamma\|\div\vv u\|_{L^2}^2\\
&\lesssim\bigl(\|\curl\vv u\|_{L^2}+\|F\|_{L^2}\bigr)^{\f12}\|\r\dot{\vv u}\|_{L^2}^{\f32}\|\na\vv u\|_{L^2}
+\|a\|_{L^3}\|\r\dot{\vv u}\|_{L^2}\|\na\vv u\|_{L^2},
\end{aligned}\eeno
 from which and the fact that $\|\r\dot{\vv u}\|_{L^2}\leq\|\sqrt\r\|_{L^\infty}\|\sqrt\r\dot{\vv u}\|_{L^2}$, we obtain \eqref{estimates for curl u and q}. This completes the proof of Proposition \ref{prop for curl u and q}.
\end{proof}

\subsubsection{Estimate  of $\sqrt\r\dot{\vv u}$} In this subsection, we derive the energy estimate for $\sqrt\r\dot{\vv u}$.

\begin{prop}\label{prop for higher order energy estimates}
Let $(\r,\vv u)$ be a smooth enough solution of \eqref{CNS} on $[0,T]$. Then for $t\leq T$, one has
\beq\label{estimates for sqrt rho dot u}\begin{aligned}
&\f12\f{d}{dt}\Bigl(\|\sqrt\r\dot{\vv u}\|_{L^2}^2+\gamma\|\div\vv u\|_{L^2}^2\Bigr)+\|\na\dot{\vv u}\|_{L^2}^2\\
&\lesssim \bigl(\bigl(\|\curl\vv u\|_{L^2}+\|F\|_{L^2}\bigr)^{\f12}\|\sqrt\r\|_{L^\infty}^{\f12}\|\sqrt\r\dot{\vv u}\|_{L^2}^{\f12}+\|a\|_{L^3}\bigr)\\
&\qquad\times\bigl(\|\sqrt\r\|_{L^\infty}\|\sqrt\r\dot{\vv u}\|_{L^2}+\|a\|_{L^6}\bigr)\bigl(\|\na\vv u\|_{L^2}+\|\na\dot{\vv u}\|_{L^2}\bigr).
\end{aligned}\eeq
\end{prop}
\begin{proof} Let us start with the derivation of the evolution equation of $\dot{\vv u}.$

\no{\Large $\bullet$} {\bf  Derivation of the equation for $\dot{\vv u}.$}

We first get, by applying $D_t\eqdefa \p_t+\vv u\cdot\na$ to $\r\dot{\vv u}$ and using the transport  equation of \eqref{CNS}, that
\beq\label{E 21}
D_t(\r\dot{\vv u})=\r D_t\dot{\vv u}+D_t\r\dot{\vv u}=\r D_t\dot{\vv u}-\r\dot{\vv u}\,\div\vv u.
\eeq
While in view of \eqref{E 7}, we write
\beq\label{E 22}\begin{aligned}
&\vv u\cdot\na\bigl(\curl\curl\vv u-\na F\bigr)\\
&=\div\bigl(\curl\curl\vv u\otimes\vv u\bigr)-\div\bigl(\na F\otimes\vv u\bigr)
-\bigl(\curl\curl\vv u-\na F\bigr)\div\vv u\\
&=-\div\bigl(\D\vv u\otimes\vv u\bigr)+\div\bigl(\na\div\vv u\otimes\vv u\bigr)-\div\bigl(\na F\otimes\vv u\bigr)+\r\dot{\vv u}\div\vv u.
\end{aligned}\eeq
Then by applying $D_t$ to \eqref{E 7} and using
 \eqref{E 21} and \eqref{E 22},  we obtain
\beq\label{eq for dot u}\begin{aligned}
\r D_t\dot{\vv u}-\D\dot{\vv u}+\na D_ta=&\bigl[\div\bigl(\D\vv u\otimes\vv u\bigr)-\D(\vv u\cdot\na\vv u)\bigr]\\
&+\bigl[\na(\vv u\cdot\na\div\vv u)-\div\bigl(\na\div\vv u\otimes\vv u\bigr)\bigr]\\
&+\bigl[\div\bigl(\na F\otimes\vv u\bigr)-\na(\vv u\cdot\na F)\bigr],
\end{aligned}\eeq
where we used the fact that $F=\div\vv u-a$.

\no{\Large $\bullet$} {\bf Energy estimate of  $\dot{\vv u}$.}

By taking $L^2$-inner product of \eqref{eq for dot u} with $\dot{\vv u},$ we find
\beq\label{E 23}\begin{aligned}
&\f12\f{d}{dt}\|\sqrt\r\dot{\vv u}\|_{L^2}^2+\|\na\dot{\vv u}\|_{L^2}^2=
\bigl(D_ta\,|\,\div\dot{\vv u}\bigr)+I+II+III,
\end{aligned}\eeq
where
\beq\label{E 24}\begin{aligned}
I&\eqdefa\bigl(\div\bigl(\D\vv u\otimes\vv u\bigr)-\D(\vv u\cdot\na\vv u)\,|\,\dot{\vv u}\bigr),\\
II&\eqdefa\bigl(\na(\vv u\cdot\na\div\vv u)-\div\bigl(\na\div\vv u\otimes\vv u\bigr)\,|\,\dot{\vv u}\bigr),\\
III&\eqdefa\bigl(\div\bigl(\na F\otimes\vv u\bigr)-\na(\vv u\cdot\na F)\,|\,\dot{\vv u}\bigr).
\end{aligned}\eeq

For the first term on the r.h.s. of \eqref{E 23}, we get, by using \eqref{E 11}, that
\beno\begin{aligned}
\bigl(D_ta\,|\,\div\dot{\vv u}\bigr)&=-\gamma\bigl(\div\vv u\,|\,\div\dot{\vv u}\bigr)
-\gamma\bigl(a\div\vv u\,|\,\div\dot{\vv u}\bigr)\\
&
=-\f{\gamma}{2}\f{d}{dt}\|\div\vv u\|_{L^2}^2-\gamma\bigl(\div\vv u\,|\,\div(\vv u\cdot\na\vv u)\bigr)-\gamma\bigl(a\div\vv u\,|\,\div\dot{\vv u}\bigr),
\end{aligned}\eeno
from which and \eqref{E 23}, we deduce that
\beq\label{E 25}\begin{aligned}
&\f12\f{d}{dt}\Bigl(\|\sqrt\r\dot{\vv u}\|_{L^2}^2+\gamma\|\div\vv u\|_{L^2}^2\Bigr)+\|\na\dot{\vv u}\|_{L^2}^2=I+II+III+IV,
\end{aligned}\eeq
where
\beno
IV\eqdefa-\gamma\bigl(\div\vv u\,|\,\div(\vv u\cdot\na\vv u)\bigr)-\gamma\bigl(a\div\vv u\,|\,\div\dot{\vv u}\bigr).
\eeno

 \no{\Large $\bullet$} {\bf The estimates of $I$ to $IV.$}

In view of \eqref{E 24}, by using integration by parts, we write
\beno\begin{aligned}
I&=-\bigl(\D\vv u\,|\,\vv u\cdot\na\dot{\vv u}\bigr)+\bigl(\na(\vv u\cdot\na\vv u)\,|\,\na\dot{\vv u}\bigr)\\
&=-\bigl(\div\vv u\na\vv u\,|\,\na\dot{\vv u}\bigr)+\sum_{j=1}^3\bigl[\bigl(\p_j\vv u\,|\,\p_j\vv u\cdot\na\dot{\vv u}\bigr)+\bigl(\p_j\vv u\cdot\na\vv u\,|\,\p_j\dot{\vv u}\bigr)\bigr],
\end{aligned}\eeno
\beno\begin{aligned}
II&=-\bigl(\vv u\cdot\na\div\vv u\,|\,\div\dot{\vv u}\bigr)
+\bigl(\na\div\vv u\,|\,\vv u\cdot\na\dot{\vv u}\bigr)\\
&=\bigl(|\div\vv u|^2\,|\,\div\dot{\vv u}\bigr)-\sum_{j=1}^3\bigl(\div\vv u\,|\,\p_j\vv u\cdot\na\dot{u}^j\bigr),
\end{aligned}\eeno
\beno\begin{aligned}
III&=-\bigl(\na F\,|\,\vv u\cdot\na\dot{\vv u}\bigr)+\bigl(\vv u\cdot\na F\,|\,\div\dot{\vv u}\bigr)=\bigl(F\,|\,\sum_{j=1}^3\p_j\vv u\cdot\na\dot{\vv u}-\div\vv u\div\dot{\vv u}\bigr),
\end{aligned}\eeno
and
\beno\begin{aligned}
\bigl(\div\vv u\,|\,\div(\vv u\cdot\na\vv u)\bigr)&
=\bigl(\div\vv u\,|\,\vv u\cdot\na\div\vv u)\bigr)+\sum_{j=1}^3\bigl(\div\vv u\,|\,\p_j\vv u\cdot\na u^j)\bigr)\\
&
=-\f12\bigl(|\div\vv u|^2\,|\,\div\vv u\bigr)+\sum_{j=1}^3\bigl(\div\vv u\,|\,\p_j\vv u\cdot\na u^j\bigr),
\end{aligned}\eeno
from which, we infer
\beq\label{E 27}
\begin{split}
&|I|+|II|\lesssim\|\na\vv u\|_{L^3}\|\na\vv u\|_{L^6}\|\na\dot{\vv u}\|_{L^2},\\
&|III|\lesssim\|F\|_{L^6}\|\na\vv u\|_{L^3}\|\na\dot{\vv u}\|_{L^2}
\lesssim\|\na F\|_{L^2}\|\na\vv u\|_{L^3}\|\na\dot{\vv u}\|_{L^2},\\
&|IV|\lesssim\|\div\vv u\|_{L^3}\|\na\vv u\|_{L^6}\|\na\vv u\|_{L^2}+\|a\|_{L^6}\|\div\vv u\|_{L^3}\|\div\dot{\vv u}\|_{L^2}.
\end{split}
\eeq

 By substituting the estimates \eqref{E 27} into  \eqref{E 25}, we achieve
\beq\label{E 31}\begin{aligned}
&\f12\f{d}{dt}\Bigl(\|\sqrt\r\dot{\vv u}\|_{L^2}^2+\gamma\|\div\vv u\|_{L^2}^2\Bigr)+\|\na\dot{\vv u}\|_{L^2}^2\\
&\lesssim\|\div\vv u\|_{L^3}\|\na\vv u\|_{L^6}\|\na\vv u\|_{L^2}
+\|\na\vv u\|_{L^3}\Bigl(\|\na\vv u\|_{L^6}+\|\na F\|_{L^2}+\|a\|_{L^6}\Bigr)\|\na\dot{\vv u}\|_{L^2}.
\end{aligned}\eeq
Yet it follows from  \eqref{E 18} and \eqref{elliptic estimates} that
\beno\begin{aligned}
\|\na\vv u\|_{L^3}&\lesssim\|\curl\vv u\|_{L^3}+\|F\|_{L^3}+\|a\|_{L^3}\\
&\lesssim\bigl(\|\curl\vv u\|_{L^2}+\|F\|_{L^2}\bigr)^{\f12}\|\r\dot{\vv u}\|_{L^2}^{\f12}+\|a\|_{L^3},\\
\|\na\vv u\|_{L^6}&\lesssim\|\na\curl\vv u\|_{L^2}+\|\na F\|_{L^2}+\|a\|_{L^6}\lesssim\|\r\dot{\vv u}\|_{L^2}+\|a\|_{L^6}.
\end{aligned}\eeno
By inserting the above estimates into \eqref{E 31}, we arrive at
 \eqref{estimates for sqrt rho dot u}. This completes the proof of  Proposition \ref{prop for higher order energy estimates}.
\end{proof}

\subsubsection{Estimate of $a$} In this subsection, we shall derive the estimates of $\|a\|_{L^6}$.

\begin{prop}\label{prop for a in L 6}
{\sl Let $(\r,\vv u)$ be a smooth enough solution of \eqref{CNS} on $[0,T]$. Then for $t\leq T$, one has
\beq\label{estimates for a in L 6}
6\f{d}{dt}\|a(t)\|_{L^6}^2+\|a\|_{L^6}^2
\lesssim
(1+\|a\|_{L^\infty}^2)\|\r\|_{L^\infty}\|\sqrt\r\dot{\vv u}\|_{L^2}^2.
\eeq}
\end{prop}
\begin{proof} By multiplying \eqref{E 11} with $a^5$ and then integrating the resulting equation over $\R^3$, we find
\beno
\f16\f{d}{dt}\|a\|_{L^6}^6+\f16\int_{\R^3}\vv u\cdot\na a^6\, dx
+\gamma\int_{\R^3}\div\vv u\, a^5\, dx+\gamma\int_{\R^3}\div\vv u\, a^6\, dx=0,
\eeno
from which, we infer
\beq\label{E 32}
\f16\f{d}{dt}\|a\|_{L^6}^6
+\gamma\int_{\R^3}\div\vv u\, a^5\, dx+(\gamma-\f16)\int_{\R^3}\div\vv u \,a^6\, dx=0,
\eeq

Noticing that
\beno\begin{aligned}
&\int_{\R^3}\div\vv u\, a^5\, dx=\int_{\R^3}F a^5\, dx
+\int_{\R^3}a^6\, dx,\\
&\int_{\R^3}\div\vv u\, a^6\, dx=\int_{\R^3}F a^6\, dx
+\int_{\R^3}a^7\, dx,
\end{aligned}\eeno
so that we deduce from \eqref{E 32} that
\beq\label{E 33}
\f16\f{d}{dt}\|a\|_{L^6}^6+\int_{\R^3}\bigl(\gamma+(\gamma-\f16)a\bigr) a^6\, dx=
-\gamma\int_{\R^3}F a^5\, dx-(\gamma-\f16)\int_{\R^3}F a^6\, dx.
\eeq
Observing that $a=\r^\gamma-1\geq-1$, one has
\beno
\gamma+(\gamma-\f16)a\geq\f16,
\eeno
and hence
\beno
\f{d}{dt}\|a\|_{L^6}^6+\|a\|_{L^6}^6\leq 6\gamma
(1+\|a\|_{L^\infty})\|F\|_{L^6}\|a\|_{L^6}^5.
\eeno
Then we get
\beno
3\f{d}{dt}\|a\|_{L^6}^2+\|a\|_{L^6}^2\leq 6\gamma
(1+\|a\|_{L^\infty})\|F\|_{L^6}\|a\|_{L^6}.
\eeno
 Applying Young's inequality yields
\beno
6\f{d}{dt}\|a\|_{L^6}^2+\|a\|_{L^6}^2\leq C
(1+\|a\|_{L^\infty}^2)\|\na F\|_{L^2}^2,
\eeno
from which and \eqref{elliptic estimates}, we deduce \eqref{estimates for a in L 6}. This completes the  proof of Proposition \ref{prop for a in L 6}.
\end{proof}

\subsubsection{Estimate of $\r$} To close  the energy estimate, we  need the upper bound of $\r$.

\begin{lem}\label{lemma for rho}
Let $(\r,\vv u)$ be a smooth enough solution of \eqref{CNS} on $[0,T]$. Then for $t\leq T$, one has
\beq\label{estimates for rho}
\|\r\|_{L^\infty_t(L^\infty)}\leq M\exp\Bigl(\f{1}{2(M^\gamma-1)}\|F\|_{L^2_t(L^\infty)}^2\Bigr),
\eeq
where  $M$ is a positive constant so that  $M>\max\bigl(1,\|\r_0\|_{L^\infty}\bigr)$.
\end{lem}
\begin{proof} Let us first
introduce the flow map $X(t;\tau,y)$ associated with $\vv u$ as follows
\beq\label{trajectory}
\left\{\begin{aligned}
&\f{d}{dt}X(t;\tau,y)=\vv u(t,X(t;\tau,y)),\quad0\leq\tau\leq t,\\
&X(\tau;\tau,y)=y.
\end{aligned}\right.
\eeq
In particular,  we denote $X(t,y)\eqdefa X(t;0,y).$
Then we deduce from the transport equation of \eqref{CNS} that
\beq\label{E 37}
\f{d}{dt}\r(t,X(t,y))+(\r\div\vv u)(t,X(t,y))=0,
\eeq
which gives rise to
\beq\label{expression for rho}\begin{aligned}
\r(t,X(t,y))&=\r(t_1,X(t_1,y))\exp\Bigl(-\int_{t_1}^t\bigl(a+F\bigr)(\tau,X(\tau,y))d\tau\Bigr),
\end{aligned}\eeq
for any $t_1,t\in[0,T]$.

For fixed $y\in\R^3$, due to $M>\max\bigl(1,\|\r_0\|_{L^\infty}\bigr)$ and the continuity,  we may assume that there exists some $t_1>0$ so that
\beq \label{E37b}
\r(t,X(t,y))<M\quad \mbox{for any}\ \ t\in [0,t_1)\quad \text{and}\quad
\r(t_1,X(t_1,y))=M.
\eeq
Otherwise, there holds $\r(t,X(t,y))\leq M$ for any $t\in [0,T]$.

Now without loss of generality, we may further assume that $t_1<T$. Suppose that there exists $t_2>t_1$ so that
 $\r(t,X(t,y))\geq M\,\,\text{for any}\, t\in[t_1,t_2]$,
which ensures that
\beno
a(t,X(t,y))=\r^\gamma(t,X(t,y))-1\geq M^\gamma-1\quad\text{for any}\,\, t\in[t_1,t_2].
\eeno
 We thus deduce from \eqref{expression for rho} that for any $t\in[t_1,t_2]$
\beno\begin{aligned}
\r(t,X(t,y))
&\leq M\exp\Bigl(-(M^\gamma-1)(t-t_1)+(t-t_1)^{\f12}\|F\|_{L^2_t(L^\infty)}\Bigr)\\
&\leq M\exp\Bigl(-\f12(M^\gamma-1)(t-t_1)+\f{1}{2(M^\gamma-1)}\|F\|_{L^2_t(L^\infty)}^2\Bigr),
\end{aligned}\eeno
which together with \eqref{E37b} implies that for all $t\in[0,t_2]$,
\beq \label{E37a}
\r(t,X(t,y))\leq
 M\exp\Bigl(\f{1}{2(M^\gamma-1)}\|F\|_{L^2_t(L^\infty)}^2\Bigr).
 \eeq
If $t_2=T$, then we end the proof of the lemma. Otherwise  there holds $
\r(t_2,X(t_2,y))= M.$ This means that we may repeat the above argument to prove that \eqref{E37a} holds for $t\in [0,T].$
Since \eqref{E37a} holds for any $y\in\R^3$, we conclude the proof of \eqref{estimates for rho}. The lemma is proved.
\end{proof}

Now the bound of $\|\r\|_{L^\infty_t(L^\infty)}$ is reduced to the estimate of $\|F\|_{L^2_t(L^\infty)}$.

\begin{lem}\label{lemma for q}
Let $(\r,\vv u)$ be a smooth enough solution of \eqref{CNS} on $[0,T]$. Then for $t\leq T$, one has
\beq\label{estimates for q}
\|F\|_{L^2_t(L^\infty)}\lesssim\|\r\|_{L^\infty_t(L^\infty)}^{\f34}\|\sqrt\r\dot{\vv u}\|_{L^2_t(L^2)}^{\f12}\|\na\dot{\vv u}\|_{L^2_t(L^2)}^{\f12}.
\eeq
\end{lem}
\begin{proof}
Thanks to the elliptic estimate \eqref{elliptic estimates} and the following interpolation inequality
\beq\label{interpolation estimate for q}
\|F\|_{L^\infty}\lesssim\|F\|_{L^6}^{\f12}\|\na F\|_{L^6}^{\f12}\lesssim\|\na F\|_{L^2}^{\f12}\|\na F\|_{L^6}^{\f12},
\eeq
we get
\beno
\|F\|_{L^\infty}\lesssim\|\r\dot{\vv u}\|_{L^2}^{\f12}\|\r\dot{\vv u}\|_{L^6}^{\f12}\lesssim\|\sqrt\r\|_{L^\infty}^{\f32}\|\sqrt\r\dot{\vv u}\|_{L^2}^{\f12}\|\na\dot{\vv u}\|_{L^2}^{\f12}.
\eeno
Taking $L^2$ norm w.r.t. time $t$ on $[0,T]$ leads to
 \eqref{estimates for q}. This completes the proof of the lemma.
\end{proof}

\subsubsection{Estimate of $\na a$} To propagate the higher regularity of the solution, we need the esitmate of $\na a$.
\begin{prop}\label{prop for na a a}
Let $(\r,\vv u)$ be a smooth enough solution of \eqref{CNS} on $[0,T]$. Then for $t\leq T$, one has
\beq\label{estimate of na a 1}\begin{aligned}
\f{1}{p}\f{d}{dt}\|\na a(t)\|_{L^p}^p+\gamma\|\na a\|_{L^p}^p
\leq& \gamma\bigl(2\|\na\vv u\|_{L^\infty}+\|a\|_{L^\infty}\bigr)\|\na a\|_{L^p}^p\\
&+\gamma\bigl(1+\|a\|_{L^\infty}\bigr)\|\na F\|_{L^p}\|\na a\|_{L^p}^{p-1},\quad\text{for}\ p\geq 2.
\end{aligned}\eeq
\end{prop}
\begin{proof}
By substituting  $\div\vv u=F+a$ in  the $a$-equation \eqref{E 11},  we write
\beno
\p_ta+\vv u\cdot\na a+\gamma a+\gamma F+\gamma a\div\vv u=0.
\eeno
Applying $\na$ to the above equation gives
\beq\label{E 72}
\p_t(\na a)+\vv u\cdot\na(\na a)+\gamma\na a=-\gamma\na F-\gamma\na a\div\vv u-\gamma a\na\div\vv u-(\na\vv u)^T\na a.
\eeq
By taking $L^2$-inner product of \eqref{E 72} with $|\na a|^{p-2}\na a$ and using  integrating by parts, we find
\beno\begin{aligned}
\f{1}{p}\f{d}{dt}\|\na a(t)\|_{L^p}^p+\gamma\|\na a\|_{L^p}^p=&-\gamma\bigl(\na F\,\big|\,|\na a|^{p-2}\na a\bigr)
-(\gamma-\f{1}{p})\int_{\R^3}\div\vv u|\na a|^pdx\\
&-\gamma\bigl(a\na\div\vv u\,\big|\,|\na a|^{p-2}\na a\bigr)-\bigl(\na\vv u\,\big|\,|\na a|^{p-2}\na a\otimes\na a\bigr)\\
\leq&\gamma\|\na F\|_{L^p}\|\na a\|_{L^p}^{p-1}+(\gamma-\f{1}{p})\|\div\vv u\|_{L^\infty}\|\na a\|_{L^p}^p\\
&+\gamma\|a\|_{L^\infty}\|\na\div\vv u\|_{L^p}\|\na a\|_{L^p}^{p-1}+\|\na\vv u\|_{L^\infty}\|\na a\|_{L^p}^p.
\end{aligned}\eeno
Using once again the fact that $\div\vv u=F+a$, we obtain \eqref{estimate of na a 1}. The lemma is proved.
\end{proof}

\subsection{Proof of Theorem \ref{global existence thm}}
 Based on the   {\it a priori }  estimates for smooth enough solutions of   \eqref{CNS} obtained in the previous subsections, we are in a position
 to complete the proof of Theorem \ref{global existence thm}. We first observe from \eqref{def of H rho} that
\beq\label{E 6}
h(\r)=\f{\gamma}{2}\int_0^1[\th\r+(1-\th)]^{\gamma-2}\,d\th(\r-1)^2\,\,\text{for}\,\, \gamma\geq1,
\eeq
which implies
\beno
H(\r)\geq\left\{\begin{aligned}
&\f{\gamma}{2(\gamma-1)}\|\r-1\|_{L^2}^2,\,\,\text{for}\,\, \gamma\geq 2,\\
&\f{\gamma}{2}(\|\r\|_{L^\infty}+1)^{\gamma-2}\|\r-1\|_{L^2}^2,\,\,\text{for}\,\, 1\leq\gamma<2.
\end{aligned}\right.
\eeno

 \begin{lem}\label{prop for a in L 3}
 {\sl Let $H(\r)$ be given by \eqref{def of H rho}.  Then for  $R>1$ and $p\in[2,6)$, we have
\beq\label{estimates for a in L 3}
\|a\|_{L^p}^p\leq4\gamma R^{p-1} H(\r)+R^{p-6}\|a\|_{L^6}^6,
\eeq}
\end{lem}
\begin{proof} We first observe that
\beq\label{E 35}
|a|^p=|a|^p 1_{|a|\leq R}+|a|^p1_{|a|>R}\leq R^{p-2}|a|^2 1_{|a|\leq R}
+R^{p-6}|a|^6.
\eeq
Due to $a=\r^\gamma-1=\gamma\int_0^1[\th\r+(1-\th)]^{\gamma-1}\,d\th\cdot(\r-1)$, we get,
by applying  \eqref{E 6} and the H\"older's inequality,
 that
\beno
\f{a^2}{h(\r)}=2\gamma\f{\bigl|\int_0^1[\th\r+(1-\th)]^{\gamma-1}d\th\bigr|^2}{\int_0^1[\th\r+(1-\th)]^{\gamma-2}d\th}
\leq 2\gamma\int_0^1[\th\r+(1-\th)]^\gamma d\th
\leq2\gamma\max\{\r^\gamma,1\},
\eeno
which implies that if $R>1$,
\beno
|a|^2\leq 4\gamma R h(\r) \quad\text{for }\,\, |a|\leq R.
\eeno
By substituting the above estimate into \eqref{E 35}, we obtain \eqref{estimates for a in L 3}.
\end{proof}

\begin{lem}\label{S2lem5}
{\sl Let $\vv u$ be a smooth vector field and $q\in (3,\infty).$ We denote $\mathfrak{h}\eqdefa \|\curl\vv u\|_{L^\infty}+\|\div\vv u\|_{L^\infty}.$ Then one has
\beq \label{S2eq5}
\|\na\vv u\|_{L^\infty}\leq C\Bigl(\|\na \vv u\|_{L^2}+\f{q}{q-3}\mathfrak{h}\log_2\Bigl(2+\f{\|\na^2 \vv u\|_{L^q}}{\mathfrak{h}}\Bigr)\Bigr).
\eeq}
\end{lem}

\begin{proof}
In view of \eqref{S2eq3}, we write
\beno
\na\vv u=\na(-\D)^{-1}\curl\curl\vv u-\na(-\D)^{-1}\na\div\vv u.
\eeno
Let $\dD_j$ be the dyadic operator defined in \eqref{S1eq18}. Then we have
\beq \label{S2eq4}
\begin{aligned}
\|\dD_j\na \vv u\|_{L^\infty}\lesssim & \|\dD_j\curl\vv u\|_{L^\infty}+\|\dD_j\div\vv u\|_{L^\infty}\\
\lesssim & \|\curl\vv u\|_{L^\infty}+\|\div\vv u\|_{L^\infty}\eqdefa \mathfrak{h}.
\end{aligned}
\eeq

Thanks to \eqref{S2eq4} and Bernstein Lemma (see \cite{bcd}), for any integer $N,$ we write
\begin{align*}
\|\na\vv u\|_{L^\infty}\leq &\sum_{j\leq 0}\|\dD_j\na \vv u\|_{L^\infty}+\sum_{j=1}^N\|\dD_j\na \vv u\|_{L^\infty}
+\sum_{j>N}\|\dD_j\na \vv u\|_{L^\infty}\\
\lesssim &\sum_{j\leq 0}2^{\f32j}\|\dD_j\na \vv u\|_{L^2}+N\mathfrak{h}+\sum_{j>N}2^{-\left(1-\f3q\right)j}\|\dD_j\na^2 \vv u\|_{L^q}\\
\lesssim & \|\na \vv u\|_{L^2}+N\mathfrak{h}+2^{-\left(1-\f3q\right)N}\|\na^2 \vv u\|_{L^q}.
\end{align*}
Let us take $N$ so that
\beno
N\sim \f{q}{q-3}\log_2\Bigl(2+\f{\|\na^2 \vv u\|_{L^q}}{\mathfrak{h}}\Bigr)
\eeno
in the above inequality, we achieve \eqref{S2eq5}. It completes the proof of the lemma.
\end{proof}

Now we are in a position to complete the proof of   Theorem  \ref{global existence thm}.

\begin{proof}[Proof of Theorem \ref{global existence thm}]
Firstly, the assumptions of Theorem \ref{global existence thm} show that
\beno
\r_0(x)>0, \ \na\r_0\in L^q(\R^3),\ \vv u_0\in H^2(\R^3)\hookrightarrow W^{2-\f{2}{q},q}(\R^3),
\eeno
where
\beno
W^{2-\f2q,q}(\R^3)\eqdefa \left\{\ f\in W^{1,q}(\R^3)\,|\, \Bigl(\int_{\R^3}\int_{\R^3}\f{|\na f(x)-\na f(y)|^q}{|x-y|^{3+q-2}}dydx\Bigr)^{\f1q}<\infty\ \right\}.
\eeno
By the classical theory of compressible Navier-Stokes equations (see \cite{Solo76} for instance), the system  \eqref{CNS}  has a unique local strong solution $(\r,\vv u)$ on $[0,T^\ast)$ such that for any $T<T^\ast$
\beno
\r\in L^\infty\bigl([0,T];L^\infty\cap\dot{W}^{1,q}(\R^3)\bigr),\quad
\vv u\in W^{2,1}_q([0,T]\times\R^3).
\eeno

We assume that $T^*$ is the lifespan of this  solution to \eqref{CNS}. We shall prove below by a continuous argument that $T^\ast=\infty$
as long as $\e$ is sufficiently small in \eqref{initial ansatz}.

Due to \eqref{initial ansatz for density}, let $\mathcal{C}=4\delta^{\alpha}(\|\r_0\|_{L^\infty}+1)^\gamma\in[4,\,4\times 3^\gamma]$. We define $T^\star$ to be the maximal time $T\in (0,T^\ast)$ so that there hold
\beq\label{ansatz 1}
\sup_{t\in[0,T^\star]}\|a(t)\|_{L^\infty}\leq 2\mathcal{C}\delta^{-\al},
\eeq
and
\beq\label{ansatz 2}\begin{aligned}
&\delta^{\al-3}\Bigl\{\sup_{t\in[0,T^\star]}\bigl(\|(\sqrt{\r}\vv u, \curl\vv u, F)(t)\|_{L^2}^2+ H(\r(t))\bigr)+\|(\na\vv u,\sqrt\r\dot{\vv u})\|_{L^2_{T^\star}(L^2)}^2\Bigr\}\leq 2\mathcal{M}\e^2,\\
&
\delta^{2\al-1}\Bigl\{\sup_{t\in[0,T^\star]}\bigl(\|(\sqrt\r\dot{\vv u},\div\vv u)(t)\|_{L^2}^2
+\|a(t)\|_{L^6}^2\bigr)+\|\na\dot{\vv u}\|_{L^2_{T^\star}(L^2)}^2+\|a\|_{L^2_{T^\star}(L^6)}^2\Bigr\}\leq 2\mathcal{M}\e^2,
\end{aligned}\eeq
where $\e>0$ is a sufficiently small constant independent of $\delta,$ while $\mathcal{M}>1$ is an universal constant which will be determined later on.

In what follows, $C$ is always referred to the universal constant. We divide the proof of Theorem \ref{global existence thm} into the following two steps:\smallskip

\no{\bf  Step 1. $T^\star=\infty$ if $\e$ is small enough}\smallskip

 As a convention in this step, we shall always assume that $t\leq T^\star.$

 \no{\Large $\bullet$} {\bf  The estimate of $\sup_{t\in[0,T^\star]}\|a(t)\|_{L^3}$ and $\sup_{t\in[0,T^\star]}\|\r(t)\|_{L^\infty}$.}

 Thanks to \eqref{ansatz 2}, we deduce from \eqref{estimates for a in L 3} that for $R>1$,
\beno\begin{aligned}
\|a(t)\|_{L^3}^3&\leq 4\gamma R^2 H(\r)(t)+R^{-3}\|a(t)\|_{L^6}^6\\
&\leq C\mathcal{M}\e^2R^2\delta^{3-\al}
+C\mathcal{M}^3\e^6R^{-3}\delta^{3(1-2\al)}.
\end{aligned}\eeno
Taking $R=2\delta^{-\al}$ in the above inequality leads to
\beq\label{E 40}
\sup_{t\in[0,T^\star]}\|a(t)\|_{L^3}\leq C\mathcal{M}\e^{\f23}\delta^{1-\al}.
\eeq

While it follows from \eqref{ansatz 1} that
\beq\label{E 44}
\sup_{t\in[0,T^\star]} \|\r(t)\|_{L^\infty}\leq \bigl(2\mathcal{C}\delta^{-\al}+1\bigr)^{\f{1}{\gamma}}\leq C\mathcal{C}^{\f1\gamma}\delta^{-\f{\al}{\gamma}}.
\eeq

 \no{\Large $\bullet$} {\bf The energy estimates of $\left(\sqrt{\r}\vv u, \curl\vv u, F, \sqrt{\r}\dot{\vv u}, \div\vv u\right)$.}

 Let us denote
\beq \label{S2eq10}
\begin{split}
E_1(t)\eqdefa&\gamma\|\sqrt\r\vv u(t)\|_{L^2}^2+\gamma H(\r(t))+\f12\|\curl\vv u(t)\|_{L^2}^2+\f12\|F(t)\|_{L^2}^2,\\
E_2(t)\eqdefa&\|\sqrt\r\dot{\vv u}(t)\|_{L^2}^2+\gamma\|\div\vv u(t)\|_{L^2}^2+12\|a(t)\|_{L^6}^2.
\end{split}
\eeq
Then we get, by summing up  \eqref{basic energy equality}$\times2\ga$ with \eqref{estimates for curl u and q}, that
\beq\label{E 39}\begin{aligned}
\f{d}{dt}E_1(t)+2\gamma\|\curl\vv u\|_{L^2}^2+\gamma\|\div\vv u\|_{L^2}^2&+\|\sqrt\r\dot{\vv u}\|_{L^2}^2
\lesssim\|\r\|_{L^\infty}^{\f12} \mathfrak{A}(t)\|\sqrt\r\dot{\vv u}\|_{L^2}\|\na\vv u\|_{L^2}.
\end{aligned}\eeq
where
\beno
\mathfrak{A}(t)\eqdefa\bigl(\|\curl\vv u(t)\|_{L^2}+\|F(t)\|_{L^2}\bigr)^{\f12}\|\sqrt{\r(t)}\|_{L^\infty}^{\f12}\|\sqrt\r\dot{\vv u}(t)\|_{L^2}^{\f12}+\|a(t)\|_{L^3}.
\eeno

It follows from  \eqref{ansatz 2}, \eqref{E 40} and \eqref{E 44} that
\beq\label{F 1}\begin{aligned}
\mathfrak{A}(t)&\le C(\mathcal{C}^{\f1{4\gamma}}\mathcal{M}^{\f12}\e\delta^{1-\f{\al}{4}(3+\f{1}{\gamma})}+\mathcal{M}\e^{\f23}\delta^{1-\al}),
\end{aligned}\eeq
 from which and \eqref{E 44}, we deduce from \eqref{E 39} that
\beno\begin{aligned}
&\f{d}{dt}E_1(t)+\gamma\|\na\vv u\|_{L^2}^2
+\|\sqrt\r\dot{\vv u}\|_{L^2}^2
\le C \bigl(\mathcal{C}^{\f3{4\gamma}}\mathcal{M}^{\f12}\e\delta^{1-\f{3\al}{4}(1+\f{1}{\gamma})}+\mathcal{C}^{\f1{2\gamma}}\mathcal{M}\e^{\f23}\delta^{1-(1+\f{1}{2\gamma})\al}\bigr)\|\sqrt\r\dot{\vv u}\|_{L^2}\|\na\vv u\|_{L^2}.
\end{aligned}\eeno
 By the assumption that $\al\leq\f{2\gamma}{1+{2\gamma}}\leq \f{4}{3(1+\f{1}{\gamma})}$ and using  Young's inequality, we find that for $\e>0$ sufficiently small,
\beq\label{E 45}
\f{d}{dt}E_1(t)+\f\gamma2\|\na\vv u\|_{L^2}^2
+\f12\|\sqrt\r\dot{\vv u}\|_{L^2}^2\leq 0,
\eeq
which along with \eqref{initial ansatz} ensure
\beno\begin{aligned}
&\quad E_1(t)+\f\gamma2\|\na\vv u\|_{L^2_t(L^2)}^2
+\f12\|\sqrt\r\dot{\vv u}\|_{L^2_t(L^2)}^2\leq E_1(0)\leq\gamma\e^2\delta^{3-\al}\quad\mbox{for any}\ \ t\leq T^\star.
\end{aligned}\eeno
Then in view of \eqref{S2eq10}, there exists $C_1>1$ such that
\beq\label{E 47}\begin{aligned}
\delta^{\al-3}\Bigl\{\sup_{t\in [0,T^\star]}\bigl(\|(\sqrt{\r}\vv u,\curl\vv u, F)(t)\|_{L^2}^2+ H(\r(t))\bigr)
+\|(\na\vv u,\sqrt\r\dot{\vv u})\|_{L^2_{T^\star}(L^2)}^2\Bigr\}\leq C_1\e^2.
\end{aligned}\eeq

On the other hand, in view of \eqref{ansatz 1}, \eqref{ansatz 2}, \eqref{E 44} and \eqref{F 1}, we deduce from \eqref{estimates for sqrt rho dot u} and \eqref{estimates for a in L 6} that
\beno\begin{aligned}
\f12\f{d}{dt}E_2(t)+\|\na\dot{\vv u}\|_{L^2}^2+\|a\|_{L^6}^2
\le  & \mathfrak{A}(t)\bigl(\|\sqrt\r\|_{L^\infty}\|\sqrt\r\dot{\vv u}\|_{L^2}+\|a\|_{L^6}\bigr)\bigl(\|\na\vv u\|_{L^2}+\|\na\dot{\vv u}\|_{L^2}\bigr)\\
&+
(1+\|a\|_{L^\infty}^2)\|\r\|_{L^\infty}\|\sqrt\r\dot{\vv u}\|_{L^2}^2\\
\le &C \bigl(\mathcal{C}^{\f3{4\gamma}}\mathcal{M}^{\f12}\e\delta^{1-\f{3\al}{4}(1+\f{1}{\gamma})}+\mathcal{C}^{\f1{2\gamma}}\mathcal{M}\e^{\f23}\delta^{1-(1+\f{1}{2\gamma})\al}\bigr)\bigl(\|\sqrt\r\dot{\vv u}\|_{L^2}+\|a\|_{L^6}\bigr)\\
&\times\bigl(\|\na\vv u\|_{L^2}+\|\na\dot{\vv u}\|_{L^2}\bigr)+C\mathcal{C}^{2+\f1{\gamma}}\delta^{-2\al-\f{\al}{\gamma}}\|\sqrt\r\dot{\vv u}\|_{L^2}^2,
\end{aligned}\eeno
which along with the assumption that $\al\leq\f{2\gamma}{1+2\gamma}\leq \f{4}{3(1+\f{1}{\gamma})}$ implies
\beno
 \f12\f{d}{dt}E_2(t)+\|\na\dot{\vv u}\|_{L^2}^2+\|a\|_{L^6}^2&\le& C \bigl(\mathcal{C}^{\f3{4\gamma}}\mathcal{M}^{\f12}\e+\mathcal{C}^{\f1{2\gamma}}\mathcal{M}\e^{\f23}\bigr)\bigl(\|\sqrt\r\dot{\vv u}\|_{L^2}+\|a\|_{L^6}\bigr)\\
&&\times\bigl(\|\na\vv u\|_{L^2}+\|\na\dot{\vv u}\|_{L^2}\bigr)
+C\mathcal{C}^{2+\f1{\gamma}}\delta^{-2\al(1+\f{1}{2\gamma})}\|\sqrt\r\dot{\vv u}\|_{L^2}^2.
\eeno
So that for
 $\e$ sufficiently small, we obtain
\beq\label{E 48}\begin{aligned}
\f{d}{dt}E_2(t)&+\|\na\dot{\vv u}\|_{L^2}^2+\|a\|_{L^6}^2\le C\mathcal{C}^{2+\f1{\gamma}}\delta^{-2\al(1+\f{1}{2\gamma})}\bigl(\|\sqrt\r\dot{\vv u}\|_{L^2}^2+\e^{\f23}\|\na\vv u\|_{L^2}^2\bigr),
\end{aligned}\eeq
which along with \eqref{E 47}  ensures
\beno\begin{aligned}
E_2(t)+\|\na\dot{\vv u}\|_{L^2_t(L^2)}^2+\|a\|_{L^2_t(L^6)}^2&\leq E_2(0)+C\mathcal{C}^{2+\f1{\gamma}}\e^2\delta^{3-3\al-\f{\al}{\gamma}} \quad\mbox{for any}\ \ t\leq T^\star.
\end{aligned}\eeno
Due to the assumption that $\al\leq\f{2\gamma}{1+2\gamma}\leq \f{4}{3(1+\f{1}{\gamma})}$, \eqref{S2eq10} and \eqref{initial ansatz}, there exists
$C_2>1$ such that
\beq\label{E 49}\begin{aligned}
\delta^{2\al-1}\Bigl\{\sup_{t\in [0,T^\star]}\bigl(
\|(\sqrt\r\dot{\vv u},\div\vv u)(t)\|_{L^2}^2
+\|a(t)\|_{L^6}^2\bigr)+\|\na\dot{\vv u}\|_{L^2_{T^\star}(L^2)}^2+\|a\|_{L^2_{T^\star}(L^6)}^2\Bigr\}\leq C_2\e^2.
\end{aligned}
\eeq

 \no{\Large $\bullet$} {\bf  The continuity argument.}

 We shall prove that $T^\star=\infty$ as long as $\e$ is small enough. Indeed
by virtue of \eqref{E 44}, \eqref{E 47} and \eqref{E 49}, we deduce from \eqref{estimates for q} that
\beq\label{E 41}\begin{aligned}
\|F\|_{L^2_{T^\star}(L^\infty)}&\leq C\|\r(t)\|_{L^\infty_{T^\star}(L^\infty)}^{\f34}\|\sqrt\r\dot{\vv u}\|_{L^2_{T^\star}(L^2)}^{\f12}\|\na\dot{\vv u}\|_{L^2_{T^\star}(L^2)}^{\f12}\\
&\leq C\mathcal{C}^{\f3{4\gamma}}\mathcal{M}^{\f12}\e\delta^\beta,  \with \beta\eqdefa 1-\f{3\al}{4}\Bigl(1+\f{1}{\gamma}\Bigr).
\end{aligned}\eeq

Due to \eqref{initial ansatz for density}, we have
\beq\label{E 50}
M(\r_0)\eqdefa\|\r_0\|_{L^\infty}+1=\Bigl(\f{\mathcal{C}}{4}\Bigr)^{\f1\gamma}\delta^{-\f{\al}{\gamma}}\geq 2.
\eeq
Thanks to \eqref{E 41} and \eqref{E 50}, we deduce from \eqref{estimates for rho} that
\beq\label{E 42}\begin{aligned}
\sup_{t\in[0,T^\star]}\|\r(t)\|_{L^\infty}
&\leq M(\r_0)\exp\Bigl(\f{1}{2(M(\r_0)^\gamma-1)}\|F\|_{L^2_{T^\star}(L^\infty)}^2\Bigr)\\
&\leq \Bigl(\f{\mathcal{C}}{4}\Bigr)^{\f1\gamma}\delta^{-\f{\al}{\gamma}}\exp\Bigl(C\mathcal{M}\e^2\delta^{2\left(1-\f{3\al}{4}(1+\f{1}{\gamma})\right)+\al}\Bigr).
\end{aligned}\eeq

Due to the assumption that $\al\leq\f{2\gamma}{1+2\gamma}\leq\f{4}{3(1+\f{1}{\gamma})}$, by taking $\e>0$ so small  that
\beno
C\mathcal{M}\e^2\delta^{2[1-\f{3\al}{4}(1+\f{1}{\gamma})]+\al}\leq C\mathcal{M}\e^2\delta^\al\leq\ln 2^{\f{1}{\gamma}},
\eeno
we
deduce from \eqref{E 42} that
\beno
\sup_{t\in[0,T^\star]}\|\r(t)\|_{L^\infty}\leq\Bigl(\f{\mathcal{C}}{2}\Bigr)^{\f1\gamma}\delta^{-\f{\al}{\gamma}},
\eeno
which implies that
\beq\label{E 43}
\sup_{t\in[0,T^\star]}\|a(t)\|_{L^\infty}\leq \mathcal{C}\delta^{-\al}.
\eeq

Let us take  $\mathcal{M}\eqdefa \max\{C_1,C_2\}.$ Then \eqref{E 47}, \eqref{E 49} and \eqref{E 43} contradicts with
the definition of $T^\star$ defined at the beginning of the proof of  Theorem \ref{global existence thm}. This in turn shows that as long as $\e$ is small enough,
 $T^\star=T^\ast=\infty,$ and
there hold  the upper bound of $\rho$ in \eqref{estimate for a} and \eqref{total enery estimate}. The lower bound of $\rho$ in \eqref{estimate for a} will be proved in Lemma \ref{lower bound lemma}. Furthermore,  \eqref{energy in H 2} follows from  \eqref{elliptic estimates},
 \eqref{estimate for a} and \eqref{total enery estimate}.

\smallskip

\no{\bf Step 2. The propagation of regularities for $(\r,\vv u)\in\dot{W}^{1,q}(\R^3)\times\dot{W}^{2,1}_q([0,T]\times\R^3).$}
\smallskip

As a convention in this step, we shall use the notation $C_\delta$ to denote the constant that depends on $\delta^{-1}$ increasingly and $C_{\delta,T}$ to denote the constant that depends on $\delta^{-1},\,T$ increasingly.

\no{\Large $\bullet$} {\bf  $L^q$ estimate of $\na a.$}

Due to $\div\vv u=F+a,$ using Young's inequality and \eqref{estimate for a}, we deduce from \eqref{estimate of na a 1} that
\beno
\f1q\f{d}{dt}\|\na a(t)\|_{L^q}^q+\f\ga2\|\na a\|_{L^q}^q
\leq\gamma\bigl(\|F\|_{L^\infty}+3\|\na\vv u\|_{L^\infty}\bigr)\|\na a\|_{L^q}^q+C_\delta\|\na F\|_{L^q}^q.
\eeno
Applying Gronwall's inequality gives rise to
\beq \label{S2eq1g}
\begin{aligned}
&\|\na a(t)\|_{L^q}^q+\f{q\gamma}{2}\|\na a\|_{L^q_t(L^q)}^q\\
&\leq \bigl(\|\na a_0\|_{L^q}^q+C_\delta\int_0^t\|\na F(t')\|_{L^q}^q\,dt'\bigr)\exp\Bigl(3q\gamma\int_0^t\bigl(\|F(t')\|_{L^\infty}+\|\na\vv u(t')\|_{L^\infty}\bigr)\,dt'\Bigr).
\end{aligned}
\eeq
Yet it follows from \eqref{estimate for a} and Lemma \ref{lemma for elliptic system} that
\begin{align*}
\|\na F(t)\|_{L^q}^q\leq C_\delta\|\dot{\vv u}(t)\|_{L^q}^q\leq C_\delta\|\dot{\vv u}(t)\|_{L^2}^{3-\f{q}2}\|\na \dot{\vv u}(t)\|_{L^2}^{3\left(\f{q}2-1\right)}.
\end{align*}
Notice that for $q\in \bigl(3,\f{10}3\bigr],$ we have $3\left(\f{q}2-1\right)\leq 2,$ so that we deduce from \eqref{total enery estimate} that
\beq \label{S2eq1a}
\int_0^T\bigl(\|\na F(t)\|_{L^q}^q+\|\dot{\vv u}(t)\|_{L^q}^q\bigr)\,dt\leq C_{\delta,T},
\eeq
from which, \eqref{estimates for q} and \eqref{S2eq1g}, we infer
\beq \label{S2eq2a}
\|\na a(t)\|_{L^q}^q\leq C_{\delta,T}\exp\Bigl(3q\gamma\int_0^t\|\na\vv u(t')\|_{L^\infty}\,dt'\Bigr)\quad\mbox{for any} \ \ t\leq T<\infty.
\eeq

\no{\Large $\bullet$} {\bf  The estimates of $\|\na a\|_{L^\infty_T(L^q)}$ and $\|(\p_t\vv u,\na^2\vv u)\|_{L^q([0,T]\times\R^3)}.$}

Observing that for any $b>0,$ $\tau\log_2\bigl(2+\f{b}\tau\bigr)$ is an increasing function of $\tau,$  so that we deduce
from \eqref{total enery estimate}, \eqref{S2eq5} and \eqref{S2eq2a}  that
\beq\label{S2eq6}
\|\na a(t)\|_{L^q}^q\leq C_{\delta,T}\exp\Bigl(\f{Cq}{q-3}\int_0^t\hbar(t')\log_2\Bigl(2+\f{\|\na^2 \vv u(t')\|_{L^q}}{\hbar(t')}\Bigr)\,dt'\Bigr)
\quad\mbox{for any} \ \ t\leq T<\infty,
\eeq
where $\hbar(t)\eqdefa\max\bigl(1,\mathfrak{h}(t)\bigr)$ with  $\mathfrak{h}(t)\eqdefa \|\curl\vv u(t)\|_{L^\infty}+\|\div\vv u(t)\|_{L^\infty}.$

On the other hand, we write the momentum equation of \eqref{CNS} as
\beno
\D\vv u=\r\dot{\vv u}+\na a,
\eeno
from which, \eqref{estimate for a} and \eqref{S2eq6}, we infer
\beq\label{S2eq6a}
\begin{aligned}
\|\na^2\vv u(t)\|_{L^q}\leq &C_\delta\bigl(\|\dot{\vv u}(t)\|_{L^q}+\|\na a(t)\|_{L^q}\bigr)\\
\leq &C_{\delta,T}\Bigl(\|\dot{\vv u}(t)\|_{L^q}+\exp\Bigl(\f{C}{q-3}\int_0^t\hbar(t')\log_2\Bigl(2+\f{\|\na^2 \vv u(t')\|_{L^q}}{\hbar(t')}\Bigr)\,dt'\Bigr)\Bigr).
\end{aligned} \eeq
Therefore, due to $\log_2(a+b)\leq \log_2(2a)+\log_2(2b)$ for $a,b\geq 1$ and $\hbar(t)\geq 1,$ we have
\begin{align*}
\log_2\Bigl(2+\f{\|\na^2 \vv u(t)\|_{L^q}}{\hbar(t)}\Bigr)\leq &\log_2 (2C_{\delta,T})+\log_2(2\|\dot{\vv u}(t)\|_{L^q})\\
&+\underbrace{\f{C}{q-3}\int_0^t\hbar(t')\log_2\Bigl(2+\f{\|\na^2 \vv u(t')\|_{L^q}}{\hbar(t')}\Bigr)\,dt'}_{G(t)},
\end{align*}
which implies
\beno
\f{d}{dt}G(t)\leq \f{C}{q-3}\hbar(t)G(t)+\f{C}{q-3}\hbar(t)\Bigl(\log_2 (2C_{\delta,T})+\log_2(2\|\dot{\vv u}(t)\|_{L^q})\Bigr).
\eeno
Applying Gronwall's inequality to the above inequality, we deduce the bound of $G(t)$ and then
\beq \label{S2eq8}
\begin{aligned}
\log_2\Bigl(2&+\f{\|\na^2 \vv u(t)\|_{L^q}}{\hbar(t)}\Bigr)\leq C_{\delta,T}+\log_2(2\|\dot{\vv u}(t)\|_{L^q})\\
&+\f{C}{q-3}\int_0^t\hbar(t')\Bigl(C_{\delta,T}+\log_2(2\|\dot{\vv u}(t)\|_{L^q})\Bigr)\,dt'\cdot\exp\Bigl(\f{C}{q-3}\int_0^t\hbar(t')\,dt'\Bigr).
\end{aligned} \eeq
Yet it follows from \eqref{estimate for a} and \eqref{elliptic estimates} that for any $q>3$
\begin{align*}
\hbar(t)\leq &1+ \|\curl\vv u(t)\|_{L^\infty}+\|F(t)\|_{L^\infty}+\|a(t)\|_{L^\infty}\\
\leq &C\bigl(1+\|\curl\vv u(t)\|_{L^2}+\|F(t)\|_{L^2}+\|\na\curl\vv u(t)\|_{L^q}+\|\na F(t)\|_{L^q}+\|a(t)\|_{L^\infty}\bigr)\\
\leq &C_\delta\bigl(1+\|\curl\vv u(t)\|_{L^2}+\|F(t)\|_{L^2}+\|\dot{\vv u}(t)\|_{L^q}+\|a(t)\|_{L^\infty}\bigr),
\end{align*}
which together with \eqref{total enery estimate} and \eqref{S2eq1a} ensures that for $q\in (3, 10/3]$
\begin{align*}
\int_0^T\hbar(t)\bigl(1+\|\dot{\vv u}(t)\|_{L^q}\bigr)\,dt\leq C_{\delta,T}\bigl(&1+\|(\curl\vv u, F)\|_{L^\infty_T(L^2)}^2\\
&+\|a\|_{L^\infty_T(L^\infty)}^2
+\|\dot{\vv u}(t)\|_{L^2_T(L^q)}^2\bigr)\leq C_{\delta,T}.
\end{align*}
So that we deduce from \eqref{total enery estimate} and \eqref{S2eq8} that
\beq\label{S2eq8a}
\log_2\Bigl(2+\f{\|\na^2 \vv u(t)\|_{L^q}}{\hbar(t)}\Bigr)\leq C_{\delta,T}+\log_2(2\|\dot{\vv u}(t)\|_{L^q})\quad\mbox{for any} \ \ t\leq T<\infty,
\eeq
which implies
\beno
\|\na^2 \vv u(t)\|_{L^q}\leq C_{\delta,T}\hbar(t)\|\dot{\vv u}(t)\|_{L^q}\quad\mbox{for any} \ \ t\leq T<\infty.
\eeno
Then we deduce from \eqref{S2eq6} and \eqref{S2eq8a} that
\beq\label{S2eq67}
\|\na a\|_{L^\infty_T(L^q)}\leq C_{\delta,T}\exp\Bigl(\f{C}{q-3}\int_0^t\hbar(t')\log_2\Bigl(2+C_{\delta,T}\|\dot{\vv u}(t)\|_{L^q}\Bigr)\,dt'\Bigr)
\leq C_{\delta,T},
\eeq
which together with \eqref{S2eq6a} and \eqref{S2eq1a} implies
\beq \label{S2eq68}
\|\na^2 \vv u\|_{L^q_T(L^q)}\leq C_{\delta,T}\bigl(\|\dot{\vv u}\|_{L^q_T(L^q)}+\|\na a\|_{L^\infty_T(L^q)}\bigr)\leq C_{\delta,T}.
\eeq

Notice that for $r$ satisfying $\f1r=\f1q-\f16,$ one has
\begin{align*}
\|\na \vv u\|_{L^r}\lesssim &\|\na \vv u\|_{L^2}^\theta\|\na^2\vv u\|_{L^q}^{1-\theta}\with \f1r=\f\theta2+(\f1q-\f13)(1-\theta)\\
\lesssim &\bigl(\|\div \vv u\|_{L^2}+\|\curl \vv u\|_{L^2}\bigr)^\theta\|\na^2\vv u\|_{L^q}^{1-\theta},
\end{align*}
from which, \eqref{total enery estimate} and \eqref{S2eq68}, we infer
\beno
\|\na \vv u\|_{L^q_T(L^r)}\leq C_{\delta,T} \bigl(\|\div \vv u\|_{L^\infty_T(L^2)}+\|\curl \vv u\|_{L^\infty_T(L^2)}\bigr)^\theta\|\na^2\vv u\|_{L^q_T(L^q)}^{1-\theta}\leq C_{\delta,T}.
\eeno
Hence we obtain
\begin{align*}
\|\p_t\vv u\|_{L^q_T(L^q)}\leq &\|\dot{\vv u}\|_{L^q_T(L^q)}+\|\vv u\|_{L^\infty_T(L^6)}\|\na\vv u\|_{L^q_T(L^r)}\\
\leq &\|\dot{\vv u}\|_{L^q_T(L^q)}+C\bigl(\|\div \vv u\|_{L^\infty_T(L^2)}+\|\curl \vv u\|_{L^\infty_T(L^2)}\bigr)\|\na\vv u\|_{L^q_T(L^r)}\leq C_{\delta,T},
\end{align*}
which together with \eqref{S2eq68} ensures that for any $T<\infty$
\beq \label{S2eq69}
\|(\p_t\vv u,\,\na^2 \vv u)\|_{L^q_T(L^q)}\leq C_{\delta,T}.
\eeq

Due to \eqref{estimate for a},  we have
$\|\na\r\|_{L^q}\leq C\|\na a\|_{L^q},$
which along with \eqref{S2eq67} and \eqref{S2eq69} implies
\beno
\|\na\r\|_{L^\infty_T(L^q)}+\|(\p_t\vv u,\,\na^2 \vv u)\|_{L^q_T(L^q)}\leq C_{\delta,T},\quad\text{for any}\ T<\infty.
\eeno
This completes the proof of Theorem \ref{global existence thm}.
\end{proof}

\setcounter{equation}{0}
\section{Dynamic behavior  of the density function and the decay of the solutions}

The aim of this section is to investigate the collapse phenomenon of the density function and energy in the bump regime. We shall provide precise estimate for the behavior of the density function both in the short time regime and in the long time regime. Roughly speaking, in the short time regime, the density function decays very fast. While in long time regime, the density function is very close  to $1$. As a consequence, the energy also decays very fast after a short time.

As a convention in this section, we always assume that $(\r,\vv u)$ is the unique strong solution of \eqref{CNS} obtained in Theorem \ref{global existence thm}  and $\e>0$ is sufficiently small which is independent of $\delta$.
To investigate the dynamic behavior of $a$, we  need the lower bound of the density function $\r$. Indeed
by virtue of \eqref{trajectory} and  $\div\vv u=F+a,$ we deduce from $\rho$ equation of \eqref{CNS} that
\beq\label{E 59}
\f{d}{dt}\r(t,X(t;\tau,y))+\bigl((\r^\gamma-1)\r\bigr)(t,X(t;\tau,y))=-(\r F)(t,X(t;\tau,y)),
\eeq
or
\beno
\f{d}{dt}\r(t,X(t;\tau,y))=-\bigl(\r (F+a)\bigr)(t,X(t;\tau,y)).
\eeno
Solving the above ODE gives rise to
\beq\label{expression for rho 1}\begin{aligned}
\r(t,X(t;\tau,y))&=\r(t_1,X(t_1;\tau,y))\exp\Bigl(-\int_{t_1}^t\bigl(a+F\bigr)(s,X(s;\tau,y))ds\Bigr),
\end{aligned}\eeq
for any $t\geq t_1\geq\tau$.

\subsection{The dynamic behavior of $a$ on the time interval $[0,T_0]$ .}  We first study the dynamics of $a$ away from the initial time. The amptitude of $a$ is supposed to collapse very fast when it is away from the initial time. While the energy of $a$ may grow very fast in this short time interval.

\begin{lem}\label{lower bound lemma}
{\sl Under the assumptions of Theorem \ref{global existence thm}, one has
\beq\label{lower bound of rho}
\r(t,x)\geq 2^{-\f{1}{\gamma}},\ \ \forall\ t>0,\,\,\forall\ x\in\R^3.
\eeq
Moreover,  if $\r(\tau,y)\geq m>1$, one has
\beq\label{lower bound of rho 1}
\r(t,X(t;\tau,y))
\geq m\exp\Bigl(-2(m^\gamma-1)(t-\tau)-\f{1}{4(m^\gamma-1)}\|F\|_{L^2([\tau,t];L^\infty)}^2\Bigr),
\ \ \forall\ t\geq\tau\geq 0.
\eeq}
\end{lem}
\begin{proof} We will follow the same line as  the proof of  Lemma \ref{lemma for rho}. For fixed $(\tau,y)\in[0,\infty)\times\R^3$, we assume that
$\r(\tau,y)\ge r$ with $r>0$. By continuity, we further assume that there exists $t_1\in(\tau,T)$ such that
\beno
\r(t,X(t;\tau,y))>r,\,\,\forall\, t\in[\tau,t_1)
\quad\text{and}\quad
\r(t_1,X(t_1;\tau,y))=r.
\eeno
Otherwise, there holds $\r(t,X(t;\tau,y))\geq r$ for any $t\in[\tau,T]$.

Again by continuity and the definition of $t_1$,  there exists $t_2>t_1$ such that
\beno
\r(t,X(t;\tau,y))\leq r,\,\,\forall\, t\in(t_1,t_2),
\eeno
which ensures
\beno
a(s,X(s;\tau,y))=\r^\gamma(s,X(s;\tau,y))-1\leq r^\gamma-1,\,\,\forall\, s\in(t_1,t_2).
\eeno
Then we deduced from \eqref{expression for rho 1} that for any $t\in (t_1,t_2)$
\beq\label{E 38}\begin{aligned}
\r(t,X(t;\tau,y))&=\r(t_1,X(t_1;\tau,y))\exp\Bigl(-\int_{t_1}^t\bigl(a+F\bigr)(s,X(s;\tau,y))ds\Bigr)\\
&> r\exp\Bigl((1-r^\gamma)(t-t_1)-\sqrt{t-t_1}\|F\|_{L^2([t_1,t];L^\infty)}\Bigr).
\end{aligned}\eeq

\noindent$\bullet$ {\it Case 1: $\tau=0$, $r=(\f23)^{\f{1}{\gamma}}\in(0,1)$.}  We deduce from \eqref{E 38} that for any $t\in(t_1,t_2)$,
\beno
\r(t,X(t;0,y))
\geq r\exp\Bigl(\f12(1-r^\gamma)(t-t_1)-\f{1}{2(1-r^\gamma)}\|F\|_{L^2_t(L^\infty)}^2\Bigr),
\eeno
which together with \eqref{E 41} ensures that
\beno
\r(t,X(t;0,y))\geq \left(\f23\right)^{\f{1}{\gamma}}\exp\bigl(-C\e^2\delta^{2\beta}\bigr),\quad\forall\ t\in(t_1,t_2).
\eeno
Taking $\e$ sufficiently small, we have
\beno
\r(t,X(t;0,y))\geq 2^{-\f{1}{\gamma}},\quad\forall \ t\in(t_1,t_2).
\eeno

 In case $t_2<T$, there holds $\r(t_2,X(t_2;0,y))=r=(\f23)^{\f{1}{\gamma}}$. We can repeat the above argument to get \eqref{lower bound of rho}.

\smallskip

\noindent$\bullet$ {\it Case 2: $1<r=m\leq\r(\tau,y)$.} We deduce from \eqref{E 38} that for any $t\in(t_1,t_2)$
\beno\begin{aligned}
\r(t,X(t;\tau,y))
&\geq m\exp\Bigl(-2(m^\gamma-1)(t-t_1)-\f{1}{4(m^\gamma-1)}\|F\|_{L^2([\tau,t];L^\infty)}^2\Bigr).
\end{aligned}\eeno

In case $t_2<T$, there holds $\r(t_2,X(t_2;\tau,y))=m$. We can repeat the above argument to get \eqref{lower bound of rho 1}.

We
thus complete the proof of  Lemma  \ref{lower bound lemma}.
\end{proof}

\begin{rmk}
If we assume $\inf_{x\in\R^3}\r_0(x)\geq\underline{\r}$ with $\underline{\r}\in(0,1)$ instead of $\inf_{x\in\R^3}\r_0(x)\geq 1,$ it follows from the proof of Lemma  \ref{lower bound lemma} that
\beno
\r(t,x)\geq {\underline{\r}}/2,\quad\forall\ t>0, \forall x\in\R^3.
\eeno
\end{rmk}

Along the same line to the  proof of  Lemma \ref{lemma for rho} and Lemma \ref{lower bound lemma}, we also obtain the upper bound of $\r(t,X(t;\tau,y))$.

\begin{lem}\label{upper bound lemma 1}
{\sl Under the assumption of Theorem \ref{global existence thm}, for any $M>1$, if $\r(\tau,y)\leq M$, there holds
\beq\label{upper bound 1}
\r(t,X(t;\tau,y))
\leq M\exp\Bigl(\f{1}{2(M^\gamma-1)}\|F\|_{L^2([\tau,t];L^\infty)}^2\Bigr),
\ \ \forall\ t\geq\tau\geq 0.
\eeq}
\end{lem}

 Thanks to Lemma \ref{lower bound lemma}, we shall derive the estimate of $\|a(t)\|_{L^\infty}$ on the time interval  $[0,T_0]$ with $T_0\in(7\gamma^{-1}\cdot 2^{-18},\,7\gamma^{-1}\cdot 2^{-17})$.

\begin{prop}\label{dynamics prop for a over [0,1]}
{\sl We assume that $\delta\in(0,2^{-\f{15}{\al}})$ and $\|\r_0^\gamma-1\|_{L^\infty}=\delta^{-\al}\in[2^{N_0}, 2^{N_0+1})$ with $N_0\in\N_{\geq 15}$.
Then there exist
\beno
t_j\eqdefa \f{7}{8\gamma}\sum_{j'=1}^j2^{-(N_0+1-j')}\quad\mbox{for}\quad j=1,2,\cdots, N_0-14,
\eeno
such that
\beq\label{decay for a}
\|a(t)\|_{L^\infty}\leq\Bigl(\f14(1+2^{N_0+1-j})^{-1}+\gamma(t-t_{j-1})\Bigr)^{-1},\quad\forall\ t\in[t_{j-1},t_j],
\eeq
and
\beq\label{decay for a a}
\|a(t)\|_{L^\infty}\leq 4\times10^4,\quad\forall\ t\geq T_0,
\eeq
where $T_0\eqdefa t_{N_0-14}=\f{7}{8\gamma}(2^{-14}-2^{-N_0})\in(7\gamma^{-1}\cdot2^{-18},\,7\gamma^{-1}\cdot 2^{-17})$.
Moreover, there holds
\beq\label{L 1 Lip of a in [0,1]}
\int_0^{T_0}\|a(t)\|_{L^\infty}dt\leq3\gamma^{-1}\ln\|\r_0^\gamma-1\|_{L^\infty}\leq 3\gamma^{-1}\al\ln\delta^{-1}.
\eeq}
\end{prop}
\begin{proof} We  divide the proof  into the following steps.

\no{\bf Step 1. The upper bound estimate of $\r(t,X(t;\tau,y))$ for $(\tau,y)$ satisfying }
\beno
M\geq\r(\tau,y)\geq m>1.
\eeno
In this case, we define for $t\geq\tau\geq 0$ that
\beno
g_{m,\tau}(t)\eqdefa m^{-\gamma}\exp\Bigl( 2\gamma(m^\gamma-1)(t-\tau)+ \f{\gamma}{4(m^\gamma-1)} \|F\|_{L^2([\tau,t];L^\infty)}^2\Bigr).
\eeno
Then it follows from \eqref{lower bound of rho 1} that
\beq\label{E 60}
\r^{-\gamma}(t,X(t;\tau,y))\leq g_{m,\tau}(t),
\ \ \forall\ t\geq\tau\geq 0.
\eeq
Thanks to  \eqref{E 60}, we deduce from \eqref{E 59} that
\beno
\f{d}{dt}\r(t,X(t;\tau,y))+(1-g_{m,\tau}(t))\r^{\gamma+1}(t,X(t;\tau,y))\leq -(\r F)(t,X(t;\tau,y)),
\eeno
which along with \eqref{E 60} implies
\beq\label{E 61}
-\f{1}{\gamma}\f{d}{dt}\r^{-\gamma}(t,X(t;\tau,y))+(1-g_{m,\tau}(t))\leq-(\r^{-\gamma}F)(t,X(t;\tau,y))\leq g_{m,\tau}(t)\|F(t,\cdot)\|_{L^\infty}.
\eeq
By integrating \eqref{E 61} over $[\tau,t]$, we find
\beq\label{E 61a}
\r^{-\gamma}(t,X(t;\tau,y))\geq\r^{-\gamma}(\tau,y)+ \gamma(t-\tau)-\gamma\int_\tau^tg_{m,\tau}(s)ds-\gamma\|g_{m,\tau}\|_{L^2([\tau,t])}\|F\|_{L^2([\tau,t];L^\infty)}.
\eeq
Yet it follows from   \eqref{E 41} that
\beno\begin{aligned}
\int_\tau^tg_{m,\tau}(s)ds&=m^{-\gamma}\exp\Bigl(\f{\gamma}{4(m^\gamma-1)}\|F\|_{L^2([\tau,t];L^\infty)}^2\Bigr)\int_\tau^t\exp\Bigl(2\gamma(m^\gamma-1)(s-\tau)\Bigr)ds\\
&\leq \f{m^{-\gamma}}{2\gamma(m^\gamma-1)}\exp\Bigl(\f{C\gamma\e^2\delta^{2\beta}}{4(m^\gamma-1)}+ 2\gamma(m^\gamma-1)(t-\tau)\Bigr),
\end{aligned}\eeno
and
\beno\begin{aligned}
&\|g_{m,\tau}(s)\|_{L^2([\tau,t])}
\leq \f{m^{-\gamma}}{\sqrt{4\gamma(m^\gamma-1)}}\exp\Bigl(\f{C\gamma\e^2\delta^{2\beta}}{4(m^\gamma-1)}+ 2\gamma(m^\gamma-1)(t-\tau)\Bigr),
\end{aligned}\eeno
so that due to $\r(\tau,y)\leq M$,  we deduce from \eqref{E 61a} that
\beq\label{E 62}\begin{aligned}
\r^{-\gamma}(t,X(t;\tau,y))\geq M^{-\gamma}+ \gamma(t-\tau)-&\Bigl(\f{1}{(m^\gamma-1)m^\gamma}+\f{C\e\delta^{\beta}}{m^{\gamma}\sqrt{m^\gamma-1}}\Bigr)\\
&\times\exp\Bigl(\f{C\e^2\delta^{2\beta}}{4(m^\gamma-1)}+ 2\gamma(m^\gamma-1)(t-\tau)\Bigr).
\end{aligned}\eeq

\no{\bf Step 2. The estimate of $\|a(t)\|_{L^\infty}$ on $[0,t_1]$.}

Let us define
\beno
m_k\eqdefa (1+2^k)^{\f{1}{\gamma}},\quad M_k\eqdefa [2(1+2^k)]^{\f{1}{\gamma}}.
\eeno
Then in view of the assumption that $\|\r_0^\gamma-1\|_{L^\infty}=\delta^{-\al}\in[2^{N_0}, 2^{N_0+1})$, we decompose  $\R^3$ into
\beno
\R^3=\cup_{k=13}^{N_0}\left\{y\in\R^3\,|\,m_k\leq\r_0(y)\leq M_k\right\}\cup\left\{y\in\R^3\,|\,\r_0(y)\leq M_{12}\right\}.
\eeno

For  $y\in \left\{y\in\R^3\,|\,m_k\leq\r_0(y)\leq M_k\right\}$ with $k\in[13,N_0]$, we get,
by  using \eqref{E 62}, that
\beno\begin{aligned}
\r^{-\gamma}(t,X(t;0,y))\geq \f12(1+2^k)^{-1}+\gamma t-&\Bigl(\f{1}{2^k(1+2^k)}+\f{C\e\delta^{\beta}}{2^{\f{k}{2}}(1+2^k)}\Bigr)\\
&\times \exp\Bigl(C\e^2\delta^{2\beta}2^{-k}+\gamma 2^{k+1}t\Bigr).
\end{aligned}\eeno
Taking $\e$ so small that
\beno
\Bigl(\f{1}{2^k(1+2^k)}+\f{C\e\delta^{\beta}}{2^{\f{k}{2}}(1+2^k)}\Bigr)\exp\Bigl(C\e^2\delta^{2\beta}2^{-k}\Bigr)\leq 2^{-\f32k},
\eeno
leads to
\beq\label{H 1}
\r^{-\gamma}(t,X(t;0,y))\geq\f12(1+2^k)^{-1}+\gamma t-2^{-\f32k}\exp\bigl(\gamma 2^{k+1}t\bigr).
\eeq
In particular, for any $t\leq\f7{8\gamma}\cdot2^{-k},$
\beno
2^{-\f32k}\exp\bigl(\gamma 2^{k+1}t\bigr)\leq 2^{-\f32k} e^{\f74}\leq\f18\cdot2^{-k}\leq \f14(1+2^k)^{-1},\quad\text{for any }\ k\geq 13,
\eeno
we obtain
\beno
\r^{-\gamma}(t,X(t;0,y))\geq\f14(1+2^k)^{-1}+\gamma t,
\eeno
which implies
\beq\label{E 63}
\r^{\gamma}(t,X(t;0,y))\leq\Bigl(\f14(1+2^k)^{-1}+\gamma t\Bigr)^{-1},\quad\forall\ t\leq\f7{8\gamma}\cdot2^{-k}.
\eeq

Let
$
t_1\eqdefa\f7{8\gamma}\cdot2^{-N_0}.
$ Then we deduce from
 \eqref{E 63} that for any $y\in\cup_{k=13}^{N_0}\left\{y\in\R^3\,|\qquad\,\right.$ $\left. m_k\leq\r_0(y) \leq M_k\right\}$,
\beq\label{E 64}
\r^{\gamma}(t,X(t;0,y))\leq\Bigl(\f14(1+2^{N_0})^{-1}+\gamma t\Bigr)^{-1},\quad\forall\ t\in[0,t_1].
\eeq

Whereas for any $y\in\left\{y\in\R^3\,|\,\r_0(y)\leq M_{12}\right\}$, we get, by using \eqref{E 41} and \eqref{upper bound 1}, that
\beno
\r^\gamma(t,X(t;0,y))
\leq 2(1+2^{12})\exp\Bigl(\f{\gamma}{2(2^{13}+1)}C\e^2\delta^{2\beta}\Bigr),
\ \ \forall\ t\geq 0.
\eeno
Taking $\e$ so small that for any $15\leq k\leq N_0$
\beq\label{E 64a}
\r^\gamma(t,X(t;0,y))
\leq 4(1+2^{12})\leq\f89\cdot 2^{k}\leq\Bigl(\f14(1+2^{k})^{-1}+\f78\cdot 2^{-k}\Bigr)^{-1},
\ \ \forall\ t\geq 0.
\eeq

Thanks to \eqref{E 64} and \eqref{E 64a}, we conclude that for all $y\in\R^3$,
\beno
\r^\gamma(t,X(t;0,y))\leq\Bigl(\f14(1+2^{N_0})^{-1}+\gamma t\Bigr)^{-1},\quad\forall\ t\in[0,t_1],
\eeno
which yields to
\beq\label{E 65}
\|a(t)\|_{L^\infty}\leq\Bigl(\f14(1+2^{N_0})^{-1}+\gamma t\Bigr)^{-1},\quad\forall\ t\in[0,t_1].
\eeq
Here we used the fact that for any $t\in[0,t_1]$
   $$1\leq\Bigl(\f14(1+2^{N_0})^{-1}+\f78\cdot 2^{-N_0}\Bigr)^{-1}\leq\Bigl(\f14(1+2^{N_0})^{-1}+\gamma t\Bigr)^{-1}.$$
   As a result, it comes out
\beno
\|a(t_1)\|_{L^\infty}\leq\Bigl(\f14(1+2^{N_0})^{-1}+\f78\cdot 2^{-N_0}\Bigr)^{-1}\leq2^{N_0}.
\eeno

\no {\bf Step 3.  The estimate of $\|a(t)\|_{L^\infty}$ on $[t_{j-1},t_j]$, $j=2,3,\cdots, N_0-14$.}

Following the same strategy in Step 2, we take $a(t_1,\cdot)$ as initial data and decompose $\R^3$ into
\beno
\R^3=\cup_{k=13}^{N_0-1}\left\{y\in\R^3\,|\,m_k\leq\r(t_1,y)\leq M_k\right\}\cup\left\{y\in\R^3\,|\,\r(t_1,y)\leq M_{12}\right\}.
\eeno
Along the same line to the  derivations of \eqref{E 63} and
\eqref{E 64a}, there exists $t_2\eqdefa t_1+\f7{8\gamma}\cdot2^{-(N_0-1)}$ such that
\beq\label{E 66}
\|a(t)\|_{L^\infty}\leq\Bigl(\f14(1+2^{N_0-1})^{-1}+\gamma (t-t_1)\Bigr)^{-1},\quad\forall\ t\in[t_1,t_2].
\eeq
 In particular there holds
\beno
\|a(t_2)\|_{L^\infty}\leq\Bigl(\f14(1+2^{N_0-1})^{-1}+\f78\cdot2^{-(N_0-1)}\Bigr)^{-1}\leq 2^{N_0-1}.
\eeno

By
repeating the above arguments, we conclude that there exist $\{t_j\}_{j=1}^{N_0-14}$ such that
$t_j=t_{j-1}+\f7{8\gamma}\cdot2^{-(N_0+1-j)}$ and there holds
 \eqref{decay for a}. We remark that the procedure in Step 2 are repeated $N_0-14$ times according to the relation \eqref{E 64a}.

\no{\bf Step 4.  The estimate of $\|a(t)\|_{L^\infty}$ on $t\geq T_0\eqdefa t_{N_0-14}$.}

We first get, by using the definition of $t_j$, that
\beno
T_0\eqdefa t_{N_0-14}=\f7{8\gamma}\cdot \sum_{j=1}^{N_0-14}2^{-(N_0+1-j)}=\f7{8\gamma}\bigl(2^{-14}-2^{-N_0}\bigr).
\eeno
Then it follows from \eqref{decay for a} that
\beno
\|a(T_0)\|_{L^\infty}\leq\Bigl(\f14(1+2^{14+1})^{-1}+\f78\cdot2^{-(14+1)} \Bigr)^{-1}\leq\f{9}{10}\cdot2^{15},
\eeno
which implies that
\beno
\|\r(T_0)\|_{L^\infty}\leq \bigl(1+\f{9}{10}\cdot2^{15}\bigr)^{\f{1}{\gamma}}.
\eeno
So that thanks to \eqref{upper bound 1} and \eqref{E 41}, we get  for sufficiently small $\varepsilon>0$ and $ t\geq T_0$ that
\beno\begin{aligned}
&\r(t,X(t;T_0,y))
\leq\bigl(1+\f{9}{10}\cdot2^{15}\bigr)^{\f{1}{\gamma}}\exp\Bigl(C2^{-15}\cdot\e^2\delta^{2\beta}\Bigr)<2^{15/\gamma},
\end{aligned}\eeno
which yields to
\beno
\|a(t)\|_{L^\infty}\leq 2^{15}<4\times10^4,\quad \forall\ t\geq T_0.
\eeno
This gives rise to \eqref{decay for a a}.

\no{\bf Step 5. The estimate of  $\|a\|_{L^1_{T_0}(L^\infty)}$.}

Thanks to   \eqref{decay for a}, we get
\beno\begin{aligned}
\int_0^{T_0}\|a(t,\cdot)\|_{L^\infty}dt&\leq\sum_{j=1}^{N_0-14}\int_{t_{j-1}}^{t_j}\Bigl(\f14(1+2^{N_0+1-j})^{-1}+\gamma(t-t_{j-1})\Bigr)^{-1}dt\\
&\leq\gamma^{-1}\sum_{j=1}^{N_0-14}\ln\Bigl(1+4(1+2^{N_0+1-j})\cdot\f78\cdot2^{-(N_0+1-j)}\Bigr)\\
&\leq 2\gamma^{-1}(N_0-14).
\end{aligned}\eeno
Due to $\|\r_0^\gamma-1\|_{L^\infty}=\delta^{-\al}\in[2^{N_0}, 2^{N_0+1})$, we obtain
\beno
\int_0^{T_0}\|a(t,\cdot)\|_{L^\infty}dt\leq\f{2\gamma^{-1}}{\ln 2}\ln\|\r_0^\gamma-1\|_{L^\infty}\leq 3\gamma^{-1}\ln\|\r_0^\gamma-1\|_{L^\infty},
\eeno
which is  \eqref{L 1 Lip of a in [0,1]}. This completes the proof of the proposition.
\end{proof}

\begin{rmk}
It follows from \eqref{E 6} and \eqref{decay for a a} that
\beq\label{dynamic 12}
\|a(t)\|_{L^2}^2\sim\|\r(t)-1\|_{L^2}^2\sim H(\r(t)), \quad\forall \ t\geq T_0.
\eeq
\end{rmk}



\subsection{The dynamic behavior of $a(t,\cdot)$ for $t\geq T_0$ .}
We first improve the estimates of $\|a(t)\|_{L^\infty}$ for $t\geq T_0$ obtained in Proposition \ref{dynamics prop for a over [0,1]}.

\begin{prop}\label{dynamics prop for L infty a}
{\sl Under the assumptions of Proposition \ref{dynamics prop for a over [0,1]}, one has
\beq\label{dynamic 5}
\|a(t)\|_{L^\infty}\leq C\bigl(e^{-\f{\gamma}{2}t}+\e\delta^{1-\f34\al}\bigr),\quad\forall\ t\geq T_0.
\eeq
In particular, there hold
\beq\label{dynamic 6}
\|a(t)\|_{L^\infty}\leq C\e\delta^{1-\f34\al},\quad\forall\ t\geq T_1\eqdefa 2\gamma^{-1}\ln(\e\delta^{1-\f34\al})^{-1},
\eeq and
\beq\label{L 1 Lip of a over long time interval}
\int_{T_0}^T\|a(t)\|_{L^\infty}dt\leq C,\quad\text{for any}\ T_0\leq T\leq CT_2\eqdefa C(\e\delta^{1-\f34\al})^{-2}.
\eeq}
\end{prop}

\begin{proof}
Recalling the $a$-equation \eqref{E 11} and using the fact that $\div\vv u=F+a$, we write
\beno
\p_ta+\vv u\cdot\na a+ \gamma\r^\gamma a=-\gamma\r^\gamma F,
\eeno
which implies
\beno
\p_t|a|+\vv u\cdot\na |a|+ \gamma \r^\gamma|a|=-\gamma\r^\gamma F\cdot\f{a}{|a|}.
\eeno
Thanks to \eqref{lower bound of rho} and \eqref{decay for a a}, we get
\beq\label{E 69}
\p_t|a|+\vv u\cdot\na |a|+\f{\gamma}{2}|a|\leq C\|F(t)\|_{L^\infty},\ \forall\ t\geq T_0.
\eeq
Then along the trajectory $X(t;T_0,y)$  defined in \eqref{trajectory}, we write
\beno
\f{d}{dt}|a(t,X(t;T_0,y))|+\f{\gamma}{2}|a(t,X(t;T_0,y))|\leq C\|F(t)\|_{L^\infty},
\eeno
from which and \eqref{decay for a a}, we infer
\beq\label{E 70}\begin{aligned}
|a(t,X(t;T_0,y))|&\leq e^{-\f{\gamma}{2}(t-T_0)}|a(T_0,y)|+C\int_{T_0}^te^{-\f{\gamma}{2}(t-\tau)}\|F(\tau)\|_{L^\infty}d\tau\\
&\leq C\bigl(e^{-\f{\gamma}{2}(t-T_0)}+\|F\|_{L^2((T_0,t);L^\infty)}\bigr).
\end{aligned}\eeq
Yet it follows from \eqref{estimates for q}, \eqref{decay for a a} and \eqref{total enery estimate} that
\beq\label{E 70a}
\|F\|_{L^2((T_0,t);L^\infty)}\lesssim\|\r\|_{L^\infty_t(L^\infty)}^{\f34}\|\sqrt\r\dot{\vv u}\|_{L^2_t(L^2)}^{\f12}\|\na\dot{\vv u}\|_{L^2_t(L^2)}^{\f12}\lesssim\e\delta^{1-\f34\al},\quad\forall\ t\geq T_0.
\eeq
By substituting  \eqref{E 70a} into \eqref{E 70} and using
 the fact that $T_0=O(1),$  we obtain \eqref{dynamic 5}.

Moreover, we deduce from  \eqref{E 70a} that for any $T>T_0$,
\begin{align*}
\int_{T_0}^T\int_{T_0}^te^{-\f{\gamma}{2}(t-\tau)}\|F(\tau)\|_{L^\infty}d\tau \,dt\leq &\|e^{-\f{\gamma}{2}(t-T_0)}\|_{L^1((T_0,T))}\|F\|_{L^1((T_0,T);L^\infty)}\\
\leq & C\sqrt{T-T_0}\cdot\e\delta^{1-\f34\al},
\end{align*}
from which, \eqref{decay for a a} and \eqref{E 70}, we deduce that
\beno\begin{aligned}
\int_{T_0}^T\|a(t)\|_{L^\infty}dt&\leq C\int_{T_0}^Te^{-\f{\gamma}{2}(t-T_0)}dt+C\int_{T_0}^T\int_{T_0}^te^{-\f{\gamma}{2}(t-\tau)}\|F(\tau)\|_{L^\infty}d\tau dt\\
&\leq C+C\e\delta^{1-\f34\al}\sqrt{T-T_0},
\end{aligned}\eeno
which leads to \eqref{L 1 Lip of a over long time interval}. This completes the proof of  Proposition \ref{dynamics prop for L infty a}.
\end{proof}

Thanks to \eqref{decay for a a}, we are going to improve the estimates of $\|a(t)\|_{L^2}$ and $\|a(t)\|_{L^6}$.

\begin{prop}\label{dynamics prop for L 6 a}
 {\sl Under the assumptions of Theorem \ref{global dynamics thm}, we have
 \begin{subequations} \label{S3eq567}
\begin{gather}
\|a(t)\|_{L^6}^2\lesssim\e^2\bigl(\delta^{1-2\al}\cdot e^{-\f{t}{6}}+\delta^{3-\al}\bigr),\quad\forall\ t\geq T_0;\label{dynamic 7}
\\
 \|a(t)\|_{L^6}^2\lesssim\e^2\delta^{3-\al},\quad\forall\ t\geq 6(2+\al)\ln\delta^{-1};\label{dynamic 8}
\end{gather}
\end{subequations}
and
\begin{subequations} \label{S3eq568}
\begin{gather}
 \|a(t)\|_{L^2}^2\lesssim\e^{\f65}\delta^{3-2\al}\cdot e^{-\f{t}{2}}+\e^2\delta^{3-\al},\quad\forall\ t\geq 0;\label{dynamic 9}
\\
\|a(t)\|_{L^2}^2\lesssim\e^{\f65}\delta^{3-\al},\quad\forall\ t\geq2\al\ln\delta^{-1}.\label{dynamic 10}
\end{gather}
\end{subequations}
}
\end{prop}

\begin{proof}
By virtue of  \eqref{decay for a a}, we deduce from \eqref{estimates for a in L 6} that
\beno
6\f{d}{dt}\|a(t)\|_{L^6}^2+\|a\|_{L^6}^2
\lesssim
(1+\|a\|_{L^\infty}^2)\|\r\|_{L^\infty}\|\sqrt\r\dot{\vv u}\|_{L^2}^2\lesssim\|\sqrt\r\dot{\vv u}\|_{L^2}^2,\quad\forall\ t\geq T_0,
\eeno
which implies
\beno
\|a(t)\|_{L^6}^2\lesssim e^{-\f{t}{6}}\|a(T_0)\|_{L^6}^2+\|\sqrt\r\dot{\vv u}\|_{L^2_t(L^2)}^2,\quad\forall\ t\geq T_0.
\eeno
Thanks to \eqref{total enery estimate}, we obtain
\beq\label{E 71}
\|a(t)\|_{L^6}^2\lesssim\e^2\bigl(\delta^{1-2\al} e^{-\f{t}{6}}+\delta^{3-\al}\bigr),\quad\forall\ t\geq T_0.
\eeq
This yields \eqref{dynamic 7} and \eqref{dynamic 8}.

On the other hand,  to derive the dynamical estimates \eqref{dynamic 9} and \eqref{dynamic 10}, we first observe from \eqref{initial ansatz} and \eqref{initial ansatz in L 1} that
\beno
\|a_0\|_{L^1}\leq(\ga-1)H(\r_0)+\ga\|\r_0-1\|_{L^1}
\leq C\d^{3-\al},
\eeno
which implies
\beq\label{a 0}
\|a_0\|_{L^2}\leq \|a_0\|_{L^1}^{\f25}\|a_0\|_{L^6}^{\f35}\leq C\e^{\f35}\d^{\f32-\al}.
\eeq

Along the same line to the derivation of \eqref{E 33}, we get, by taking $L^2$-inner product of \eqref{E 11} with $a,$ that
\beno
\f12\f{d}{dt}\|a(t)\|_{L^2}^2+\int_{\R^3}\bigl(\gamma+(\gamma-\f12)a\bigr) a^2\, dx=
-\gamma\int_{\R^3}F a\, dx-\bigl(\gamma-\f12\bigr)\int_{\R^3}Fa^2\, dx.
\eeno
Since $\gamma+\bigl(\gamma-\f12\bigr)a=\gamma+\bigl(\gamma-\f12\bigr)\bigl(\r^\gamma-1\bigr)>\f12$, we obtain
\beno
\f12\f{d}{dt}\|a(t)\|_{L^2}^2+\f{1}{2}\|a\|_{L^2}^2
\leq\gamma\|F\|_{L^2}\|a\|_{L^2}+(\gamma-\f12)\|F\|_{L^\infty}\|a\|_{L^2}^2.
\eeno
Using Young's inequality gives rise to
\beq\label{E 75}
\f{d}{dt}\|a(t)\|_{L^2}^2+\f{1}{2}\|a\|_{L^2}^2
\leq 2\gamma^2\|F\|_{L^2}^2+2\gamma^2\|F\|_{L^\infty}^2\|a\|_{L^2}^2.
\eeq
By applying Gronwall's inequality to \eqref{E 75} and using \eqref{E 41}, we find
\beno\begin{aligned}
\|a(t)\|_{L^2}^2&\leq e^{-\f{t}{2}+2\gamma^2\|F\|_{L^2_t(L^\infty)}^2}\Bigl(\|a_0\|_{L^2}^2+2\gamma^2\int_0^te^{\f{\tau}{2}}d\tau \|F\|_{L^\infty_t(L^2)}^2\Bigr)\\
&\leq e^{-\f{t}{2}+C\e^2\delta^{2\beta}}\|a_0\|_{L^2}^2+2\gamma^2e^{C\e^2\delta^{2\beta}}\|F\|_{L^\infty_t(L^2)}^2\\
&\leq Ce^{-\f{t}{2}}\|a_0\|_{L^2}^2+C\|F\|_{L^\infty_t(L^2)}^2,
\end{aligned}\eeno
from which, \eqref{total enery estimate} and \eqref{a 0}, we infer
\beno
\|a(t)\|_{L^2}^2\lesssim \e^{\f65}\delta^{3-2\al}\cdot e^{-\f{t}{2}}+\e^2\delta^{3-\al},
\eeno
which leads to \eqref{dynamic 9} and \eqref{dynamic 10}. This completes the proof of the  proposition.
\end{proof}

\subsection{The proof of Theorem \ref{global dynamics thm}.} In this subsection, we shall   derive the large time
 decay estimates  for the solutions of \eqref{CNS} for $t\geq T_2\eqdefa(\e\delta^{1-\f34\al})^{-2}$. In order to do so,
 we  need the following estimates concerning the low frequency parts  of $\varrho\eqdefa\r-1$ and $\r\vv u$.

\begin{lem}\label{lemma for rho and rho u}
{\sl  For $r>1,$ let
\beno
{\varrho}\eqdefa\r-1,\quad{\varrho}_0\eqdefa\r_0-1\quad\text{and}\quad S_r(t)\eqdefa\left\{\xi\in\R^3\,|\, |\xi|\leq r\w{t}^{-\f12}\right\},
\eeno
 Then if ${\varrho}_0,\,\r_0\vv u_0\in L^1(\R^3),$ one has
\beq\label{estimate for rho and rho u over low frequencies}
\begin{aligned}
\int_{S_r(t)}&\bigl(\gamma|\widehat{{\varrho}}(t,\xi)|^2+|\widehat{\r\vv u}(t,\xi)|^2\bigr)\,d\xi
\lesssim  r^3\w{t}^{-\f32}\bigl(\gamma\|{\varrho}_0\|_{L^1}^2+\|\r_0\vv u_0\|_{L^1}^2\bigr)\\
&+r^5\w{t}^{-\f52}\|{\varrho}\|_{L^2_t(L^2)}^2\|\vv u\|_{L^\infty_t(L^2)}^2+r^4\w{t}^{-2}\int_0^t\bigl(\|\sqrt\r\vv u\|_{L^2}^2+H(\r)\bigr)\|{\varrho}\|_{L^2}\|\vv u\|_{L^2}\,d\tau\\
&+r^3\w{t}^{-\f32}\int_0^t\bigl(\|\sqrt\r\vv u\|_{L^2}^4+H^2(\r)\bigr)\,d\tau.
\end{aligned}\eeq}
\end{lem}
\begin{proof}
We first get,  by taking Fourier transform to the transport equation of \eqref{CNS}, that
\beno
\p_t\widehat{{\varrho}}(t,\xi)+i\xi\cdot\widehat{\r\vv u}(t,\xi)=0,
\eeno
from which, we infer
\beq\label{H 9a}
\f12\f{d}{dt}|\widehat{{\varrho}}(t,\xi)|^2-\im\left(\bigl(\xi\cdot\widehat{\r\vv u}\bigr)\overline{\widehat{{\varrho}}}(t,\xi)\right)=0.
\eeq

While we get,  by using the definition of $h(\r)$ in \eqref{def of H rho}, that
\beno
a=(\gamma-1)h(\r)+\gamma{\varrho},
\eeno
from which and the $\vv u$-equation of \eqref{CNS}, we write
\beno
\p_t(\r\vv u)+\div(\r\vv u\otimes\vv u)-\D\vv u+(\gamma-1)\na h(\r)+\gamma\na{\varrho}=\vv 0.
\eeno
By taking Fourier transform to the above equation and then multiplying the resulting equation
by $\overline{\widehat{\r\vv u}},$ and finally taking the real part, we obtain
\beq\label{H 9}\begin{aligned}
&\f12\f{d}{dt}|\widehat{\r\vv u}(t,\xi)|^2+|\xi|^2\Re\bigl(\widehat{\vv u}\cdot\overline{\widehat{\r\vv u}}\bigr)-\gamma \im\bigl(\bigl(\xi\cdot\overline{\widehat{\r\vv u}}\bigr)\cdot\widehat{{\varrho}}\bigr)\\
&=\im\bigl(\widehat{\r\vv u\otimes\vv u}:\overline{\widehat{\r\vv u}}\otimes\xi\bigr)+(\gamma-1)\im\bigl(\widehat{h(\r)}\bigl(\xi\cdot\overline{\widehat{\r\vv u}}\bigr)\bigr).
\end{aligned}\eeq
Due to $\r\vv u=\vv u+{\varrho}\vv u$, we write
\beno
|\xi|^2\Re\bigl(\widehat{\vv u}\cdot\overline{\widehat{\r\vv u}}\bigr)=|\xi|^2|\widehat{\vv u}|^2+|\xi|^2\Re\bigl(\widehat{\vv u}\cdot\overline{\widehat{{\varrho}\vv u}}\bigr),
\eeno
so that we deduce from   \eqref{H 9a} and \eqref{H 9}  that
\beno\begin{aligned}
&\gamma|\widehat{{\varrho}}(t,\xi)|^2+|\widehat{\r\vv u}(t,\xi)|^2
+2\int_0^t|\xi|^2|\widehat{\vv u}|^2\,d\tau=\gamma|\widehat{{\varrho}}(0,\xi)|^2+|\widehat{\r_0\vv u_0}(\xi)|^2
\\
&-2\int_0^t
|\xi|^2\Re\bigl(\widehat{\vv u}\cdot\overline{\widehat{{\varrho}\vv u}}\bigr)
+2\int_0^t\im\bigl(\widehat{\r\vv u\otimes\vv u}:\overline{\widehat{\r\vv u}}\otimes\xi\bigr)\,d\tau+2(\gamma-1)\int_0^t\im\bigl(\widehat{h(\r)}\bigl(\xi\cdot\overline{\widehat{\r\vv u}}\bigr)\bigr)\,d\tau.
\end{aligned}\eeno
Integrating the above equality over $S_r(t)$ gives rise to
\beq\label{H 10}\begin{aligned}
&\int_{S_r(t)}\Bigl(\gamma|\widehat{{\varrho}}(t,\xi)|^2+|\widehat{\r\vv u}(t,\xi)|^2\Bigr)d\xi
+2\int_{S_r(t)}\int_0^t\bigl|\xi|^2|\widehat{\vv u}\bigr|^2d\tau d\xi\\
&\quad=\int_{S_r(t)}\Bigl(\gamma|\widehat{{\varrho}_0}(\xi)|^2+|\widehat{\r_0\vv u_0}(\xi)|^2\Bigr)d\xi+\mathfrak{B}_1+\mathfrak{B}_2+\mathfrak{B}_3,
\end{aligned}\eeq
where
\beno\begin{aligned}
&\mathfrak{B}_1\eqdefa-2\int_{S_r(t)}\int_0^t|\xi|^2\Re\bigl(\widehat{\vv u}\cdot\overline{\widehat{{\varrho}\vv u}}\bigr)\,d\tau\, d\xi,\quad \mathfrak{B}_2\eqdefa2\int_{S_r(t)}\int_0^t
\im\bigl(\widehat{\r\vv u\otimes\vv u}:\overline{\widehat{\r\vv u}}\otimes\xi\bigr)\, d\tau \,d\xi,\\
&\mathfrak{B}_3\eqdefa2(\gamma-1)\int_{S_r(t)}\int_0^t\im\bigl(\widehat{h(\r)}\bigl(\xi\cdot\overline{\widehat{\r\vv u}}\bigr)\bigr)\,d\tau \,d\xi.
\end{aligned}\eeno

Next we estimate the terms on r.h.s. of \eqref{H 10} term by term.

\no{\Large $\bullet$} {\bf The  Estimate of $\mathfrak{B}_1$.}

 Recalling that $S_r(t)\eqdefa\{\xi\in\R^3\,|\, |\xi|\leq r\w{t}^{-\f12}\}$,
we get, by applying H\"older's and Young's inequalities, that
\beno\begin{aligned}
|\mathfrak{B}_1|
&\leq\f{1}{2}\int_{S_r(t)}\int_0^t|\xi|^2|\widehat{\vv u}|^2d\tau d\xi
+2\int_{S_r(t)}\int_0^t|\xi|^2|\widehat{{\varrho}\vv u}|^2d\tau d\xi\\
&\leq\f{1}{2}\int_{S_r(t)}\int_0^t|\xi|^2|\widehat{\vv u}|^2d\tau d\xi
+2\int_{S_r(t)}|\xi|^2d\xi\int_0^t\|\widehat{{\varrho}\vv u}(\tau)\|_{L^\infty_\xi}^2d\tau \\
&\leq\underbrace{\f{1}{2}\int_{S_r(t)}\int_0^t|\xi|^2|\widehat{\vv u}|^2d\tau d\xi}_{\text{absorbed by the l.h.s. of \eqref{H 10}}}+Cr^5\w{t}^{-\f52}\|{\varrho}\|_{L^2_t(L^2)}^2\|\vv u\|_{L^\infty_t(L^2)}^2.
\end{aligned}\eeno

\no{\Large $\bullet$} {\bf The estimate of $\mathfrak{B}_2$.}

Since $\r\vv u=\vv u+{\varrho}\vv u$, we have
\beno\begin{aligned}
|\mathfrak{B}_2|
&=2\Bigl|\int_{S_r(t)}\int_0^t \im\bigl(\widehat{\r\vv u\otimes\vv u}:(\overline{\widehat{\vv u}}+\overline{\widehat{{\varrho}\vv u}})\otimes\xi\bigr)\, d\tau d\xi\Bigr|\\
&
\leq\f{1}{2}\int_{S_r(t)}\int_0^t|\xi|^2|\widehat{\vv u}|^2d\tau d\xi+2\int_{S_r(t)}d\xi\int_0^t\|\widehat{\r\vv u\otimes\vv u}(\tau)\|_{L^\infty_\xi}^2 d\tau\\
&\quad+2\int_{S_r(t)}|\xi|d\xi\int_0^t\|\widehat{\r\vv u\otimes\vv u}(\tau)\|_{L^\infty_\xi}\|\widehat{{\varrho}\vv u}(\tau)\|_{L^\infty_\xi}\,d\tau\\
&\leq\underbrace{\f{1}{2}\int_{S_r(t)}\int_0^t|\xi|^2|\widehat{\vv u}|^2d\tau d\xi}_{\text{absorbed by the l.h.s. of \eqref{H 10}}}+C r^3\w{t}^{-\f32} \int_0^t\|\sqrt\r\vv u(\tau)\|_{L^2}^4\,d\tau\\
&\quad+C r^4\w{t}^{-2}\int_0^t\|\sqrt\r\vv u(\tau)\|_{L^2}^2\|{\varrho}(\tau)\|_{L^2}\|\vv u(\tau)\|_{L^2}\,d\tau.
\end{aligned}\eeno

\no{\Large $\bullet$} {\bf The estimate of  $\mathfrak{B}_3$.}

Similar to the estimate of $\mathfrak{B}_2$, we have
\beno\begin{aligned}
|\mathfrak{B}_3|&=2(\gamma-1)\Bigl|\int_{S_r(t)}\int_0^t\im\bigl(\widehat{h(\r)}\xi\cdot\bigl(\overline{\widehat{\vv u}}+\overline{\widehat{{\varrho}\vv u}}\bigr)\bigr)\,d\tau d\xi\Bigr|\\
&\leq\f12\int_{S_r(t)}\int_0^t|\xi|^2|\widehat{\vv u}|^2d\tau d\xi+2(\gamma-1)^2\int_{S_r(t)}d\xi\int_0^t\|\widehat{h(\r)}(\tau)\|_{L^\infty_\xi}^2\,d\tau\\
&\quad+2(\gamma-1)\int_{S_r(t)}|\xi|d\xi\cdot\int_0^t\|\widehat{h(\r)}(\tau)\|_{L^\infty_\xi}\|\widehat{{\varrho}\vv u}(\tau)\|_{L^\infty_\xi}\,d\tau\\
&\leq\underbrace{\f12\int_{S_r(t)}\int_0^t|\xi|^2|\widehat{\vv u}|^2d\tau d\xi}_{\text{absorbed by the l.h.s. of \eqref{H 10}}}+C r^3\w{t}^{-\f32}\int_0^tH^2(\r(\tau))\,d\tau\\
&\quad+C r^4\w{t}^{-2}\int_0^tH(\r(\tau))\|{\varrho}(\tau)\|_{L^2}\|\vv u(\tau)\|_{L^2}d\tau,
\end{aligned}\eeno
where we used the fact that
$
\|\widehat{h(\r)}(\tau)\|_{L^\infty_\xi}\leq\|h(\r)(\tau)\|_{L^1}=H(\r(\tau)).
$

Finally it is easy to observe that
\beno\begin{aligned}
&\int_{S_r(t)}\bigl(\gamma|\widehat{{\varrho_0}}(\xi)|^2+|\widehat{\r_0\vv u_0}(\xi)|^2\bigr)\,d\xi\leq C r^3\w{t}^{-\f32}\bigl(\gamma\|\varrho_0\|_{L^1}^2+\|\r_0\vv u_0\|_{L^1}^2\bigr).
\end{aligned}\eeno

By substituting the above estimates into \eqref{H 10},
we achieve \eqref{estimate for rho and rho u over low frequencies}.  This completes the proof of the lemma.
\end{proof}

Now we are in a position to complete  the proof of Theorem \ref{global dynamics thm}.

\begin{proof}[Proof of Theorem \ref{global dynamics thm}]
We
first deduce \eqref{dynamic for a in L infty} and \eqref{dynamic for a in L 1 L infty} from
 Propositions \ref{dynamics prop for a over [0,1]} and \ref{dynamics prop for L infty a}. Whereas \eqref{dynamic for a in L 6}
  follows from \eqref{total enery estimate} and Proposition \ref{dynamics prop for L 6 a}.

It remains to justify the large time   decay estimates \eqref{decay of energy 1} and \eqref{decay of energy 2} for the strong solutions of \eqref{CNS}. In what follows, we divide the proof into the following steps:

\no{\bf Step 1. Energy estimates for $t\geq T_*$.}

 Here $T_*\geq T_0$ will be determined  later on. Along the same line to the derivation of \eqref{E 48}, and
  using the fact that $\|a(t)\|_{L^\infty}\leq C$ for $t\geq T_0$, we find
\beq\label{H 15}
\f{d}{dt}E_2(t)+\|\na\dot{\vv u}\|_{L^2}^2+\|a\|_{L^6}^2\leq C_1\|\sqrt\r\dot{\vv u}\|_{L^2}^2+C_1\e^{\f23}\|\na\vv u\|_{L^2}^2,
\eeq
for $E_2(t)$ being defined by \eqref{S2eq10}.

Let $E_1(t)$ be determined by \eqref{S2eq10}, we denote
\beq\label{defi of E and D}\begin{aligned}
&E(t)\eqdefa E_1(t)+\f{1}{4C_1}E_2(t),\\
&D(t)\eqdefa \f\gamma4\|\na\vv u\|_{L^2}^2+\f14\|\sqrt\r\dot{\vv u}\|_{L^2}^2
+\f{1}{4C_1}\|\na\dot{\vv u}\|_{L^2}^2+\f{1}{4C_1}\|a\|_{L^6}^2.
\end{aligned}\eeq
Then for sufficiently small $\e$, we deduce from  \eqref{E 45} and \eqref{H 15} that
\beq\label{H 17}
\f{d}{dt}E(t)+D(t)\leq 0,\quad\forall t\geq T_0.
\eeq

It is easy to observe from  \eqref{lower bound of rho} and \eqref{decay for a a} that
\beq\label{H 20}
\r(t,x)\sim 1, \quad\forall t\geq T_0,\,\forall\ x\in\R^3.
\eeq
Whereas in view of the definition of $F$ and  \eqref{dynamic 12}, we have for any $t\geq T_0$
\begin{align*}
&\|\div\vv u\|_{L^2}^2
\leq2\|F\|_{L^2}^2+2\|a\|_{L^2}^2
\leq2\|F\|_{L^2}^2+CH(\r) \andf\\
&H(\r)\sim\|a\|_{L^2}^2\leq \|F\|_{L^2}^2+\|\div \vv u\|_{L^2}^2,
\end{align*}
from which and \eqref{defi of E and D}, we infer
\beq\label{H 18}\begin{aligned}
E(t)
\sim& \|(\vv u, \na \vv u, a, F, \dot{\vv u})(t)\|_{L^2}^2+H(\r(t))+\|a(t)\|_{L^6}^2\\
\sim& \|(\vv u, \na\vv u, F, \dot{\vv u})(t)\|_{L^2}^2+\|a(t)\|_{L^6}^2,\quad\forall\ t\geq T_0.
\end{aligned}\eeq

On the other hand, it follows from \eqref{elliptic estimates} that
\beno
\|\na F\|_{L^2}\leq C\|\r\dot{\vv u}\|_{L^2}\leq \|\dot{\vv u}\|_{L^2},\quad\forall\ t\geq T_0,
\eeno
which together  with \eqref{H 20} ensures that
\beq\label{H 19}\begin{aligned}
D(t)& \sim \|(\na\vv u,\dot{\vv u},\na\dot{\vv u})(t)\|_{L^2}^2+\|a(t)\|_{L^6}^2\\
& \gtrsim\|(\na\vv u,\na F,\na\dot{\vv u})(t)\|_{L^2}^2+\|a(t)\|_{L^6}^2.
\end{aligned}\eeq

Recalling that $S_r(t)=\{\xi\in\R^3\,|\,|\xi|\leq r\w{t}^{-\f12}\}$, we have
\beno
\|\na\vv u\|_{L^2}^2 =(2\pi)^{-3} \int_{\R^3}|\xi|^2|\widehat{\vv u}|^2d\xi
\geq r^2\w{t}^{-1}\|\vv u\|_{L^2}^2-(2\pi)^{-3}r^2\w{t}^{-1}\int_{S_r(t)}|\widehat{\vv u}(t,\xi)|^2d\xi,
\eeno
where the constant $r>1$ will be chosen later on.
Hence thanks to \eqref{H 18}, \eqref{H 19} and the fact that
\beno
1\geq r^2\w{t}^{-1},\quad\  \forall\ t\geq r^2>T_0,
\eeno
 there exists a constant $c_0\in(0,\f12)$ so that
for any $t\geq r^2>T_0$,
\beno
D(t)\geq 4c_0r^2\w{t}^{-1}E(t)-4c_0r^2\w{t}^{-1}\int_{S_r(t)}\bigl(|\widehat{\vv u}(t,\xi)|^2+|\widehat{F}(t,\xi)|^2\bigr)\,d\xi.
\eeno
By inserting the above estimate into \eqref{H 17}, we obtain
\beno
\f{d}{dt}E(t)+4c_0r^2\w{t}^{-1}E(t)\leq 4c_0r^2\w{t}^{-1}\int_{S_r(t)}\bigl(|\widehat{\vv u}(t,\xi)|^2+|\widehat{F}(t,\xi)|^2\bigr)\,d\xi.
\eeno
Taking $r=c_0^{-\f12}$ in the above inequality gives rise to
\beq\label{H 21}
\f{d}{dt}E(t)+4\w{t}^{-1}E(t)\leq 4\w{t}^{-1}\int_{S_r(t)}\bigl(|\widehat{\vv u}(t,\xi)|^2+|\widehat{F}(t,\xi)|^2\bigr)\,d\xi,
\eeq  for any $t\geq T_*\eqdefa c_0^{-1}.$
We remark that the interest of \eqref{H 21} is that the term on the left is  a typical
one that creates decay. Thus it needs a careful analysis. In the rest of the proof, we always treat $r=c_0^{-\f12}$ as a given constant.

Notice that
\beno
F=\div\vv u-a\andf a=\r^\gamma-1=(\gamma-1)h(\r)+\gamma{\varrho},
\eeno
we have
\beno\begin{aligned}
\int_{S_{r}(t)}|\widehat{F}(t,\xi)|^2\,d\xi
&\leq 2\int_{S_r(t)}\bigl(|\xi|^2|\widehat{\vv u}|^2+|\widehat{a}|^2\bigr)d\xi\\
&\lesssim\w{t}^{-1}\int_{S_r(t)}|\widehat{\vv u}|^2d\xi
+\int_{S_r(t)}\Bigl((\gamma-1)^2|\widehat{h(\r)}|^2+\gamma^2|\widehat{{\varrho}}|^2\Bigr)d\xi\\
&\lesssim\int_{S_r(t)}|\widehat{\vv u}|^2d\xi+\|\widehat{h(\r)}(t)\|_{L^\infty_\xi}^2\int_{S_r(t)}d\xi+\int_{S_r(t)}|\widehat{{\varrho}}|^2d\xi\\
&\lesssim\int_{S_r(t)}|\widehat{\vv u}|^2d\xi+\w{t}^{-\f32}H^2(\r(t))+\int_{S_r(t)}|\widehat{{\varrho}}|^2d\xi.
\end{aligned}\eeno
Substituting the above inequality into \eqref{H 21} leads to
\beq\label{H 21a}\begin{aligned}
\f{d}{dt}E(t)+4\w{t}^{-1}E(t)\leq & C\w{t}^{-\f52}\cdot H^2(\r(t))+C\w{t}^{-1}\int_{S_r(t)}\bigl(|\widehat{\vv u}(t,\xi)|^2+|\widehat{{\varrho}}(t,\xi)|^2\bigr)d\xi.
\end{aligned}\eeq

Due to $|{\varrho}|=|\r-1|\leq |\r^\gamma-1|=|a|$, we get, by using \eqref{dynamic 9}, that
\beq\label{H 11}
\|{\varrho}(t)\|_{L^2}^2\lesssim \|a(t)\|_{L^2}^2\lesssim\e^{\f65}\delta^{3-2\al}, \quad\forall\ t>0,
\eeq
so that one has
\beno\begin{aligned}
\int_{S_r(t)}|\widehat{\vv u}|^2d\xi&\leq2\int_{S_r(t)}|\widehat{\r\vv u}|^2d\xi+2\int_{S_r(t)}|\widehat{{\varrho}\vv u}|^2d\xi\\
&\leq 2\int_{S_r(t)}|\widehat{\r\vv u}|^2d\xi+C\w{t}^{-\f32}\|{\varrho}\|_{L^2}^2\|\vv u\|_{L^2}^2\\
&\leq 2\int_{S_r(t)}|\widehat{\r\vv u}|^2d\xi+C\e^{\f65}\delta^{3-2\al}\|\vv u\|_{L^2}^2.
\end{aligned}\eeno
Since $\e$ is sufficiently small, the term $C\e^2\delta^{3-2\al}\|\vv u(t)\|_{L^2}^2\w{t}^{-1}$ can be absorbed by the l.h.s. of \eqref{H 21a}.
As a result, it comes out
\beq\label{H 22}\begin{aligned}
\f{d}{dt}E(t)+2\w{t}^{-1}E(t)\leq & C\w{t}^{-\f52}\cdot H^2(\r(t))\\
& +C\w{t}^{-1}\int_{S_r(t)}\bigl(|\widehat{\r\vv u}(t,\xi)|^2+|\widehat{{\varrho}}(t,\xi)|^2\bigr)d\xi,\quad\forall\ t\geq T_*.
\end{aligned}\eeq

\no{\bf Step 2. The rough estimate for the last term on r.h.s of \eqref{H 22}.}

In order to do so, we shall use \eqref{estimate for rho and rho u over low frequencies}.
Due to $|{\varrho}|\leq |\r^\gamma-1|=|a|$, we first get, by using \eqref{dynamic 9}, that
\beq\label{H 12}\begin{aligned}
\|{\varrho}\|_{L^2_t(L^2)}^2\leq\int_0^t\|a\|_{L^2}^2d\tau
&\leq C\int_0^t\bigl(e^{-\f{t}{2}}\cdot\e^{\f65}\delta^{3-2\al}+\e^2\delta^{3-\al}\bigr)d\tau\\
&\leq C\e^{\f65}\delta^{3-2\al}+C\e^2\delta^{3-\al} t
\leq C\e^{\f65}\delta^{3-2\al}\langle t\rangle.
\end{aligned}\eeq
While it follows from \eqref{total enery estimate}, \eqref{lower bound of rho} and \eqref{H 12} that
\beq\label{H 12a}
\|{\varrho}\|_{L^2_t(L^2)}^2\|\vv u\|_{L^\infty_t(L^2)}^2
\lesssim\|{\varrho}\|_{L^2_t(L^2)}^2\|\sqrt\r\vv u\|_{L^\infty_t(L^2)}^2\lesssim \e^{\f{16}{5}}\delta^{6-3\al}\w{t},
\eeq
and
\beno\begin{aligned}
\int_0^t&\bigl(\|\sqrt\r\vv u\|_{L^2}^2+H(\r)\bigr)\|{\varrho}\|_{L^2}\|\vv u\|_{L^2}\,d\tau\\
&\lesssim t^{\f12}\|{\varrho}\|_{L^2_t(L^2)}\sup_{\tau\in(0,t)}\Bigl(\bigl(\|\sqrt\r\vv u\|_{L^2}^2+H(\r)\bigr)\|\sqrt\r\vv u\|_{L^2}\Bigr)\\
&\lesssim t^{\f12}\cdot\w{t}^{\f12}\cdot\e^{\f35}\delta^{\f32-\al}\cdot\Bigl(\e^2\delta^{3-\al}\Bigr)^{\f32}\\
&\lesssim
\e^{\f{18}{5}}\delta^{6-\f52\al}\w{t},
\end{aligned}\eeno
and
\beno\begin{aligned}
&\int_0^t\bigl(\|\sqrt\r\vv u\|_{L^2}^4+H(\r)^2\bigr)d\tau
\leq t\sup_{\tau\in(0,t)}\bigl(\|\sqrt\r\vv u\|_{L^2}^4+H^2(\r)\bigr)\leq Ct\e^4\delta^{6-2\al},
\end{aligned}\eeno
from which and \eqref{initial ansatz in L 1}, we deduce from \eqref{estimate for rho and rho u over low frequencies} that
\beno\begin{aligned}
&\int_{S_r(t)}\bigl(\gamma|\widehat{{\varrho}}(t,\xi)|^2+|\widehat{\r\vv u}(t,\xi)|^2\bigr)\,d\xi\\
&\lesssim \w{t}^{-\f32}\cdot\delta^{2(3-\al)}+\w{t}^{-\f32}\cdot\e^{\f{16}{5}}\delta^{6-3\al}+\w{t}^{-1}
\cdot\e^{\f{18}{5}}\delta^{6-\f52\al}
+\w{t}^{-\f12}\cdot\e^4\delta^{6-2\al}.
\end{aligned}\eeno
Then there holds
\beq\label{H 14}
\int_{S_r(t)}\bigl(\gamma|\widehat{{\varrho}}(t,\xi)|^2+|\widehat{(\r\vv u)}(t,\xi)|^2\bigr)d\xi\lesssim\delta^{6-3\al}\w{t}^{-\f12}.
\eeq

\no{\bf Step 3. The rough large time decay estimates.}

We get, by applying  \eqref{total enery estimate}, \eqref{H 22} and \eqref{H 14}, that
\beno\begin{aligned}
\f{d}{dt}E(t)+2\w{t}^{-1}E(t)
&\leq C\e^4\delta^{6-2\al}\w{t}^{-\f52}+C\delta^{6-3\al}\w{t}^{-\f32}\\
&\leq C\delta^{6-3\al}\w{t}^{-\f32},
\end{aligned}\eeno
from which, we infer
\beno
E(t)\leq \bigl(\w{T_*}^2E(T_*)+C\delta^{6-3\al}\w{t}^{\f32}\bigr)\w{t}^{-2},\quad\forall\ t\geq T_*.
\eeno
Yet it follows from  \eqref{total enery estimate} that
\beq\label{H 25a}
E(T_*)\leq C\e^2\delta^{1-2\al}.
\eeq
As a result, it comes out
\beq\label{H 23}
E(t)\leq C\delta^{1-2\al}\w{t}^{-\f12},\quad\forall\ t\geq T_*.
\eeq

\no{\bf Step 4. The first improved decay estimates for $t\geq T_*$.}

Thanks to \eqref{H 20},  \eqref{H 18}, \eqref{H 11} and \eqref{H 23}, we have
\beno
\|(\vv u, \sqrt\r\vv u, {\varrho})(t)\|_{L^2}^2+H(\r)(t)\lesssim E(t)\lesssim\delta^{1-2\al}\w{t}^{-\f12},\quad\forall\ t\geq T_*,
\eeno
which along with \eqref{total enery estimate} implies that for any $t\geq T_*$
\beno\begin{aligned}
&\int_{T_*}^t\bigl(\|\sqrt\r\vv u\|_{L^2}^2+H(\r)\bigr)\|{\varrho}\|_{L^2}\|\vv u\|_{L^2}d\tau+\int_{T_*}^t\bigl(\|\sqrt\r\vv u\|_{L^2}^4+H^2(\r)\bigr)d\tau\\
&\leq\sup_{t\geq T_*}\bigl(\|\sqrt\r\vv u\|_{L^2}^2+H(\r)\bigr)\cdot
\int_{T_*}^t\bigl(\|{\varrho}\|_{L^2}\|\vv u\|_{L^2}+\|\sqrt\r\vv u\|_{L^2}^2+H(\r)\bigr)d\tau\\
&\lesssim\e^2\delta^{3-\al}\cdot\delta^{1-2\al}\cdot\int_{T_*}^t(1+\tau)^{-\f12}d\tau\lesssim \e^2\delta^{4-3\al}\w{t}^{\f12}.
\end{aligned}\eeno
While we get, by using \eqref{total enery estimate}, \eqref{lower bound of rho} and \eqref{H 11}, that
\beq\label{H 29}\begin{aligned}
\int_0^{T_*}\bigl(\|\sqrt\r\vv u\|_{L^2}^2&+H(\r)\bigr)\|{\varrho}\|_{L^2}\|\vv u\|_{L^2}d\tau+\int_0^{T_*}\bigl(\|\sqrt\r\vv u\|_{L^2}^4+H^2(\r)\bigr)d\tau\\
&\qquad\qquad\lesssim\Bigl(\bigl(\e^2\delta^{3-\al}\bigr)^{\f32}\cdot\e^{\f35}\delta^{\f32-\al}+\bigl(\e^2\delta^{3-\al}\bigr)^2\Bigr) T_*\lesssim\e^{\f{18}{5}}\delta^{6-\f52\al}.
\end{aligned}\eeq
As a consequence, we obtain
\beno
\int_0^t\bigl(\|\sqrt\r\vv u\|_{L^2}^2+H(\r)\bigr)\|{\varrho}\|_{L^2}\|\vv u\|_{L^2}d\tau+\int_0^t\bigl(\|\sqrt\r\vv u\|_{L^2}^4+H(\r)^2\bigr)d\tau\lesssim\e^2\delta^{4-3\al}\w{t}^{\f12}.
\eeno
By substituting the above estimate, \eqref{H 12a}, \eqref{initial ansatz in L 1} into  \eqref{estimate for rho and rho u over low frequencies}, we achieve
\beq\label{H 24a}
\int_{S_r(t)}\bigl(\gamma|\widehat{{\varrho}}(t,\xi)|^2+|\widehat{(\r\vv u)}(t,\xi)|^2\bigr)d\xi\lesssim\delta^{4-3\al}\w{t}^{-1},
\eeq
which along with \eqref{total enery estimate}, \eqref{H 22} and \eqref{H 24a} ensures that
\beno
\f{d}{dt}E(t)+2\w{t}^{-1}E(t)\lesssim\delta^{4-3\al}\w{t}^{-2},
\eeno
from which, \eqref{H 20}, \eqref{H 18} and \eqref{H 25a}, we infer
\beq\label{H 26a}\begin{aligned}
\|(\vv u, \sqrt\r\vv u,{\varrho})(t)\|_{L^2}^2+H(\r(t))
\lesssim E(t)&\lesssim \bigl(\langle T_*\rangle^2E(T_*)+\delta^{4-3\al}t\bigr)\w{t}^{-2}\\
&\lesssim\delta^{1-2\al}\w{t}^{-1},\quad\forall\ t\geq T_*.
\end{aligned}\eeq

\no{\bf Step 5. The second improved decay estimates for $t\geq T_*$.}

Thanks to \eqref{total enery estimate} and \eqref{H 26a}, we have
\beno\begin{aligned}
&\int_{T_*}^t\bigl(\|\sqrt\r\vv u\|_{L^2}^2+H(\r)\bigr)\|{\varrho}\|_{L^2}\|\vv u\|_{L^2}d\tau+\int_{T_*}^t\bigl(\|\sqrt\r\vv u\|_{L^2}^4+H(\r)^2\bigr)d\tau\\
&\lesssim\bigl(\e^2\delta^{3-\al}\bigr)^{\f34}\cdot\bigl(\delta^{1-2\al}\bigr)^{\f54}\cdot\int_{T_*}^t(1+\tau)^{-\f54}d\tau
\lesssim\e^{\f32}\delta^{\f72-\f{13}{4}\al},
\end{aligned}\eeno
which along with \eqref{H 29} implies
\beno
\int_0^t\bigl(\|\sqrt\r\vv u\|_{L^2}^2+H(\r)\bigr)\|{\varrho}\|_{L^2}\|\vv u\|_{L^2}d\tau+\int_0^t\bigl(\|\sqrt\r\vv u\|_{L^2}^4+H(\r)^2\bigr)d\tau\lesssim\e^{\f32}\delta^{\f72-\f{13}{4}\al}.
\eeno
Then by inserting the above estimate, \eqref{H 12a} and  \eqref{initial ansatz in L 1} into \eqref{estimate for rho and rho u over low frequencies},
we  improve the estimate \eqref{H 14} to be
\beno\label{H 24}
\int_{S_r(t)}\bigl(\gamma|\widehat{{\varrho}}(t,\xi)|^2+|\widehat{(\r\vv u)}(t,\xi)|^2\bigr)\,d\xi\lesssim\delta^{\f72-\f{13}{4}\al}\w{t}^{-\f32},
\eeno
from which, \eqref{total enery estimate} and \eqref{H 22}, we deduce that
\beno
\f{d}{dt}E(t)+2\w{t}^{-1}E(t)\lesssim\delta^{\f72-\f{13}{4}\al}\w{t}^{-\f52}.
\eeno
Then, we infer
\beq\label{H 26}
E(t)\leq \bigl(\w{t_0}^2E(t_0)+C\delta^{\f72-\f{13}{4}\al}\w{t}^{\f12}\bigr)\w{t}^{-2},\quad\forall\ t\geq t_0\geq T_*.
\eeq
Taking $t_0=T_*$ and using \eqref{H 25a}, we deduce from \eqref{H 26} that
\beq\label{H 27}
E(t)\leq C\delta^{1-2\al}\w{t}^{-2}+C\delta^{\f72-\f{13}{4}\al}\w{t}^{-\f32},\quad\forall\ t\geq T_*,
\eeq
which together  with \eqref{H 18} ensures \eqref{decay of energy 1}.

\no{\bf Step 6. The improved energy estimates for $t\geq T_*$.}

By integrating \eqref{H 17} over $[T,t]$ and using \eqref{H 27}, we obtain for any $t\geq T\geq T_*$,
\beno
E(t)+\int_{T}^{t}D(\tau)d\tau\leq E(T)\leq C\delta^{1-2\al}\w{T}^{-2}+C\delta^{\f72-\f{13}{4}\al}\w{T}^{-\f32},
\eeno
which together  with \eqref{H 18} and \eqref{H 19} ensures \eqref{decay of energy 2}. This completes the proof of the theorem.
\end{proof}
\begin{rmk}
By taking $T=T_2\eqdefa (\e\delta^{1-\f34\al})^{-2}$ in \eqref{decay of energy 2}, we  obtain
\beq\label{decay of energy 2a}\begin{aligned}
&\sup_{t\geq T_2}\bigl(\|(\vv u,\na\vv u,a,F,\dot{\vv u})(t)\|_{L^2}^2+\|a(t)\|_{L^6}^2\bigr)\\
&\quad+\int_{T_2}^\infty\bigl(\|(\na\vv u,\dot{\vv u},\na\dot{\vv u})(t)\|_{L^2}^2+\|a(t)\|_{L^6}^2\bigr)dt\leq C\e^3\delta^{5(1-\al)}.
\end{aligned}
\eeq
\end{rmk}

\setcounter{equation}{0}

\section{The propagation of regularities and global dynamics.}
In this section, we shall show the precise estimates for the propagation of the high regularities, in particular,
we shall derive the estimates of $\|\na a(t)\|_{L^6}$ and $\|\na\vv u\|_{L^1_t(L^\infty)}$. Then we have the full picture on the global dynamic behavior of the solutions to \eqref{CNS} obtained in Theorem \ref{global existence thm}.

\subsection{The estimates of $\|\na a(t)\|_{L^p}$ with $p=2,6$}

\begin{lem}\label{lemma for na a}
{\sl Let $(\r,\vv u)$ be the unique strong solution of \eqref{CNS} obtained in Theorem \ref{global existence thm}. Then there holds
\beq\label{estimate of na a}\begin{aligned}
\f{d}{dt}\|\na a(t)\|_{L^p}^2+\f\gamma2\|\na a\|_{L^p}^2\leq & C\bigl(1+\|a\|_{L^\infty}^2\bigr)\|\r\dot{\vv u}\|_{L^p}^2\\
&+C\bigl(\|\r\dot{\vv u}\|_{L^2}\|\r\dot{\vv u}\|_{L^6}+\|a\|_{L^6}^2+\|a\|_{L^\infty}\bigr)\|\na a\|_{L^p}^2\\
&+8\bar{c}_*\gamma \ln\bigl(\|\na a\|_{L^6}^2+1\bigr)\cdot\|a\|_{L^\infty}\|\na a\|_{L^p}^2,\quad\text{for}\ p=2,6,\\
 \mbox{where\,\,} \bar{c}_*\eqdefa&\|\mathcal{F}^{-1}\bigl(\f{\xi\otimes\xi}{|\xi|^2}\varphi(\xi)\bigr)\|_{L^1} \andf
 \varphi \ \  \mbox{is given by Definition} \ref{S0def1}.
\end{aligned}\eeq
}
\end{lem}
\begin{proof}
We
first get, by applying Young's inequality to \eqref{estimate of na a 1}, that
\beq\label{E 73}\begin{aligned}
\f{d}{dt}\|\na a(t)\|_{L^p}^2+\gamma\|\na a\|_{L^p}^2
\leq&4\gamma\|\na\vv u\|_{L^\infty}\|\na a\|_{L^p}^2+2\gamma\|a\|_{L^\infty}\|\na a\|_{L^p}^2\\
&+4\gamma\bigl(1+\|a\|_{L^\infty}^2\bigr)\|\na F\|_{L^p}^2.
\end{aligned}\eeq
Observing that $-\D\vv u=\curl\curl\vv u-\na F-\na a$, we have
\beno
\|\na\vv u\|_{L^\infty}\leq\|\na(-\D)^{-1}\curl\curl\vv u\|_{L^\infty}+\|(\na\otimes\na)(-\D)^{-1} F\|_{L^\infty}+\|(\na\otimes\na)(-\D)^{-1} a\|_{L^\infty}.
\eeno
By using the interpolation inequality that $\|f\|_{L^\infty}\lesssim\|f\|_{L^6}^{\f12}\|\na f\|_{L^6}^{\f12}\lesssim\|\na f\|_{L^2}^{\f12}\|\na f\|_{L^6}^{\f12}$ and the elliptic estimate \eqref{elliptic estimates}, we get
\beq\label{E 74}\begin{aligned}
\|\na\vv u\|_{L^\infty}
&\leq C\|\na\curl\vv u\|_{L^2}^{\f12}\|\na\curl\vv u\|_{L^6}^{\f12}+C\|\na F\|_{L^2}^{\f12}\|\na F\|_{L^6}^{\f12}+\|(\na\otimes\na)(-\D)^{-1} a\|_{L^\infty}\\
&\leq C\|\r\dot{\vv u}\|_{L^2}^{\f12}\|\r\dot{\vv u}\|_{L^6}^{\f12}+\|(\na\otimes\na)(-\D)^{-1} a\|_{L^\infty}.
\end{aligned}\eeq

To handle the estimate of $\|(\na\otimes\na)(-\D)^{-1} a\|_{L^\infty}$, for any integer $N\in\N,$ we get, by applying Definition \ref{S0def1}
 and Bernstein Lemma (see \cite{bcd}), that
\beno\begin{aligned}
&\|(\na\otimes\na)(-\D)^{-1} a\|_{L^\infty}
\leq\|\D_{-1}(\na\otimes\na)(-\D)^{-1} a\|_{L^\infty}\\
&\qquad+\sum_{j=0}^{N-1}\|\dot{\D}_j(\na\otimes\na)(-\D)^{-1} a\|_{L^\infty}+\sum_{j=N}^\infty\|\dot{\D}_j(\na\otimes\na)(-\D)^{-1} a\|_{L^\infty}\\
&\leq c_*\|\D_{-1} a\|_{L^6}+\sum_{j=0}^{N-1}\|\dot{\D}_j(\na\otimes\na)(-\D)^{-1} a\|_{L^\infty}
+c_*\sum_{j=N}^\infty2^{\f{j}{2}}\|\dot{\D}_ja\|_{L^6}\\
&\leq c_*\|a\|_{L^6}+\bar{c}_*N\|a\|_{L^\infty}
+c_*2^{-\f{N}{2}}\|\na a\|_{L^6}.
\end{aligned}\eeno
Here
$c_*\ge1$ is a universal constant independent of $\gamma,\al$, and  $\bar{c}_*$ is given by \eqref{estimate of na a}. By
taking
\beno
N\sim 2(\ln2)^{-1}\cdot\ln\bigl(16c_*(\|\na a\|_{L^6}+1)\bigr),
\eeno in the above inequality,
we find
\beno
c_*2^{-\f{N}{2}}\|\na a\|_{L^6}\le\f{1}{16}
\eeno
and
\beq\label{E 77}
\|(\na\otimes\na)(-\D)^{-1} a\|_{L^\infty}
\leq c_*\|a\|_{L^6}+4\bar{c}_*\ln\bigl(16c_*(\|\na a\|_{L^6}+1)\bigr)\|a\|_{L^\infty}
+\f{1}{16}.
\eeq

Thanks to \eqref{elliptic estimates}, \eqref{E 73}, \eqref{E 74} and \eqref{E 77}, we deduce that
\beno\begin{aligned}
\f{d}{dt}\|\na a(t)\|_{L^p}^2&+\f34\gamma\|\na a\|_{L^p}^2
\leq C\bigl(1+\|a\|_{L^\infty}^2\bigr)\|\r\dot{\vv u}\|_{L^p}^2+C\bigl(\|\r\dot{\vv u}\|_{L^2}^{\f12}\|\r\dot{\vv u}\|_{L^6}^{\f12}\\
&\qquad+\|a\|_{L^6}+\|a\|_{L^\infty}\bigr)\|\na a\|_{L^p}^2+16\bar{c}_*\gamma\ln\bigl(16c_*(\|\na a\|_{L^6}+1)\bigr)\|a\|_{L^\infty}\cdot\|\na a\|_{L^p}^2.
\end{aligned}\eeno
Since
\beno
16\bar{c}_*\gamma\ln\bigl(16c_*(\|\na a\|_{L^6}+1)\bigr)\leq16\bar{c}_*\gamma\ln\bigl(32c_*\bigr)+8\bar{c}_*\gamma\ln\bigl(\|\na a\|_{L^6}^2+1\bigr),
\eeno
by using Young's inequality, we obtain \eqref{estimate of na a}. It completes the proof of the lemma.
\end{proof}

\subsection{The estimate of $\|\na a(t)\|_{L^6}$ on the time interval $[0, 4T_2]$}

\begin{prop}\label{propagation prop for na a 1}
{\sl Let $\al\in\bigl(0,\f{1}{1+\f{1}{2\gamma}}\bigr]$ and $\delta\in\bigl(0,2^{-\f{15}{\al}}\bigr)$. If
\beq\label{initial ansatz for na a}
\|\r_0-1\|_{L^1}\leq C\delta^{3-\al},\quad\|\na a_0\|_{L^6}^2\leq C\delta^{-(1+2\al)},
\eeq
one has
\beq\label{rough estimate for L 6 na a 2}
\|\na a(t)\|_{L^6}^2\leq \exp\bigl(C\delta^{-(24\bar{c}_*+1)\al}\bigr),\quad\forall\  t\leq 4T_2\eqdefa4(\e\delta^{1-\f34\al})^{-2}.
\eeq
Furthermore, if $\al\leq\min\Bigl\{\f{1+\f{1}{3\gamma}}{2+\f{1}{4\gamma}},\ \f{1}{24\bar{c}_*+\f{7}{4}}\Bigr\}$,  there holds
\beq\label{value of na a between T_2 and 2T_2}
\|\na a(t)\|_{L^6}^2\leq C\e^3\delta^{3-\f72\al},\quad\forall\ t\in[2T_2,4T_2].
\eeq}
\end{prop}
\begin{proof}
We shall use the dynamical  estimates of $a$ and the decay estimates of the energy to derive the bound of $\|\na a(t)\|_{L^6}$ step by step.

\no {\bf Step 1. Rough estimate of $\|\na a(t)\|_{L^6}$ over time interval $[0,4T_2]$.}

By taking $p=6$ in \eqref{estimate of na a}, we write
\beq\label{H 2}\begin{aligned}
&\f{d}{dt}\|\na a(t)\|_{L^6}^2+\f\gamma2\|\na a\|_{L^6}^2\leq C\bigl(1+\|a\|_{L^\infty}^2\bigr)\|\r\|_{L^\infty}^2\|\na\dot{\vv u}\|_{L^2}^2+f(t)\|\na a\|_{L^6}^2\with\\
&
f(t)\eqdefa C\bigl(\|\r\dot{\vv u}\|_{L^2}\|\r\dot{\vv u}\|_{L^6}+\|a\|_{L^6}^2\bigr)+C\|a\|_{L^\infty}+8\bar{c}_*\gamma\ln\bigl(\|\na a\|_{L^6}^2+1\bigr)\|a\|_{L^\infty}.
\end{aligned}\eeq

Integrating \eqref{H 2} over $[0,t]$ yields to
\beno\begin{aligned}
\|\na a(t)\|_{L^6}^2+\f\gamma2\|\na a\|_{L^2_t(L^6)}^2&\leq\|\na a_0\|_{L^6}^2+C\bigl(1+\|a\|_{L^\infty_t(L^\infty)}^2\bigr)\|\r\|_{L^\infty_t(L^\infty)}^2\|\na\dot{\vv u}\|_{L^2_t(L^2)}^2\\
&\qquad+\int_0^tf(\tau)\|\na a(\tau)\|_{L^6}^2d\tau,
\end{aligned}\eeno
which along with \eqref{estimate for a}, \eqref{total enery estimate} and \eqref{initial ansatz for na a} implies
\beq\label{H 3}
\|\na a(t)\|_{L^6}^2+\f\gamma2\|\na a\|_{L^2_t(L^6)}^2\leq\underbrace{C\delta^{-(1+2\al)}+C\e^2\delta^{1-2\al(2+\f{1}{\gamma})}+\int_0^tf(\tau)\|\na a(\tau)\|_{L^6}^2d\tau}_{G(t)}.
\eeq
It follows from \eqref{H 3} that
\beno
\f{d}{dt}G(t)\leq f(t)G(t),
\eeno
and
\beno
f(t)\leq C\bigl(\underbrace{\|\r\dot{\vv u}\|_{L^2}\|\r\dot{\vv u}\|_{L^6}+\|a\|_{L^6}^2+\|a\|_{L^\infty}}_{g(t)}\bigr)+8\bar{c}_*\gamma\|a\|_{L^\infty}\ln (G(t)+1),
\eeno
which imply
\beq\label{H 4}
\f{d}{dt}\ln (G(t)+1)\leq 8\bar{c}_*\gamma\|a\|_{L^\infty}\ln (G(t)+1)+Cg(t).
\eeq
Applying Gronwall's inequality to \eqref{H 4}, we have
\beq\label{H 5}
\ln (G(t)+1)\leq \exp\Bigl\{8\bar{c}_*\gamma\int_0^t\|a\|_{L^\infty}d\tau\Bigr\}\cdot\Bigl(\ln (G(0)+1)+C\int_0^t g(\tau)d\tau\Bigr).
\eeq

It is easy to observe from
\eqref{L 1 Lip of a in [0,1]} and \eqref{L 1 Lip of a over long time interval} that for any $t\leq 4T_2$,
\beq\label{H 5a}
\int_0^t\|a(\tau)\|_{L^\infty}\,d\tau\leq\int_0^{T_0}\|a(\tau)\|_{L^\infty}\,d\tau+\int_{T_0}^t\|a(\tau)\|_{L^\infty}\,d\tau
\leq3\gamma^{-1}\al\ln\delta^{-1}+C,
\eeq
from which, \eqref{estimate for a} and \eqref{total enery estimate},
we deduce that for any $t\leq4T_2$
\beq\label{H 5b}\begin{aligned}
\int_0^t g(\tau)\,d\tau&\leq \|\r\|_{L^\infty_t(L^\infty)}^{\f32}\|\sqrt\r\dot{\vv u}\|_{L^2_t(L^2)}\|\na\dot{\vv u}\|_{L^2_t(L^2)}+\|a\|_{L^2_t(L^6)}^2+\int_0^t\|a(\tau)\|_{L^\infty}\,d\tau\\
&\leq C\e^2\bigl(\delta^{2-\f{3\al}{2}(1+\f{1}{\gamma})}+\delta^{1-2\al}\bigr)+C\al\ln\delta^{-1}+C,
\end{aligned}\eeq
and
\beno
\exp\Bigl\{8\bar{c}_*\gamma\int_0^t\|a\|_{L^\infty}d\tau\Bigr\}\leq\exp\Bigl\{24\bar{c}_*\al\ln\delta^{-1}+C\gamma\Bigr\}
\leq C\delta^{-24\bar{c}_*\al}.
\eeno
Since $2-\f{3\al}{2}(1+\f{1}{\gamma})>0$ and
\beno
G(0)\leq C\bigl(\delta^{-(1+2\al)}+\e^2\delta^{1-2\al(2+\f{1}{\gamma})}\bigr),
\eeno
 we deduce from \eqref{H 5} that
\beno\begin{aligned}
\ln (G(t)+1)
\leq &C\delta^{-24\bar{c}_*\al}\cdot\Bigl(\ln\bigl(\delta^{-(1+2\al)}+\e^2\delta^{1-2\al(2+\f{1}{\gamma})}\bigr)+\al\ln\delta^{-1}+\e^2\delta^{1-2\al}\Bigr)\\
\leq &C\delta^{-(24\bar{c}_*+1)\al},\quad\forall\ t\leq 4T_2,
\end{aligned}\eeno
where we used the facts that $\al<1$ and $\ln\delta^{-1}\leq C_\eta\delta^{-\eta}$ for any $\eta>0$. Then we get
\beq\label{rough estimate for L 6 na a}\begin{aligned}
&\ln(\|\na a(t)\|_{L^6}^2+1)\leq\ln (G(t)+1)
\leq C\delta^{-(24\bar{c}_*+1)\al},\\
&\|\na a(t)\|_{L^6}^2\leq\exp\Bigl(C\delta^{-(24\bar{c}_*+1)\al}\Bigr),\quad\forall\  t\leq4T_2,
\end{aligned}\eeq
which leads to
 \eqref{rough estimate for L 6 na a 2}.

\no {\bf Step 2. Improved estimate of $\|\na a(t)\|_{L^6}$ on $[T_1,4T_2]\eqdefa[2\gamma^{-1}\ln(\e\delta^{1-\f34\al})^{-1},4(\e\delta^{1-\f34\al})^{-2}]$.}

Thanks to \eqref{dynamic 6} and \eqref{rough estimate for L 6 na a}, we deduce from
 the definition of $f(t)$  that
\beq\label{H 2a}\begin{aligned}
f(t)\lesssim& \bigl(\|\sqrt\r\dot{\vv u}\|_{L^2}\|\na\dot{\vv u}\|_{L^2}+\|a\|_{L^6}^2\bigr)+\delta^{-(24\bar{c}_*+1)\al}\cdot\e\delta^{1-\f34\al},\quad\forall
t\geq T_1,
\end{aligned}\eeq
which along with \eqref{H 2} implies that
\beno\label{H 6}\begin{aligned}
\f{d}{dt}\|\na a\|_{L^6}^2+\f\gamma2\|\na a\|_{L^6}^2\lesssim &\|\na\dot{\vv u}\|_{L^2}^2+\bigl(\|\sqrt\r\dot{\vv u}\|_{L^2}\|\na\dot{\vv u}\|_{L^2}+\|a\|_{L^6}^2\bigr)\|\na a\|_{L^6}^2\\
&+\e\delta^{1-\bigl(24\bar{c}_*+\f74\bigr)\al}\|\na a\|_{L^6}^2,
\end{aligned}\eeno
so that
for $\al\leq\f{1}{24\bar{c}_*+\f{7}{4}}$ and  sufficiently small  $\e$, there holds for any $t\geq T_1$
\beq\label{H 7}\begin{aligned}
\f{d}{dt}\|\na a(t)\|_{L^6}^2+\f\gamma4\|\na a\|_{L^6}^2&\lesssim\|\na\dot{\vv u}\|_{L^2}^2+\bigl(\|\sqrt\r\dot{\vv u}\|_{L^2}\|\na\dot{\vv u}\|_{L^2}+\|a\|_{L^6}^2\bigr)\|\na a\|_{L^6}^2.
\end{aligned}\eeq
By applying Gronwall's inequality to \eqref{H 7} and using \eqref{total enery estimate}, we find
\beno\begin{aligned}
\|\na a(t)\|_{L^6}^2&\leq \Bigl(\|\na a(T_1)\|_{L^6}^2\exp\Bigl(-\f\gamma4(t-T_1)\Bigr)+C\int_{T_1}^t\|\na\dot{\vv u}\|_{L^2}^2d\tau\Bigr)\\
&\qquad\times \exp\Bigl(C\bigl(\|\sqrt\r\dot{\vv u}\|_{L^2_t(L^2)}\|\na\dot{\vv u}\|_{L^2_t(L^2)}+\int_{T_1}^t\|a\|_{L^6}^2d\tau\bigr)\Bigr).
\end{aligned}\eeno
It follows from  \eqref{dynamic 7} and the assumption that $\al\leq\min\left\{\f{1+\f{1}{3\gamma}}{2+\f{1}{4\gamma}},\ \f{1}{24\bar{c}_*+\f{7}{4}}\right\}$ that
\beno\begin{aligned}
\int_{T_1}^{4T_2}\|a(t)\|_{L^6}^2dt&\lesssim\int_{T_1}^{4T_2}e^{-\f{t}{6}}dt\cdot\e^2\delta^{1-2\al}+\e^2\delta^{3-\al}(4T_2-T_1)\\
&\lesssim\bigl(\e\delta^{1-\f34\al}\bigr)^{\f{1}{3\gamma}}\cdot\e^2\delta^{1-2\al}+\e^2\delta^{3-\al}\cdot\bigl(\e\delta^{1-\f34\al}\bigr)^{-2}\\
&\lesssim\e^2\Bigl(\delta^{1+\f{1}{3\gamma}-(2+\f{1}{4\gamma})\al}+\delta^{1+\f12\al}\Bigr)\lesssim1.
\end{aligned}\eeno
While we observe from \eqref{total enery estimate} that
\beno
\|\sqrt\r\dot{\vv u}\|_{L^2_t(L^2)}\|\na\dot{\vv u}\|_{L^2_t(L^2)}
\lesssim\e^2\delta^{2-\f32\al}\leq 1\quad\text{and}\quad \int_{T_1}^t\|\na\dot{\vv u}(\tau)\|_{L^2}^2\,d\tau\lesssim\e^2\delta^{1-2\al}.
\eeno
Then using \eqref{rough estimate for L 6 na a 2} for $t=T_1$, we get for any $t\in[T_1,4T_2]$,
\beq\label{H 8}\begin{aligned}
\|\na a(t)\|_{L^6}^2&\leq \exp\Bigl(-\f{\gamma}{4}(t-T_1)+C\delta^{-(24\bar{c}_*+1)\al}\Bigr)+C\e^2\delta^{1-2\al}.
\end{aligned}\eeq

\no{\bf Step 3. Improved estimate of $\|\na a(t)\|_{L^6}^2$ on $[2T_2,4T_2]$.}

 Due to \eqref{H 8} and the assumption that $\al\leq\min\left\{\f{1+\f{1}{3\gamma}}{2+\f{1}{4\gamma}},\ \f{1}{24\bar{c}_*+\f{7}{4}}\right\}$, we have for any $t\in[T_2,4T_2]$,
\beno\begin{aligned}
\|\na a(t)\|_{L^6}^2&\leq C\exp\Bigl(-\f\gamma4\cdot\e^{-2}\delta^{-2(1-\f34\al)}+\f12\ln\bigl(\e\delta^{1-\f34\al}\bigr)^{-1}
+C\delta^{-(24\bar{c}_*+1)\al}\Bigr)+C\e^2\delta^{1-2\al}\\
&\leq C\exp\Bigl(-\f\gamma5\e^{-2}\delta^{-2(1-\f34\al)}\Bigr)+C\e^2\delta^{1-2\al}\leq C\e^2\delta^{2-\f32\al}+C\e^2\delta^{1-2\al}.
\end{aligned}\eeno
which implies
\beq\label{H 8a}
\|\na a(t)\|_{L^6}^2\leq C\e^2\delta^{1-2\al},\quad\forall\, t\in[T_2,4T_2].
\eeq
Whereas it follows from \eqref{decay of energy 2a} that
\beno
\int_{T_2}^\infty\|\na\dot{\vv u}(t)\|_{L^2}^2\,dt\leq C\e^3\delta^{5(1-\al)}.
\eeno
Then by repeating the argument in Step 2 and using \eqref{H 8a} for $t=T_2$, we improve the estimate \eqref{H 8a} to be
\beno\begin{aligned}
\|\na a(t)\|_{L^6}^2&\leq C\e^2\delta^{1-2\al}\exp\Bigl(-\f\gamma4(t-T_2)\Bigr)+C\e^3\delta^{5(1-\al)}\\
&\leq C\e^2\delta^{1-2\al}\exp\Bigl(-\f\gamma4 \e^{-2}\delta^{-2(1-\f34\al)}\Bigr)+C\e^3\delta^{5(1-\al)}\\
&\leq C\e^4\delta^{3-\f72\al}+C\e^3\delta^{5(1-\al)}\leq C\e^3\delta^{3-\f72\al},\quad\forall\ t\in[2T_2,4T_2],
\end{aligned}\eeno
which leads to \eqref{value of na a between T_2 and 2T_2}.
This finishes  the proof of  Proposition \ref{propagation prop for na a 1}.
\end{proof}

\subsection{The proof of Theorem \ref{propagation regularity thm}}
In this subsection, we shall first prove that both $\|\na a(t)\|_{L^2}^2$ and $\|\na^2\vv u(t)\|_{L^2}^2$ at time $2T_2=2(\e\delta^{1-\f34\al})^{-2}$ are sufficiently small. So that the classical theory guarantees that the regularities of the solutions can be propagated afterwards.

\begin{proof}[The proof of Theorem \ref{propagation regularity thm}]
We divide the proof into the following steps:

\no{\bf Step 1. The estimate of $\|\na a(t)\|_{L^2}$ and $\|\na^2\vv u(t)\|_{L^2}$ on $[0,4T_2]$.}

 We shall use \eqref{estimate of na a} for $p=2$, i.e.
\beq\label{H 31}
\f{d}{dt}\|\na a(t)\|_{L^2}^2+\f\gamma2\|\na a\|_{L^2}^2\leq C\bigl(1+\|a\|_{L^\infty}^2\bigr)\|\r\dot{\vv u}\|_{L^2}^2+f(t)\|\na a\|_{L^2}^2,
\eeq
where $f(t)$ is defined in \eqref{H 2}.

 \no{\Large $\bullet$} {\bf Rough estimates of $\|\na a(t)\|_{L^2}$ and $\|\na^2\vv u(t)\|_{L^2}$ on $[0,4T_2]$.}

We first get, by applying Gronwall's inequality to \eqref{H 31}, that
\beq\label{H 31a}
\|\na a(t)\|_{L^2}^2\leq \Bigl(\|\na a_0\|_{L^2}^2+C\bigl(1+\|a\|_{L^\infty_t(L^\infty)}^2\bigr)\int_0^t\|\r\dot{\vv u}\|_{L^2}^2d\tau\Bigr)\exp\Bigl(C\int_0^tf(\tau)d\tau\Bigr).
\eeq

Since $\al\leq\f{1}{24\bar{c}_*+\f{7}{4}}\leq\f47$, the first inequality of \eqref{rough estimate for L 6 na a} could be improved to
\beno
\ln(\|\na a(t)\|_{L^6}^2+1)\leq C\delta^{-(24\bar{c}_*+\f12)\al},
\quad\forall\ t\leq 4T_2,
\eeno
which along with \eqref{H 5a} and \eqref{H 5b} implies
\beno\begin{aligned}
\int_0^tf(\tau)d\tau&\lesssim\e^2\bigl(\delta^{2-\f{3\al}{2}(1+\f{1}{\gamma})}+\delta^{1-2\al}\bigr)+\ln\delta^{-1}+1+\ln\delta^{-1}\cdot\delta^{-(24\bar{c}_*+\f12)\al}\\
&\lesssim\delta^{-(24\bar{c}_*+1)\al},\quad\forall\ t\leq 4T_2.
\end{aligned}\eeno
Then by virtue of \eqref{initial ansatz for na a in L 6}, \eqref{estimate for a} and \eqref{total enery estimate}, we deduce from\eqref{H 31a} that
\beq\label{H 32}\begin{aligned}
\|\na a\|_{L^2}^2&\lesssim \bigl(\delta^{1-2\al}+\delta^{-(2+\f{1}{\gamma})\al}\cdot\e^2\delta^{3-\al}\bigr)\exp\Bigl(C\delta^{-(24\bar{c}_*+1)\al}\Bigr)\\
&\lesssim\exp\Bigl(C\delta^{-(24\bar{c}_*+1)\al}\Bigr),\quad\forall\ t\leq 4T_2.
\end{aligned}\eeq

While it follows from  \eqref{energy in H 2} and \eqref{H 32} that
\beq\label{H 41}\begin{aligned}
\|\na^2\vv u(t)\|_{L^2}^2&\lesssim\|\na\curl\vv u(t)\|_{L^2}^2+\|\na F(t)\|_{L^2}^2+
+\|\na a(t)\|_{L^2}^2\\
&\lesssim\e^2\delta^{1-2(1+\f{1}{2\gamma})\al}+\exp\Bigl(C\delta^{-(24\bar{c}_*+1)\al}\Bigr)\\
&\lesssim\exp\Bigl(C\delta^{-(24\bar{c}_*+1)\al}\Bigr),\quad\forall\,t\leq 4T_2.
\end{aligned}\eeq

By combining \eqref{H 32} with \eqref{H 41}, we obtain the first inequality of \eqref{high order energy estimate}.

 \no{\Large $\bullet$} {\bf Improved estimates of $\|\na a(t)\|_{L^2}$ on $[T_1,4T_2]$.}

  Recall that $T_1=2\gamma^{-1}\ln(\e\delta^{1-\f34\al})^{-1}>T_0$.
  By virtue of  \eqref{H 2a} and \eqref{decay for a a}, we deduce from \eqref{H 31} that
\beno\begin{aligned}
\f{d}{dt}\|\na a(t)\|_{L^2}^2+\f\gamma2\|\na a\|_{L^2}^2
&\lesssim\|\sqrt{\r}\dot{\vv u}\|_{L^2}^2+\e\delta^{1-(24\bar{c}_*+\f74)\al}\|\na a\|_{L^2}^2\\
&\qquad+\bigl(\|\sqrt\r\dot{\vv u}\|_{L^2}\|\na\dot{\vv u}\|_{L^2}+\|a\|_{L^6}^2\bigr)\|\na a\|_{L^2}^2,
\end{aligned}\eeno
which along with the assumption that $\al\leq\f{1}{24\bar{c}_*+\f74}$ implies that
\beno\begin{aligned}
\f{d}{dt}\|\na a(t)\|_{L^2}^2+\f\gamma4\|\na a\|_{L^2}^2
&\lesssim\|\sqrt{\r}\dot{\vv u}\|_{L^2}^2+\bigl(\|\sqrt\r\dot{\vv u}\|_{L^2}\|\na\dot{\vv u}\|_{L^2}+\|a\|_{L^6}^2\bigr)\|\na a\|_{L^2}^2.
\end{aligned}\eeno
Then we get, by a similar derivation as \eqref{H 8}, that
\beno
\|\na a(t)\|_{L^2}^2\leq C\exp\Bigl(-\f\gamma4(t-T_1)\Bigr)\|\na a(T_1)\|_{L^2}^2+C\int_{T_1}^t\|\sqrt{\r}\dot{\vv u}\|_{L^2}^2d\tau,\quad\forall\ \, t\in[T_1,4T_2],
\eeno
which along with \eqref{H 32} and \eqref{total enery estimate} implies
\beq\label{H 41a}
\|\na a(t)\|_{L^2}^2\leq C\exp\Bigl(-\f\gamma4(t-T_1)+C\delta^{-(24\bar{c}_*+1)\al}\Bigr)+C\e^2\delta^{3-\al},\quad\forall \, t\in[T_1,4T_2].
\eeq

 \no{\Large $\bullet$}
{\bf Improved estimates of $\|\na a(t)\|_{L^2}$ and $\|\na^2\vv u(t)\|_{L^2}$ on $[T_2,4T_2]$.}

Similar to the derivation as \eqref{H 8a}, we deduce from \eqref{H 41a} that
\beq\label{H 33}
\|\na a(t)\|_{L^2}^2\leq C\e^2\delta^{2-\f32\al}+C\e^2\delta^{3-\al}\leq C\e^2\delta^{2-\f32\al},\quad\forall\, t\in[T_2,4T_2].
\eeq

By virtue of \eqref{elliptic estimates}, \eqref{decay for a a}, \eqref{decay of energy 2a} and \eqref{H 33}, we have
\beq\label{H 33a}\begin{aligned}
\|\na^2\vv u(t)\|_{L^2}^2&\lesssim\|\na\curl\vv u(t)\|_{L^2}^2+\|\na F(t)\|_{L^2}^2
+\|\na a(t)\|_{L^2}^2\lesssim\|\dot{\vv u}(t)\|_{L^2}^2+\|\na a(t)\|_{L^2}^2\\
&\lesssim\e^3\delta^{5(1-\al)}+\e^2\delta^{2-\f32\al}\lesssim\e^2\delta^{2-\f32\al},\quad\forall\,t\in[T_2,4T_2].
\end{aligned}\eeq

\no
{\bf Step 2. The estimate of $(\rho-1,\vv u)$ in the critical spaces at time $2T_2$.}

Since $F=\div\vv u-a$ and $|\r-1|\leq|\r^\gamma-1|=|a|$, we deduce from \eqref{decay of energy 2a} and $\varrho=\r-1$ that
\beq\label{H 34a}
\sup_{t\geq T_2}\bigl(\|\vv u(t)\|_{L^2}^2+\|\na\vv u(t)\|_{L^2}^2+\|\varrho(t)\|_{L^2}^2+\|\varrho(t)\|_{L^6}^2\bigr)\lesssim\e^3\delta^{5(1-\al)},
\eeq
which imples
\beq\label{H 34}
\|\vv u(t)\|_{\dB^{\f12}_{2,1}}\lesssim\|\vv u(t)\|_{L^2}^{\f12}\|\na\vv u(t)\|_{L^2}^{\f12}\lesssim\e^{\f32}\delta^{\f52(1-\al)},\quad\forall\ t\geq T_2.
\eeq

While observing that $\na a=\gamma\r^{\gamma-1}\na\r$ and $\r\geq(\f12)^{\f{1}{\gamma}}$ (see \eqref{lower bound of rho}), we deduce from \eqref{H 33} and \eqref{value of na a between T_2 and 2T_2} that for any $t\in[2T_2,4T_2]$,
\beq\label{H 35}
\|\na\r(t)\|_{L^2}\lesssim\|\na a(t)\|_{L^2}\lesssim\e\delta^{1-\f34\al}\andf
\|\na\r(t)\|_{L^6}\lesssim\|\na a(t)\|_{L^6}\lesssim\e^{\f32}\delta^{\f32-\f74\al},
\eeq
from which and \eqref{H 34a}, we infer
\beq\label{H 36}\begin{aligned}
&\|\varrho(t)\|_{\dB^{\f12}_{2,1}}\lesssim\|\varrho(t)\|_{L^2}^{\f12}\|\na\r(t)\|_{L^2}^{\f12}\lesssim\e^{\f54}\delta^{\f74-\f{13}{8}\al},\\
&\|\varrho(t)\|_{\dB^{\f34}_{2,1}}\lesssim\|\varrho(t)\|_{L^2}^{\f14}\|\na\r(t)\|_{L^2}^{\f34}\lesssim\e^{\f{9}{8}}\delta^{\f{11}{8}-\f{19}{16}\al},\\
&\|\varrho(t)\|_{\dB^{\f34}_{6,1}}\lesssim\|\varrho(t)\|_{L^6}^{\f14}\|\na\r(t)\|_{L^6}^{\f34}\lesssim\e^{\f32}\delta^{\f74-\f{31}{16}\al}.
\end{aligned}\eeq
By using interpolation inequality  again, we deduce from  the last two inequalities of \eqref{H 36} that
\beq\label{H 37}
\|\varrho(t)\|_{\dB^{\f34}_{4,1}}\lesssim\|\varrho(t)\|_{\dB^{\f34}_{2,1}}^{\f14}\|\varrho(t)\|_{\dB^{\f34}_{6,1}}^{\f34}\lesssim\e^{\f{45}{32}}\delta^{\f{53}{32}-\f74\al}.
\eeq

 Note that $\al\leq\min\bigl(\f{1+\f{1}{3\gamma}}{2+\f{1}{4\gamma}},\ \f{1}{24\bar{c}_*+\f{7}{4}}\bigr)<\f47$, we deduce from \eqref{H 34}, \eqref{H 36} and \eqref{H 37} that
\beq\label{H 38}
\|\vv u(t)\|_{\dB^{\f12}_{2,1}}+\|\varrho(t)\|_{\dB^{\f12}_{2,1}}+\|\varrho(t)\|_{\dB^{\f34}_{4,1}}\lesssim\e^{\f54}\delta^{\f{21}{32}},\quad\forall\ t\in[2T_2,4T_2].
\eeq
In particular, \eqref{H 38} holds for $t=2T_2,$ so that the critical norms of $(\varrho(t),\vv u(t))$ are small enough at time $2T_2$.

Thanks to \eqref{H 38} and Theorem 2.1 in \cite{Dan-He},  \eqref{CNS} has a unique global solution $(\r,\vv u)$ over $[2T_2,\infty)$ such that
there holds
 \eqref{energy in critical space}.

\no{\bf Step 3. Propagation of $\|\na a(t)\|_{L^6}$.}

We first observe that Proposition \ref{propagation prop for na a 1} ensures the estimate of $\|\na a(t)\|_{L^6}$ over time interval $[0,4T_2]$ which implies the first two inequalities of \eqref{dynamics of na a in L 6}. To deal with the estimate of $\|\na a(t)\|_{L^6}$ on $[4T_2,\infty)$, by using \eqref{elliptic estimates} and \eqref{dynamic 6}, we deduce from \eqref{E 73} for $p=6$ that
\beno\begin{aligned}
\f{d}{dt}\|\na a(t)\|_{L^6}^2+\f\gamma2\|\na a\|_{L^6}^2
&\leq4\gamma\|\na\vv u\|_{L^\infty}\|\na a\|_{L^6}^2+C\|\na\dot{\vv u}\|_{L^2}^2.
\end{aligned}\eeno
Here the second term on the r.h.s of \eqref{E 73} was absorbed by the l.h.s of \eqref{E 73}.

Applying the Gronwall's inequality to the above inequality gives rise to
\beno\begin{aligned}
\|\na a(t)\|_{L^6}^2\leq&\exp\Bigl(-\f{\gamma}{2}(t-2T_2)+4\gamma\int_{2T_2}^\infty\|\na\vv u(t)\|_{L^\infty}\,dt\Bigr)\cdot\|\na a(2T_2)\|_{L^6}^2\\
&+C\exp\Bigl(4\gamma\int_{2T_2}^\infty\|\na\vv u\|_{L^\infty}dt\Bigr)\cdot\int_{2T_2}^t\|\na\dot{\vv u}(\tau)\|_{L^2}^2\,d\tau,\quad\forall\, t\geq 2T_2.
\end{aligned}\eeno
 Observing that $\int_{2T_2}^\infty\|\na\vv u\|_{L^\infty}\,dt\lesssim1$ (see \eqref{energy in critical space}), we
 get, by using \eqref{decay of energy 2a} and \eqref{value of na a between T_2 and 2T_2},  that for $t\geq 3T_2=3(\e\delta^{1-\f34\al})^{-2}$
\beno\begin{aligned}
\|\na a(t)\|_{L^6}^2&\lesssim\exp\Bigl(-\f{\gamma}{2}(t-2T_2))\Bigr)\|\na a(2T_2)\|_{L^6}^2+\int_{2T_2}^t\|\na\dot{\vv u}(\tau)\|_{L^2}^2\,d\tau\\
&\lesssim\exp\Bigl(-\f{\gamma}{2}(\e\delta^{1-\f34\al})^{-2}\Bigr)\e^3\delta^{3-\f72\al}+\e^3\delta^{5(1-\al)}\lesssim\e^3\delta^{5(1-\al)},
\end{aligned}\eeno
which leads to the last inequality of \eqref{dynamics of na a in L 6}.

\no{\bf Step 4. The estimates of $\|\na a(t)\|_{L^2}$ and $\|\na^2\vv u(t)\|_{L^2}$ on $[2T_2,\infty)$.}

After similar derivation as Step 3,
 we
get, by using \eqref{H 33} and \eqref{H 33a} for $t=2T_2$, that
\beno
\sup_{t\geq 2T_2}\Bigl(\|\na a(t)\|_{L^2}^2+\|\na^2\vv u(t)\|_{L^2}^2\Bigr)
\lesssim\e^2\delta^{2-\f32\al},
\eeno
which is the second inequality of \eqref{high order energy estimate}.
We thus complete the proof of Theorem \ref{propagation regularity thm}.
\end{proof}

\vspace{0.5cm}
\noindent {\bf Acknowledgments.}  The first author is supported by NSF of China
under  Grants 11771236 and 12141102. The second author is  supported by  NSF of China  under Grants  11671383 and 12171019.
The third author is supported by National Key R$\&$D Program of China under grant
  2021YFA1000800 and K. C. Wong Education Foundation.
  Ping Zhang is  also partially supported by  NSF of China under Grants  12288201, 11731007 and 12031006.  All data generated or analysed during this study are included in this article.

\end{document}